\documentclass[reqno]{amsart}

\usepackage{amsmath,amssymb,amsthm,amsfonts,mathrsfs}
\usepackage{fullpage}
\usepackage{xcolor}
\usepackage{graphicx}
\usepackage{listings}
\usepackage{bbm}

\definecolor{codegreen}{rgb}{0,0.6,0}
\definecolor{codegray}{rgb}{0.5,0.5,0.5}
\definecolor{codepurple}{rgb}{0.58,0,0.82}
\definecolor{backcolour}{rgb}{0.95,0.95,0.92}

\lstdefinestyle{mystyle}{
  backgroundcolor=\color{backcolour},   commentstyle=\color{codegreen},
  keywordstyle=\color{magenta},
  numberstyle=\tiny\color{codegray},
  stringstyle=\color{codepurple},
  basicstyle=\ttfamily\footnotesize,
  breakatwhitespace=false,         
  breaklines=true,                 
  captionpos=b,                    
  keepspaces=true,                 
  numbers=left,                    
  numbersep=5pt,                  
  showspaces=false,                
  showstringspaces=false,
  showtabs=false,                  
  tabsize=2
}

\usepackage{listings}
\usepackage[colorlinks=true,
linkcolor=blue,
anchorcolor=blue,
citecolor=red
]{hyperref}
\allowdisplaybreaks 
\newtheorem{theorem1number}{Theorem}
\newtheorem{theorem}{Theorem}[section]
\newtheorem{lemma}[theorem]{Lemma}

\newtheorem*{conjecture*}{Conjecture}
\theoremstyle{definition}

\theoremstyle{remark}

\newtheorem*{remark*}{remark}

\usepackage{colonequals}

\author{\textbf{Runbo Li}}
\address{International Curriculum Center, The High School Affiliated to Renmin University of China, Beijing, China}
\email{runbo.li.carey@gmail.com}

\makeatletter
\@namedef{subjclassname@2020}{\textup{}2020 Mathematics Subject Classification}
\makeatother

\title[]{On the largest prime factor of integers in short intervals III}

\begin{document}
	
\begin{abstract}
Using Watt's mean value theorem and a delicate sieve decomposition, the author shows that the interval $[x, x+x^{\frac{1}{2}+\varepsilon}]$ contains an integer with a prime factor larger than $x^{\frac{35}{36}-\varepsilon}$ for sufficiently large $x$. This gives a solution with $\gamma = \frac{1}{36}$ to the Exercise 5.1 in Harman's monograph and improves the previous record of the author proved in 2024, where $\gamma = \frac{1}{26.5}$ is obtained.
\end{abstract}

\maketitle
\hfill \textit{If you are feeling brave, try to bring in Watt's mean value theorem.}
\hfill \textit{---Glyn Harman}

\tableofcontents

\section{Introduction}
The Legendre's conjecture, which states that there is always a prime number between consecutive squares, is one of Landau's problems on prime numbers. Clearly this means that there is always a prime number in the interval $[x, x+x^{\frac{1}{2}}]$. However, we cannot prove it even on the Riemann Hypothesis. Assuming RH, one can only show that there is always a prime number in the interval $[x, x+x^{\frac{1}{2}}\log x]$. The best unconditional result is due to the author \cite{LRB052}, where he showed the interval $[x, x+x^{0.52}]$ contains primes. Under a special condition concerning the existence of Siegel zeros, he \cite{LRB4923} also showed that the interval $[x, x+x^{0.4923}]$ contains primes for long ranges of $x$.

Instead of relaxing the length of the short interval, one can attack this conjecture by relaxing our restriction of primes. A number with a large prime factor is a good approximation of prime numbers. Thus, we can try to find numbers with a large prime factor in the interval $[x, x+x^{\frac{1}{2}}]$ or $[x, x+x^{\frac{1}{2} + \varepsilon}]$. In \cite{LRB5153} and \cite{LRB7437} we considered the former and latter intervals respectively, and here we shall return to the former one again.

In this direction, Jutila \cite{Jutila23} showed in 1973 that this interval contains a number with a prime factor larger than $x^{\frac{2}{3}-\varepsilon}$. The exponent $\frac{2}{3}$ has been improved to
$$
0.73,\ 0.7338,\ 0.772,\ 0.82,\ \frac{11}{12},\ \frac{17}{18},\ \frac{19}{20},\ \frac{24}{25},\ \frac{25}{26} \text{ and } \frac{25.5}{26.5}
$$
by Balog \cite{Balog73} \cite{Balog772}, Balog, Harman and Pintz \cite{BalogHarmanPintz2}, Heath-Brown \cite{HeathBrown1112}, Heath-Brown and Jia \cite{HeathBrownJia1718}, Harman [\cite{HarmanBOOK}, Chapter 5], Haugland \cite{Haugland}, Jia and Liu \cite{JiaLiu} and the author \cite{LRB5153} respectively. In his monograph, Harman [\cite{HarmanBOOK}, Exercise 5.1] encouraged us to reduce this exponent as much as we can. In this paper, by applying Watt's mean value theorem \cite{WattTheorem} together with a delicate sieve decomposition and the arguments in \cite{LRB5153}, we give a large improvement over previous result by obtaining the following result.
\begin{theorem1number}\label{t1}
For sufficiently large $x$, the interval $[x, x+x^{\frac{1}{2}+\varepsilon}]$ contains an integer with a prime factor larger than $x^{\frac{35}{36}-\varepsilon}$.
\end{theorem1number}

Throughout this paper, we always suppose that $\varepsilon$ is a sufficiently small positive constant and $B=B(\varepsilon)$ is a sufficiently large positive constant. We choose $\varepsilon$ such that $K = \frac{8}{\varepsilon}(\frac{1}{36}+\frac{\varepsilon}{2})$ is an integer. The letters $p$ and $\beta$, with or without subscript, are reserved for primes and almost-primes respectively. Let $c_1$ and $c_2$ denote positive constants which may have different values at different places, and we write $m \sim M$ to mean that $c_1 M<m \leqslant c_2 M$. Let $b = 1+\frac{1}{\log x}$, $v=x^{\frac{35}{36}-\frac{\varepsilon}{2}}$ and $P = x^{\frac{\varepsilon}{8}}$. We use $M(s)$, $N(s)$ and some other capital letters (except $L$ and $P$) to denote $1$-bounded Dirichlet polynomials
$$
M(s)=\sum_{m \sim M} a_m m^{-s}, \quad N(s)=\sum_{n \sim N} b_n n^{-s}.
$$
We use $L(s)$ to denote a ``zeta-factor''
$$
L(s)=\sum_{l \sim L} l^{-s}.
$$
We also use $P(s)$ to denote
$$
P(s)=\sum_{P < p \leqslant 2P} p^{-s}.
$$

\section{Arithmetic Information}
In this section we provide some arithmetic information (i.e. mean value bounds for some Dirichlet polynomials) which will help us prove the asymptotic formulas for the sieve functions in next section. The first two lemmas come from the work of Heath-Brown and Jia \cite{HeathBrownJia1718}, and they are also the main arithmetic information inputs used in [\cite{HarmanBOOK}, Chapter 5] and \cite{Haugland}.

\begin{lemma}\label{l21} ([\cite{HeathBrownJia1718}, Lemma 1]).
Suppose that $M N = v$ and $v^{\frac{17}{35}} \ll M \ll v^{\frac{18}{35}}$. Then for $(\log x)^{2B} \leqslant T \leqslant x^{\frac{1}{2}-\frac{\varepsilon}{6}}$, we have
$$
\int_{T}^{2 T}|M(b+i t) N(b+i t) P^K (b+i t)| d t \ll (\log x)^{-B}.
$$
\end{lemma}
\begin{lemma}\label{l22} ([\cite{HeathBrownJia1718}, Lemma 2]).
Suppose that $M N L=v$, $M \ll v^{\frac{18}{35}}$ and $N \ll v^{\frac{9}{35}}$. Then for $\sqrt{L} \leqslant T \leqslant x^{\frac{1}{2}-\frac{\varepsilon}{6}}$, we have
$$
\int_{T}^{2 T}|M(b+i t) N(b+i t) L(b+i t) P^K (b+i t)| d t \ll (\log x)^{-B}.
$$
\end{lemma}

In Exercise 5.1 in his monograph [\cite{HarmanBOOK}, Chapter 5], Harman encouraged the readers to ``reduce the value of $\gamma$ (the exponent) as much as you can.''  He also said that ``if you are feeling brave, try to bring in Watt's mean value theorem \cite{WattTheorem}.'' Here we shall use another form of Watt's theorem given in \cite{LRB122}.
\begin{lemma}\label{Watt}([\cite{LRB122}, Lemma 4.1]).
Let $T \geqslant 1$, then
$$ 
\int_{T}^{2 T}\left|L\left(\frac{1}{2}+i t\right)\right|^{4}\left|N\left(\frac{1}{2}+i t\right)\right|^{2} d t \ll \left(T+N^{2} T^{\frac{1}{2}}+ L^2 T^{-4} (N+T) \right) T^{\frac{\varepsilon}{2}}.
$$
\end{lemma}

We shall use Lemma~\ref{Watt} to derive the following important lemma, which gives more Type-I information than in \cite{JiaLiu} and \cite{LRB5153}. This new arithmetic information, together with a more delicate sieving process similar to that in \cite{LRB052}, is the key of our improvement over \cite{LRB5153}.
\begin{lemma}\label{l23}
Suppose that $M N H L=v$, and $M$, $N$ and $H$ satisfy one of the following $3$ conditions:

(1). $M \ll v^{\frac{18}{35}},\ N \gg H,\ N^{\frac{3}{4}}H \ll v^{\frac{9}{35}},\ N H^{\frac{1}{2}} \ll v^{\frac{9}{35}},\ N^{\frac{7}{4}} H^{\frac{3}{2}} \ll v^{\frac{18}{35}}$;

(2). $M \ll v^{\frac{18}{35}},\ N \ll v^{\frac{9}{35}},\ H \ll v^{\frac{9}{140}}$;

(3). $M \ll v^{\frac{18}{35}},\ N^2 H \ll v^{\frac{18}{35}},\ H \ll v^{\frac{9}{70}}$.

Then for $\sqrt{L} \leqslant T \leqslant x^{\frac{1}{2}-\frac{\varepsilon}{6}}$, we have
$$
\int_{T}^{2 T}|M(b+i t) N(b+i t) H(b+i t) L(b+i t) P^K (b+i t)| d t \ll (\log x)^{-B}.
$$
\end{lemma}
\begin{proof}
We only need to prove that
$$
\int_{T}^{2 T}\left|M\left(\frac{1}{2}+i t\right) N\left(\frac{1}{2}+i t\right) H\left(\frac{1}{2}+i t\right) L\left(\frac{1}{2}+i t\right) P^K\left(\frac{1}{2}+i t\right)\right| d t \ll x^{\frac{1}{2}} (\log x)^{-B}
$$
holds true under each condition above.

Assume that Condition (1) holds. The proof of this case is similar to that of [\cite{JiaLiu}, Lemma 3] where [\cite{DI84}, Theorem 2] is used.

Assume that Condition (2) holds. Using H\"{o}lder's inequality and the mean value theorem for Dirichlet polynomials, we have
\begin{align}
\nonumber & \int_{T}^{2 T}\left|M\left(\frac{1}{2}+i t\right) N\left(\frac{1}{2}+i t\right) H\left(\frac{1}{2}+i t\right) L\left(\frac{1}{2}+i t\right) P^K\left(\frac{1}{2}+i t\right)\right| d t \\
\nonumber \ll&\ \left( \int_{T}^{2 T}\left|M\left(\frac{1}{2}+i t\right) \right|^2 \left| P^K\left(\frac{1}{2}+i t\right)\right|^2 d t \right)^{\frac{1}{2}} \left( \int_{T}^{2 T}\left|N\left(\frac{1}{2}+i t\right) \right|^4 d t \right)^{\frac{1}{4}} \\
\nonumber &\ \left( \int_{T}^{2 T}\left|L\left(\frac{1}{2}+i t\right) \right|^4 \left|H\left(\frac{1}{2}+i t\right) \right|^4 d t \right)^{\frac{1}{4}} \\
\nonumber \ll&\ x^{\frac{1}{4}} \left( \int_{T}^{2 T}\left|N\left(\frac{1}{2}+i t\right) \right|^4 d t \right)^{\frac{1}{4}} \left( \int_{T}^{2 T}\left|L\left(\frac{1}{2}+i t\right) \right|^4 \left|H\left(\frac{1}{2}+i t\right) \right|^4 d t \right)^{\frac{1}{4}}.
\end{align}
By the mean value theorem for Dirichlet polynomials, we have
$$
\left( \int_{T}^{2 T}\left|N\left(\frac{1}{2}+i t\right) \right|^4 d t \right)^{\frac{1}{4}} \ll \left(N^2 + T\right)^{\frac{1}{4}}.
$$
Note that $\sqrt{L} \leqslant T$, we have
$$
L^2 T^{-4} (N + T) \ll N + T.
$$
Now, using Lemma~\ref{Watt} with $N = H^2$, we get
$$
\left( \int_{T}^{2 T}\left|L\left(\frac{1}{2}+i t\right) \right|^4 \left|H\left(\frac{1}{2}+i t\right) \right|^4 d t \right)^{\frac{1}{4}} \ll \left(T + H^4 T^{\frac{1}{2}} + H^2 \right)^{\frac{1}{4}} T^{\frac{\varepsilon}{8}}.
$$
Combining the bounds above with Condition (2), we have
\begin{align}
\nonumber & \int_{T}^{2 T}\left|M\left(\frac{1}{2}+i t\right) N\left(\frac{1}{2}+i t\right) H\left(\frac{1}{2}+i t\right) L\left(\frac{1}{2}+i t\right) P^K\left(\frac{1}{2}+i t\right)\right| d t \\
\nonumber \ll&\ x^{\frac{1}{4}} \left(N^2 + T\right)^{\frac{1}{4}} \left(T + H^4 T^{\frac{1}{2}} + H^2 \right)^{\frac{1}{4}} T^{\frac{\varepsilon}{8}} \\
\nonumber \ll&\ x^{\frac{1}{2}} (\log x)^{-B}.
\end{align}
Note that this case is an improvement over Lemma~\ref{l22}.

Assume that Condition (3) holds. Using H\"{o}lder's inequality and the mean value theorem for Dirichlet polynomials, we have
\begin{align}
\nonumber & \int_{T}^{2 T}\left|M\left(\frac{1}{2}+i t\right) N\left(\frac{1}{2}+i t\right) H\left(\frac{1}{2}+i t\right) L\left(\frac{1}{2}+i t\right) P^K\left(\frac{1}{2}+i t\right)\right| d t \\
\nonumber \ll&\ \left( \int_{T}^{2 T}\left|M\left(\frac{1}{2}+i t\right) \right|^2 \left| P^K\left(\frac{1}{2}+i t\right)\right|^2 d t \right)^{\frac{1}{2}} \left( \int_{T}^{2 T}\left|N\left(\frac{1}{2}+i t\right) \right|^4 \left|H\left(\frac{1}{2}+i t\right) \right|^2 d t \right)^{\frac{1}{4}} \\
\nonumber &\ \left( \int_{T}^{2 T}\left|L\left(\frac{1}{2}+i t\right) \right|^4 \left|H\left(\frac{1}{2}+i t\right) \right|^2 d t \right)^{\frac{1}{4}} \\
\nonumber \ll&\ x^{\frac{1}{4}} \left( \int_{T}^{2 T}\left|N\left(\frac{1}{2}+i t\right) \right|^4 \left|H\left(\frac{1}{2}+i t\right) \right|^2 d t \right)^{\frac{1}{4}} \left( \int_{T}^{2 T}\left|L\left(\frac{1}{2}+i t\right) \right|^4 \left|H\left(\frac{1}{2}+i t\right) \right|^2 d t \right)^{\frac{1}{4}}.
\end{align}
By Condition (3), Lemma~\ref{Watt} with $N = H$ and the mean value theorem for Dirichlet polynomials, we have
\begin{align}
\nonumber & \int_{T}^{2 T}\left|M\left(\frac{1}{2}+i t\right) N\left(\frac{1}{2}+i t\right) H\left(\frac{1}{2}+i t\right) L\left(\frac{1}{2}+i t\right) P^K\left(\frac{1}{2}+i t\right)\right| d t \\
\nonumber \ll&\ x^{\frac{1}{4}} \left(N^2 H + T\right)^{\frac{1}{4}} \left(T + H^2 T^{\frac{1}{2}} + H \right)^{\frac{1}{4}} T^{\frac{\varepsilon}{8}} \\
\nonumber \ll&\ x^{\frac{1}{2}} (\log x)^{-B}.
\end{align}

Now the proof of Lemma~\ref{l23} is completed.
\end{proof}

\section{The final decomposition}
Let $\mathcal{C}$ denote a finite set of positive integers, $p_{j} \sim v^{t_j}$ and put
$$
N(d)=\sum_{\substack{x<p p_1 \ldots p_K \leqslant x+x^{\frac{1}{2} + \varepsilon} \\ P < p_i \leqslant 2P}}1, \quad \mathcal{A}=\{n: 2^{-K} v<n \leqslant 2 v,\ n \text{ repeats } N(n) \text{ times}\},
$$
$$
\mathcal{B}=\{n: v<n \leqslant 2 v\}, \quad \mathcal{C}_d=\{a: a d \in \mathcal{C} \}, \quad P(z)=\prod_{p<z} p, \quad S(\mathcal{C}, z)=\sum_{\substack{a \in \mathcal{C} \\ (a, P(z))=1}} 1.
$$
Then we only need to show that $S\left(\mathcal{A},(2 v)^{\frac{1}{2}}\right) >0$. 

\textit{Buchstab's identity} is the equation
$$
S\left(\mathcal{C}, z\right) = S\left(\mathcal{C}, w\right) - \sum_{w \leqslant p < z} S\left(\mathcal{C}_{p}, p\right),
$$
where $2 \leqslant w < z$.

Our aim is to show that the sparser set $\mathcal{A}$ contains the expected proportion of primes compared to the bigger set $\mathcal{B}$, which requires us to decompose $S\left(\mathcal{A}, (2 v)^{\frac{1}{2}}\right)$ and prove asymptotic formulas of the form
\begin{equation}
S\left(\mathcal{A}, z\right) = v^{-1} x^{\frac{1}{2}+\varepsilon} \left(\sum_{P < p \leqslant 2P}\frac{1}{p}\right)^{K} (1+o(1)) S\left(\mathcal{B}, z\right)
\end{equation}
for some parts of it, and drop the other positive parts.

Let $\omega(u)$ denote the Buchstab function determined by the following differential-difference equation
\begin{align*}
\begin{cases}
\omega(u)=\frac{1}{u}, & \quad 1 \leqslant u \leqslant 2, \\
(u \omega(u))^{\prime}= \omega(u-1), & \quad u \geqslant 2 .
\end{cases}
\end{align*}
Moreover, we have the upper and lower bounds for $\omega(u)$:
\begin{align*}
\omega(u) \geqslant \omega_{0}(u) =
\begin{cases}
\frac{1}{u}, & \quad 1 \leqslant u < 2, \\
\frac{1+\log(u-1)}{u}, & \quad 2 \leqslant u < 3, \\
\frac{1+\log(u-1)}{u} + \frac{1}{u} \int_{2}^{u-1}\frac{\log(t-1)}{t} d t \geqslant 0.5607, & \quad 3 \leqslant u < 4, \\
0.5612, & \quad u \geqslant 4, \\
\end{cases}
\end{align*}
\begin{align*}
\omega(u) \leqslant \omega_{1}(u) =
\begin{cases}
\frac{1}{u}, & \quad 1 \leqslant u < 2, \\
\frac{1+\log(u-1)}{u}, & \quad 2 \leqslant u < 3, \\
\frac{1+\log(u-1)}{u} + \frac{1}{u} \int_{2}^{u-1}\frac{\log(t-1)}{t} d t \leqslant 0.5644, & \quad 3 \leqslant u < 4, \\
0.5617, & \quad u \geqslant 4. \\
\end{cases}
\end{align*}
We shall use $\omega_0(u)$ and $\omega_1(u)$ to give numerical bounds for some sieve functions discussed below. We shall also use the simple upper bound $\omega(u) \leqslant \max(\frac{1}{u}, 0.5672)$ (see Lemma 8(iii) of \cite{JiaPSV}) to estimate high-dimensional integrals.

By Prime Number Theorem with Vinogradov's error term and the inductive arguments in [\cite{HarmanBOOK}, Chapter A.2], we know that, for sufficiently large $z$,
\begin{equation}
S\left(\mathcal{B}, z\right) = \sum_{\substack{a \in \mathcal{B} \\ (a, P(z))=1}} 1 = (1+o(1)) \frac{v}{\log z} \omega\left(\frac{\log v}{\log z}\right),
\end{equation}
and we expect that the similar relation also holds for $S\left(\mathcal{A}, z\right)$:
\begin{equation}
S\left(\mathcal{A}, z\right) = \sum_{\substack{a \in \mathcal{A} \\ (a, P(z))=1}} 1 = (1+o(1)) \left(\sum_{P < p \leqslant 2P}\frac{1}{p}\right)^{K} \frac{x^{\frac{1}{2}+\varepsilon}}{\log z} \omega\left(\frac{\log v}{\log z}\right).
\end{equation}
If (1) holds for $S\left(\mathcal{A}, z\right)$, then we can deduce (3) easily from (1) and (2). Otherwise we must drop this $S\left(\mathcal{A}, z\right)$. We define the \textit{loss} from this term by the size of corresponding $S\left(\mathcal{B}, z\right)$:
\begin{equation}
S\left(\mathcal{B}, z\right) = (\textit{loss}+o(1)) \frac{v}{\log x}.
\end{equation}
We note that for the lower bound problem, we can only drop positive parts and the total loss of the dropped parts must be less than $1$.

Before decomposing, we define the asymptotic regions $\mathcal{T}_0$--$\mathcal{T}_2$, $\mathcal{I}_{n}$ $\mathcal{J}_{n}$ and a two-dimensional region $\mathcal{L}$ as
\begin{align}
\nonumber \mathcal{T}_0 (m_1, m_2) :=&\ \left\{ \frac{17}{35} \leqslant m_1 \leqslant \frac{18}{35}\ \text{ or }\ \frac{17}{35} \leqslant m_2 \leqslant \frac{18}{35}\ \text{ or }\ \frac{17}{35} \leqslant m_1 + m_2 \leqslant \frac{18}{35} \right\}, \\
\nonumber \mathcal{T}_1 (m_1, m_2) :=&\ \left\{ m_1 \leqslant \frac{18}{35},\ m_2 \leqslant \frac{9}{35} \right\}, \\
\nonumber \mathcal{T}_{21} (m_1, m_2, m_3) :=&\ \left\{ m_1 \leqslant \frac{18}{35},\ m_2 \geqslant m_3,\ \frac{3}{4} m_2 + m_3 \leqslant \frac{9}{35},\ m_2 + \frac{1}{2} m_3 \leqslant \frac{9}{35},\ \frac{7}{4} m_2 + \frac{3}{2} m_3 \leqslant \frac{18}{35} \right\}, \\
\nonumber \mathcal{T}_{22} (m_1, m_2, m_3) :=&\ \left\{ m_1 \leqslant \frac{18}{35},\ m_2 \leqslant \frac{9}{35},\ m_3 \leqslant \frac{9}{140} \right\}, \\
\nonumber \mathcal{T}_{23} (m_1, m_2, m_3) :=&\ \left\{ m_1 \leqslant \frac{18}{35},\ 2 m_2 + m_3 \leqslant \frac{18}{35},\ m_3 \leqslant \frac{9}{70} \right\}, \\
\nonumber \mathcal{T}_2 (m_1, m_2, m_3) :=&\ \mathcal{T}_{21} \cup \mathcal{T}_{22} \cup \mathcal{T}_{23}, \\
\nonumber \mathcal{I}_{i} (m_1, m_2, \ldots, m_i) :=&\ \left\{ (m_1, m_2, \ldots, m_i) \text{ can be partitioned into } (\alpha_1, \alpha_2) \in \mathcal{T}_0 \right\}, \\
\nonumber \mathcal{J}_{i} (m_1, m_2, \ldots, m_i) :=&\ \left\{ (m_1, m_2, \ldots, m_i) \text{ can be partitioned into } (\alpha_1, \alpha_2) \in \mathcal{T}_1 \text{ or } (\alpha_1, \alpha_2, \alpha_3) \in \mathcal{T}_2 \right\}, \\
\nonumber \mathcal{L} (m, n) :=&\ \left\{ (m, n) \notin \mathcal{T}_0,\ (m, n, n) \notin \mathcal{J}_3,\ n \geqslant \frac{53}{255} \text{ or } m \geqslant \frac{1129}{2448} \text{ or } \frac{1}{2} m + n \geqslant \frac{9361}{24480} \right\}.
\end{align}
By the definition of $\mathcal{J}_{i}$, if $(m_1, m_2, \ldots, m_i) \in \mathcal{J}_{i}$, then $(m_1, m_2, \ldots, m_{i-1}, l) \in \mathcal{J}_{i}$ for all $l \leqslant m_i$. The cases of reducing other variables are similar.

Now we state the general sieve asymptotic formulas we need in our final decomposition.
\begin{lemma}\label{l31}
We can give an asymptotic formula of the form (1) for
$$
\sum_{t_1 \cdots t_n} S\left(\mathcal{A}_{p_1 \cdots p_n}, v^{\frac{1}{35}} \right)
$$
and similar sums involving role-reversals if $(t_1, \ldots, t_n) \in \mathcal{J}_{n}$.
\end{lemma}
\begin{proof}
This lemma can be proved by Lemma~\ref{l23} together with the discussions in [\cite{HeathBrownJia1718}, Lemmas 7,9].
\end{proof}
\begin{lemma}\label{l32}
We can give an asymptotic formula of the form (1) for
$$
\sum_{t_1 \cdots t_n} S\left(\mathcal{A}_{p_1 \cdots p_n}, p_n \right)
$$
and similar sums involving role-reversals if $(t_1, \ldots, t_n) \in \mathcal{I}_{n}$.
\end{lemma}
\begin{proof}
This lemma can be proved by Lemma~\ref{l21} together with the discussions in [\cite{HeathBrownJia1718}, Lemma 8].
\end{proof}

We start our final decomposition of $S\left(\mathcal{A}, (2 v)^{\frac{1}{2}}\right)$. By Buchstab's identity, we have
\begin{align}
\nonumber S\left(\mathcal{A}, (2 v)^{\frac{1}{2}}\right) =&\ S\left(\mathcal{A}, v^{\frac{1}{35}}\right) - \sum_{\frac{1}{35} \leqslant t_1 < \frac{17}{35}} S\left(\mathcal{A}_{p_1}, p_1\right) - \sum_{\frac{17}{35} \leqslant t_1 < \frac{1}{2}} S\left(\mathcal{A}_{p_1}, p_1\right) \\
\nonumber =&\ S\left(\mathcal{A}, v^{\frac{1}{35}}\right) - \sum_{\frac{1}{35} \leqslant t_1 < \frac{17}{35}} S\left(\mathcal{A}_{p_1}, v^{\frac{1}{35}}\right) - \sum_{\frac{17}{35} \leqslant t_1 < \frac{1}{2}} S\left(\mathcal{A}_{p_1}, p_1\right) \\
\nonumber &+ \sum_{\substack{\frac{1}{35} \leqslant t_1 < \frac{17}{35} \\ \frac{1}{35} \leqslant t_2 < \min \left(t_1, \frac{1}{2}(1 - t_1) \right) }} S\left(\mathcal{A}_{p_1 p_2}, p_2\right) \\
=&\ S_1 - S_2 - S_3 + S_4.
\end{align}
By Lemma~\ref{l21} and Lemma~\ref{l22}, we can give asymptotic formulas for $S_1$, $S_2$ and $S_3$.

Before estimating $S_{4}$, we first split it into three parts:
\begin{align}
\nonumber S_4 =&\ \sum_{\substack{\frac{1}{35} \leqslant t_1 < \frac{17}{35} \\ \frac{1}{35} \leqslant t_2 < \min \left(t_1, \frac{1}{2}(1 - t_1) \right) }} S\left(\mathcal{A}_{p_1 p_2}, p_2\right) \\
\nonumber =&\ \sum_{\substack{\frac{1}{35} \leqslant t_1 < \frac{17}{35} \\ \frac{1}{35} \leqslant t_2 < \min \left(t_1, \frac{1}{2}(1 - t_1) \right) \\ (t_1, t_2) \in \mathcal{T}_0 }} S\left(\mathcal{A}_{p_1 p_2}, p_2\right) + \sum_{\substack{\frac{1}{35} \leqslant t_1 < \frac{17}{35} \\ \frac{1}{35} \leqslant t_2 < \min \left(t_1, \frac{1}{2}(1 - t_1) \right)  \\ (t_1, t_2) \in \mathcal{L} }} S\left(\mathcal{A}_{p_1 p_2}, p_2\right) \\
\nonumber &+ \sum_{\substack{\frac{1}{35} \leqslant t_1 < \frac{17}{35} \\ \frac{1}{35} \leqslant t_2 < \min \left(t_1, \frac{1}{2}(1 - t_1) \right) \\ (t_1, t_2) \notin \mathcal{T}_0 \\ (t_1, t_2, t_2) \in \mathcal{J}_3 }} S\left(\mathcal{A}_{p_1 p_2}, p_2\right) + \sum_{\substack{\frac{1}{35} \leqslant t_1 < \frac{17}{35} \\ \frac{1}{35} \leqslant t_2 < \min \left(t_1, \frac{1}{2}(1 - t_1) \right) \\ (t_1, t_2) \notin \mathcal{T}_0 \cup \mathcal{L} \\ (t_1, t_2, t_2) \notin \mathcal{J}_3 }} S\left(\mathcal{A}_{p_1 p_2}, p_2\right) \\
\nonumber =&\ \sum_{\substack{\frac{1}{35} \leqslant t_1 < \frac{17}{35} \\ \frac{1}{35} \leqslant t_2 < \min \left(t_1, \frac{1}{2}(1 - t_1) \right) \\ (t_1, t_2) \in \mathcal{T}_0 }} S\left(\mathcal{A}_{p_1 p_2}, p_2\right) + \sum_{\substack{\frac{1}{35} \leqslant t_1 < \frac{17}{35} \\ \frac{1}{35} \leqslant t_2 < \min \left(t_1, \frac{1}{2}(1 - t_1) \right) \\ (t_1, t_2) \in \mathcal{L} }} S\left(\mathcal{A}_{p_1 p_2}, p_2\right) \\
\nonumber &+ \sum_{\substack{\frac{1}{35} \leqslant t_1 < \frac{17}{35} \\ \frac{1}{35} \leqslant t_2 < \min \left(t_1, \frac{1}{2}(1 - t_1) \right) \\ (t_1, t_2) \in \mathcal{M} }} S\left(\mathcal{A}_{p_1 p_2}, p_2\right) + \sum_{\substack{\frac{1}{35} \leqslant t_1 < \frac{17}{35} \\ \frac{1}{35} \leqslant t_2 < \min \left(t_1, \frac{1}{2}(1 - t_1) \right) \\ (t_1, t_2) \in \mathcal{N} }} S\left(\mathcal{A}_{p_1 p_2}, p_2\right) \\
=&\ S_{41} + S_{42} + S_{43} + S_{44},
\end{align}
where
\begin{align}
\nonumber \mathcal{M} (m, n) :=&\ \left\{ (m, n) \notin \mathcal{T}_0,\ (m, n, n) \in \mathcal{J}_3 \right\}, \\
\nonumber \mathcal{N} (m, n) :=&\ \left\{ (m, n) \notin \mathcal{T}_0 \cup \mathcal{L},\ (m, n, n) \notin \mathcal{J}_3 \right\}.
\end{align}

By Lemma~\ref{l32}, $S_{41}$ has an asymptotic formula. For $S_{42}$, we cannot decompose further but have to discard the whole region giving the loss
\begin{equation}
\int_{\frac{1}{35}}^{\frac{17}{35}} \int_{\frac{1}{35}}^{\min\left(t_1, \frac{1-t_1}{2}\right)} \mathbbm{1}_{(t_1, t_2) \in \mathcal{L}} \frac{\omega\left(\frac{1 - t_1 - t_2}{t_2}\right)}{t_1 t_2^2} d t_2 d t_1 < 0.7226.
\end{equation}

For $S_{43}$ we can use Buchstab’s identity to get
\begin{align}
\nonumber S_{43} =&\ \sum_{\substack{\frac{1}{35} \leqslant t_1 < \frac{17}{35} \\ \frac{1}{35} \leqslant t_2 < \min \left(t_1, \frac{1}{2}(1 - t_1) \right) \\ (t_1, t_2) \in \mathcal{M} }} S\left(\mathcal{A}_{p_1 p_2}, p_2\right) \\
\nonumber =&\ \sum_{\substack{\frac{1}{35} \leqslant t_1 < \frac{17}{35} \\ \frac{1}{35} \leqslant t_2 < \min \left(t_1, \frac{1}{2}(1 - t_1) \right) \\ (t_1, t_2) \in \mathcal{M} }} S\left(\mathcal{A}_{p_1 p_2}, v^{\frac{1}{35}}\right) - \sum_{\substack{\frac{1}{35} \leqslant t_1 < \frac{17}{35} \\ \frac{1}{35} \leqslant t_2 < \min \left(t_1, \frac{1}{2}(1 - t_1) \right) \\ (t_1, t_2) \in \mathcal{M} \\ \frac{1}{35} \leqslant t_3 < \min \left(t_2, \frac{1}{2}(1 - t_1 - t_2) \right) \\  (t_1, t_2, t_3) \in \mathcal{I}_3 }} S\left(\mathcal{A}_{p_1 p_2 p_3}, p_3 \right) \\
\nonumber &- \sum_{\substack{\frac{1}{35} \leqslant t_1 < \frac{17}{35} \\ \frac{1}{35} \leqslant t_2 < \min \left(t_1, \frac{1}{2}(1 - t_1) \right) \\ (t_1, t_2) \in \mathcal{M} \\ \frac{1}{35} \leqslant t_3 < \min \left(t_2, \frac{1}{2}(1 - t_1 - t_2) \right) \\  (t_1, t_2, t_3) \notin \mathcal{I}_3 }} S\left(\mathcal{A}_{p_1 p_2 p_3}, p_3 \right) \\
\nonumber =&\ \sum_{\substack{\frac{1}{35} \leqslant t_1 < \frac{17}{35} \\ \frac{1}{35} \leqslant t_2 < \min \left(t_1, \frac{1}{2}(1 - t_1) \right) \\ (t_1, t_2) \in \mathcal{M} }} S\left(\mathcal{A}_{p_1 p_2}, v^{\frac{1}{35}}\right) - \sum_{\substack{\frac{1}{35} \leqslant t_1 < \frac{17}{35} \\ \frac{1}{35} \leqslant t_2 < \min \left(t_1, \frac{1}{2}(1 - t_1) \right) \\ (t_1, t_2) \in \mathcal{M} \\ \frac{1}{35} \leqslant t_3 < \min \left(t_2, \frac{1}{2}(1 - t_1 - t_2) \right) \\  (t_1, t_2, t_3) \in \mathcal{I}_3 }} S\left(\mathcal{A}_{p_1 p_2 p_3}, p_3 \right) \\
\nonumber &- \sum_{\substack{\frac{1}{35} \leqslant t_1 < \frac{17}{35} \\ \frac{1}{35} \leqslant t_2 < \min \left(t_1, \frac{1}{2}(1 - t_1) \right) \\ (t_1, t_2) \in \mathcal{M} \\ \frac{1}{35} \leqslant t_3 < \min \left(t_2, \frac{1}{2}(1 - t_1 - t_2) \right) \\ (t_1, t_2, t_3) \notin \mathcal{I}_3 }} S\left(\mathcal{A}_{p_1 p_2 p_3}, v^{\frac{1}{35}}\right) + \sum_{\substack{\frac{1}{35} \leqslant t_1 < \frac{17}{35} \\ \frac{1}{35} \leqslant t_2 < \min \left(t_1, \frac{1}{2}(1 - t_1) \right) \\ (t_1, t_2) \in \mathcal{M} \\ \frac{1}{35} \leqslant t_3 < \min \left(t_2, \frac{1}{2}(1 - t_1 - t_2) \right) \\ (t_1, t_2, t_3) \notin \mathcal{I}_3 \\ \frac{1}{35} \leqslant t_4 < \min \left(t_3, \frac{1}{2}(1 - t_1 - t_2 - t_3) \right) \\ (t_1, t_2, t_3, t_4) \in \mathcal{I}_4 }} S\left(\mathcal{A}_{p_1 p_2 p_3 p_4}, p_4 \right) \\
\nonumber &+ \sum_{\substack{\frac{1}{35} \leqslant t_1 < \frac{17}{35} \\ \frac{1}{35} \leqslant t_2 < \min \left(t_1, \frac{1}{2}(1 - t_1) \right) \\ (t_1, t_2) \in \mathcal{M} \\ \frac{1}{35} \leqslant t_3 < \min \left(t_2, \frac{1}{2}(1 - t_1 - t_2) \right) \\ (t_1, t_2, t_3) \notin \mathcal{I}_3 \\ \frac{1}{35} \leqslant t_4 < \min \left(t_3, \frac{1}{2}(1 - t_1 - t_2 - t_3) \right) \\ (t_1, t_2, t_3, t_4) \notin \mathcal{I}_4 }} S\left(\mathcal{A}_{p_1 p_2 p_3 p_4}, p_4 \right) \\
=&\ S_{431} - S_{432} - S_{433} + S_{434} + S_{435}.
\end{align}
Since $(t_1, t_2) \in \mathcal{M}$, we have asymptotic formulas for $S_{431}$ and $S_{433}$ by Lemma~\ref{l31}. We also have asymptotic formulas for $S_{432}$ and $S_{434}$ by Lemma~\ref{l32}.

For the remaining $S_{435}$, we first split it into three parts:
\begin{align}
\nonumber S_{435} =&\ \sum_{\substack{\frac{1}{35} \leqslant t_1 < \frac{17}{35} \\ \frac{1}{35} \leqslant t_2 < \min \left(t_1, \frac{1}{2}(1 - t_1) \right) \\ (t_1, t_2) \in \mathcal{M} \\ \frac{1}{35} \leqslant t_3 < \min \left(t_2, \frac{1}{2}(1 - t_1 - t_2) \right) \\ (t_1, t_2, t_3) \notin \mathcal{I}_3 \\ \frac{1}{35} \leqslant t_4 < \min \left(t_3, \frac{1}{2}(1 - t_1 - t_2 - t_3) \right) \\ (t_1, t_2, t_3, t_4) \notin \mathcal{I}_4 }} S\left(\mathcal{A}_{p_1 p_2 p_3 p_4}, p_4 \right) \\
\nonumber =&\ \sum_{\substack{\frac{1}{35} \leqslant t_1 < \frac{17}{35} \\ \frac{1}{35} \leqslant t_2 < \min \left(t_1, \frac{1}{2}(1 - t_1) \right) \\ (t_1, t_2) \in \mathcal{M} \\ \frac{1}{35} \leqslant t_3 < \min \left(t_2, \frac{1}{2}(1 - t_1 - t_2) \right) \\ (t_1, t_2, t_3) \notin \mathcal{I}_3 \\ \frac{1}{35} \leqslant t_4 < \min \left(t_3, \frac{1}{2}(1 - t_1 - t_2 - t_3) \right) \\ (t_1, t_2, t_3, t_4) \notin \mathcal{I}_4 \\ (t_1, t_2, t_3, t_4, t_4) \in \mathcal{J}_5 }} S\left(\mathcal{A}_{p_1 p_2 p_3 p_4}, p_4 \right) \\
\nonumber &+ \sum_{\substack{\frac{1}{35} \leqslant t_1 < \frac{17}{35} \\ \frac{1}{35} \leqslant t_2 < \min \left(t_1, \frac{1}{2}(1 - t_1) \right) \\ (t_1, t_2) \in \mathcal{M} \\ \frac{1}{35} \leqslant t_3 < \min \left(t_2, \frac{1}{2}(1 - t_1 - t_2) \right) \\ (t_1, t_2, t_3) \notin \mathcal{I}_3 \\ \frac{1}{35} \leqslant t_4 < \min \left(t_3, \frac{1}{2}(1 - t_1 - t_2 - t_3) \right) \\ (t_1, t_2, t_3, t_4) \notin \mathcal{I}_4 \\ (t_1, t_2, t_3, t_4, t_4) \notin \mathcal{J}_5 \\ (t_1, t_2, t_3, t_4) \in \mathcal{J}_4 \text{ and } (1 - t_1 - t_2 - t_3 - t_4, t_2, t_3, t_4) \in \mathcal{J}_4 }} S\left(\mathcal{A}_{p_1 p_2 p_3 p_4}, p_4 \right) \\
\nonumber &+ \sum_{\substack{\frac{1}{35} \leqslant t_1 < \frac{17}{35} \\ \frac{1}{35} \leqslant t_2 < \min \left(t_1, \frac{1}{2}(1 - t_1) \right) \\ (t_1, t_2) \in \mathcal{M} \\ \frac{1}{35} \leqslant t_3 < \min \left(t_2, \frac{1}{2}(1 - t_1 - t_2) \right) \\ (t_1, t_2, t_3) \notin \mathcal{I}_3 \\ \frac{1}{35} \leqslant t_4 < \min \left(t_3, \frac{1}{2}(1 - t_1 - t_2 - t_3) \right) \\ (t_1, t_2, t_3, t_4) \notin \mathcal{I}_4 \\ (t_1, t_2, t_3, t_4, t_4) \notin \mathcal{J}_5 \\ \text{either } (t_1, t_2, t_3, t_4) \notin \mathcal{J}_4 \text{ or } (1 - t_1 - t_2 - t_3 - t_4, t_2, t_3, t_4) \notin \mathcal{J}_4 }} S\left(\mathcal{A}_{p_1 p_2 p_3 p_4}, p_4 \right) \\
=&\ S_{4351} + S_{4352} + S_{4353}.
\end{align}

We have three ways to get more possible savings corresponding to each sum above: For $S_{4351}$ we can use Buchstab's identity twice more in a straightforward manner to get
\begin{align}
\nonumber S_{4351} =&\ \sum_{\substack{\frac{1}{35} \leqslant t_1 < \frac{17}{35} \\ \frac{1}{35} \leqslant t_2 < \min \left(t_1, \frac{1}{2}(1 - t_1) \right) \\ (t_1, t_2) \in \mathcal{M} \\ \frac{1}{35} \leqslant t_3 < \min \left(t_2, \frac{1}{2}(1 - t_1 - t_2) \right) \\ (t_1, t_2, t_3) \notin \mathcal{I}_3 \\ \frac{1}{35} \leqslant t_4 < \min \left(t_3, \frac{1}{2}(1 - t_1 - t_2 - t_3) \right) \\ (t_1, t_2, t_3, t_4) \notin \mathcal{I}_4 \\ (t_1, t_2, t_3, t_4, t_4) \in \mathcal{J}_5 }} S\left(\mathcal{A}_{p_1 p_2 p_3 p_4}, p_4 \right) \\
\nonumber =&\ \sum_{\substack{\frac{1}{35} \leqslant t_1 < \frac{17}{35} \\ \frac{1}{35} \leqslant t_2 < \min \left(t_1, \frac{1}{2}(1 - t_1) \right) \\ (t_1, t_2) \in \mathcal{M} \\ \frac{1}{35} \leqslant t_3 < \min \left(t_2, \frac{1}{2}(1 - t_1 - t_2) \right) \\ (t_1, t_2, t_3) \notin \mathcal{I}_3 \\ \frac{1}{35} \leqslant t_4 < \min \left(t_3, \frac{1}{2}(1 - t_1 - t_2 - t_3) \right) \\ (t_1, t_2, t_3, t_4) \notin \mathcal{I}_4 \\ (t_1, t_2, t_3, t_4, t_4) \in \mathcal{J}_5 }} S\left(\mathcal{A}_{p_1 p_2 p_3 p_4}, v^{\frac{1}{35}} \right) - \sum_{\substack{\frac{1}{35} \leqslant t_1 < \frac{17}{35} \\ \frac{1}{35} \leqslant t_2 < \min \left(t_1, \frac{1}{2}(1 - t_1) \right) \\ (t_1, t_2) \in \mathcal{M} \\ \frac{1}{35} \leqslant t_3 < \min \left(t_2, \frac{1}{2}(1 - t_1 - t_2) \right) \\ (t_1, t_2, t_3) \notin \mathcal{I}_3 \\ \frac{1}{35} \leqslant t_4 < \min \left(t_3, \frac{1}{2}(1 - t_1 - t_2 - t_3) \right) \\ (t_1, t_2, t_3, t_4) \notin \mathcal{I}_4 \\ (t_1, t_2, t_3, t_4, t_4) \in \mathcal{J}_5 \\ \frac{1}{35} \leqslant t_5 < \min \left(t_4, \frac{1}{2}(1 - t_1 - t_2 - t_3 - t_4) \right) \\ (t_1, t_2, t_3, t_4, t_5) \in \mathcal{I}_5 }} S\left(\mathcal{A}_{p_1 p_2 p_3 p_4 p_5}, p_5 \right) \\
\nonumber &- \sum_{\substack{\frac{1}{35} \leqslant t_1 < \frac{17}{35} \\ \frac{1}{35} \leqslant t_2 < \min \left(t_1, \frac{1}{2}(1 - t_1) \right) \\ (t_1, t_2) \in \mathcal{M} \\ \frac{1}{35} \leqslant t_3 < \min \left(t_2, \frac{1}{2}(1 - t_1 - t_2) \right) \\ (t_1, t_2, t_3) \notin \mathcal{I}_3 \\ \frac{1}{35} \leqslant t_4 < \min \left(t_3, \frac{1}{2}(1 - t_1 - t_2 - t_3) \right) \\ (t_1, t_2, t_3, t_4) \notin \mathcal{I}_4 \\ (t_1, t_2, t_3, t_4, t_4) \in \mathcal{J}_5 \\ \frac{1}{35} \leqslant t_5 < \min \left(t_4, \frac{1}{2}(1 - t_1 - t_2 - t_3 - t_4) \right) \\ (t_1, t_2, t_3, t_4, t_5) \notin \mathcal{I}_5 }} S\left(\mathcal{A}_{p_1 p_2 p_3 p_4 p_5}, p_5 \right) \\
\nonumber =&\ \sum_{\substack{\frac{1}{35} \leqslant t_1 < \frac{17}{35} \\ \frac{1}{35} \leqslant t_2 < \min \left(t_1, \frac{1}{2}(1 - t_1) \right) \\ (t_1, t_2) \in \mathcal{M} \\ \frac{1}{35} \leqslant t_3 < \min \left(t_2, \frac{1}{2}(1 - t_1 - t_2) \right) \\ (t_1, t_2, t_3) \notin \mathcal{I}_3 \\ \frac{1}{35} \leqslant t_4 < \min \left(t_3, \frac{1}{2}(1 - t_1 - t_2 - t_3) \right) \\ (t_1, t_2, t_3, t_4) \notin \mathcal{I}_4 \\ (t_1, t_2, t_3, t_4, t_4) \in \mathcal{J}_5 }} S\left(\mathcal{A}_{p_1 p_2 p_3 p_4}, v^{\frac{1}{35}} \right) - \sum_{\substack{\frac{1}{35} \leqslant t_1 < \frac{17}{35} \\ \frac{1}{35} \leqslant t_2 < \min \left(t_1, \frac{1}{2}(1 - t_1) \right) \\ (t_1, t_2) \in \mathcal{M} \\ \frac{1}{35} \leqslant t_3 < \min \left(t_2, \frac{1}{2}(1 - t_1 - t_2) \right) \\ (t_1, t_2, t_3) \notin \mathcal{I}_3 \\ \frac{1}{35} \leqslant t_4 < \min \left(t_3, \frac{1}{2}(1 - t_1 - t_2 - t_3) \right) \\ (t_1, t_2, t_3, t_4) \notin \mathcal{I}_4 \\ (t_1, t_2, t_3, t_4, t_4) \in \mathcal{J}_5 \\ \frac{1}{35} \leqslant t_5 < \min \left(t_4, \frac{1}{2}(1 - t_1 - t_2 - t_3 - t_4) \right) \\ (t_1, t_2, t_3, t_4, t_5) \in \mathcal{I}_5 }} S\left(\mathcal{A}_{p_1 p_2 p_3 p_4 p_5}, p_5 \right) \\
\nonumber &- \sum_{\substack{\frac{1}{35} \leqslant t_1 < \frac{17}{35} \\ \frac{1}{35} \leqslant t_2 < \min \left(t_1, \frac{1}{2}(1 - t_1) \right) \\ (t_1, t_2) \in \mathcal{M} \\ \frac{1}{35} \leqslant t_3 < \min \left(t_2, \frac{1}{2}(1 - t_1 - t_2) \right) \\ (t_1, t_2, t_3) \notin \mathcal{I}_3 \\ \frac{1}{35} \leqslant t_4 < \min \left(t_3, \frac{1}{2}(1 - t_1 - t_2 - t_3) \right) \\ (t_1, t_2, t_3, t_4) \notin \mathcal{I}_4 \\ (t_1, t_2, t_3, t_4, t_4) \in \mathcal{J}_5 \\ \frac{1}{35} \leqslant t_5 < \min \left(t_4, \frac{1}{2}(1 - t_1 - t_2 - t_3 - t_4) \right) \\ (t_1, t_2, t_3, t_4, t_5) \notin \mathcal{I}_5 }} S\left(\mathcal{A}_{p_1 p_2 p_3 p_4 p_5}, v^{\frac{1}{35}} \right) \\
\nonumber &+ \sum_{\substack{\frac{1}{35} \leqslant t_1 < \frac{17}{35} \\ \frac{1}{35} \leqslant t_2 < \min \left(t_1, \frac{1}{2}(1 - t_1) \right) \\ (t_1, t_2) \in \mathcal{M} \\ \frac{1}{35} \leqslant t_3 < \min \left(t_2, \frac{1}{2}(1 - t_1 - t_2) \right) \\ (t_1, t_2, t_3) \notin \mathcal{I}_3 \\ \frac{1}{35} \leqslant t_4 < \min \left(t_3, \frac{1}{2}(1 - t_1 - t_2 - t_3) \right) \\ (t_1, t_2, t_3, t_4) \notin \mathcal{I}_4 \\ (t_1, t_2, t_3, t_4, t_4) \in \mathcal{J}_5 \\ \frac{1}{35} \leqslant t_5 < \min \left(t_4, \frac{1}{2}(1 - t_1 - t_2 - t_3 - t_4) \right) \\ (t_1, t_2, t_3, t_4, t_5) \notin \mathcal{I}_5 \\ \frac{1}{35} \leqslant t_6 < \min \left(t_5, \frac{1}{2}(1 - t_1 - t_2 - t_3 - t_4 - t_5) \right) \\ (t_1, t_2, t_3, t_4, t_5, t_6) \in \mathcal{I}_6 }} S\left(\mathcal{A}_{p_1 p_2 p_3 p_4 p_5 p_6}, p_6 \right) \\
\nonumber &+ \sum_{\substack{\frac{1}{35} \leqslant t_1 < \frac{17}{35} \\ \frac{1}{35} \leqslant t_2 < \min \left(t_1, \frac{1}{2}(1 - t_1) \right) \\ (t_1, t_2) \in \mathcal{M} \\ \frac{1}{35} \leqslant t_3 < \min \left(t_2, \frac{1}{2}(1 - t_1 - t_2) \right) \\ (t_1, t_2, t_3) \notin \mathcal{I}_3 \\ \frac{1}{35} \leqslant t_4 < \min \left(t_3, \frac{1}{2}(1 - t_1 - t_2 - t_3) \right) \\ (t_1, t_2, t_3, t_4) \notin \mathcal{I}_4 \\ (t_1, t_2, t_3, t_4, t_4) \in \mathcal{J}_5 \\ \frac{1}{35} \leqslant t_5 < \min \left(t_4, \frac{1}{2}(1 - t_1 - t_2 - t_3 - t_4) \right) \\ (t_1, t_2, t_3, t_4, t_5) \notin \mathcal{I}_5 \\ \frac{1}{35} \leqslant t_6 < \min \left(t_5, \frac{1}{2}(1 - t_1 - t_2 - t_3 - t_4 - t_5) \right) \\ (t_1, t_2, t_3, t_4, t_5, t_6) \notin \mathcal{I}_6 }} S\left(\mathcal{A}_{p_1 p_2 p_3 p_4 p_5 p_6}, p_6 \right) \\
=&\ S_{43511} - S_{43512} - S_{43513} + S_{43514} + S_{43515}.
\end{align}
We have asymptotic formulas for $S_{43512}$ and $S_{43514}$ by Lemma~\ref{l32}. Since we have $(t_1, t_2, t_3, t_4, t_4) \in \mathcal{J}_5$ in all five sums above, we also have $(t_1, t_2, t_3, t_4) \in \mathcal{J}_4$ and $(t_1, t_2, t_3, t_4, t_5) \in \mathcal{J}_5$. Now we have asymptotic formulas for $S_{43511}$ and $S_{43513}$ by Lemma~\ref{l31}. For $S_{43515}$ we can decompose those parts with $(t_1, t_2, t_3, t_4, t_5, t_6, t_6) \in \mathcal{J}_7$ in a similar way to reach an eight-dimensional sum
\begin{equation}
\sum_{\substack{\frac{1}{35} \leqslant t_1 < \frac{17}{35} \\ \frac{1}{35} \leqslant t_2 < \min \left(t_1, \frac{1}{2}(1 - t_1) \right) \\ (t_1, t_2) \in \mathcal{M} \\ \frac{1}{35} \leqslant t_3 < \min \left(t_2, \frac{1}{2}(1 - t_1 - t_2) \right) \\ (t_1, t_2, t_3) \notin \mathcal{I}_3 \\ \frac{1}{35} \leqslant t_4 < \min \left(t_3, \frac{1}{2}(1 - t_1 - t_2 - t_3) \right) \\ (t_1, t_2, t_3, t_4) \notin \mathcal{I}_4 \\ (t_1, t_2, t_3, t_4, t_4) \in \mathcal{J}_5 \\ \frac{1}{35} \leqslant t_5 < \min \left(t_4, \frac{1}{2}(1 - t_1 - t_2 - t_3 - t_4) \right) \\ (t_1, t_2, t_3, t_4, t_5) \notin \mathcal{I}_5 \\ \frac{1}{35} \leqslant t_6 < \min \left(t_5, \frac{1}{2}(1 - t_1 - t_2 - t_3 - t_4 - t_5) \right) \\ (t_1, t_2, t_3, t_4, t_5, t_6) \notin \mathcal{I}_6 \\ (t_1, t_2, t_3, t_4, t_5, t_6, t_6) \in \mathcal{J}_7 \\ \frac{1}{35} \leqslant t_7 < \min \left(t_6, \frac{1}{2}(1 - t_1 - t_2 - t_3 - t_4 - t_5 - t_6) \right) \\ (t_1, t_2, t_3, t_4, t_5, t_6, t_7) \notin \mathcal{I}_7 \\ \frac{1}{35} \leqslant t_8 < \min \left(t_7, \frac{1}{2}(1 - t_1 - t_2 - t_3 - t_4 - t_5 - t_6 - t_7) \right) \\ (t_1, t_2, t_3, t_4, t_5, t_6, t_7, t_8) \notin \mathcal{I}_8 }} S\left(\mathcal{A}_{p_1 p_2 p_3 p_4 p_5 p_6 p_7 p_8}, p_8 \right),
\end{equation}
and we discard the remaining parts where $(t_1, t_2, t_3, t_4, t_5, t_6, t_6) \notin \mathcal{J}_7$:
\begin{equation}
\sum_{\substack{\frac{1}{35} \leqslant t_1 < \frac{17}{35} \\ \frac{1}{35} \leqslant t_2 < \min \left(t_1, \frac{1}{2}(1 - t_1) \right) \\ (t_1, t_2) \in \mathcal{M} \\ \frac{1}{35} \leqslant t_3 < \min \left(t_2, \frac{1}{2}(1 - t_1 - t_2) \right) \\ (t_1, t_2, t_3) \notin \mathcal{I}_3 \\ \frac{1}{35} \leqslant t_4 < \min \left(t_3, \frac{1}{2}(1 - t_1 - t_2 - t_3) \right) \\ (t_1, t_2, t_3, t_4) \notin \mathcal{I}_4 \\ (t_1, t_2, t_3, t_4, t_4) \in \mathcal{J}_5 \\ \frac{1}{35} \leqslant t_5 < \min \left(t_4, \frac{1}{2}(1 - t_1 - t_2 - t_3 - t_4) \right) \\ (t_1, t_2, t_3, t_4, t_5) \notin \mathcal{I}_5 \\ \frac{1}{35} \leqslant t_6 < \min \left(t_5, \frac{1}{2}(1 - t_1 - t_2 - t_3 - t_4 - t_5) \right) \\ (t_1, t_2, t_3, t_4, t_5, t_6) \notin \mathcal{I}_6 \\ (t_1, t_2, t_3, t_4, t_5, t_6, t_6) \notin \mathcal{J}_7}} S\left(\mathcal{A}_{p_1 p_2 p_3 p_4 p_5 p_6}, p_6 \right).
\end{equation}
The corresponding loss from $S_{4351}$ are two sums in (11) and (12).

For $S_{4352}$ we cannot use Buchstab's identity twice more in a straightforward manner, but we can use Buchstab's identity twice more with a variable role-reversal, which can be seen as ``an application of Buchstab's identity on large prime variables''. We first apply Buchstab's identity once to get
\begin{align}
\nonumber S_{4352} =&\ \sum_{\substack{\frac{1}{35} \leqslant t_1 < \frac{17}{35} \\ \frac{1}{35} \leqslant t_2 < \min \left(t_1, \frac{1}{2}(1 - t_1) \right) \\ (t_1, t_2) \in \mathcal{M} \\ \frac{1}{35} \leqslant t_3 < \min \left(t_2, \frac{1}{2}(1 - t_1 - t_2) \right) \\ (t_1, t_2, t_3) \notin \mathcal{I}_3 \\ \frac{1}{35} \leqslant t_4 < \min \left(t_3, \frac{1}{2}(1 - t_1 - t_2 - t_3) \right) \\ (t_1, t_2, t_3, t_4) \notin \mathcal{I}_4 \\ (t_1, t_2, t_3, t_4, t_4) \notin \mathcal{J}_5 \\ (t_1, t_2, t_3, t_4) \in \mathcal{J}_4 \text{ and } (1 - t_1 - t_2 - t_3 - t_4, t_2, t_3, t_4) \in \mathcal{J}_4 }} S\left(\mathcal{A}_{p_1 p_2 p_3 p_4}, p_4 \right) \\
\nonumber =&\ \sum_{\substack{\frac{1}{35} \leqslant t_1 < \frac{17}{35} \\ \frac{1}{35} \leqslant t_2 < \min \left(t_1, \frac{1}{2}(1 - t_1) \right) \\ (t_1, t_2) \in \mathcal{M} \\ \frac{1}{35} \leqslant t_3 < \min \left(t_2, \frac{1}{2}(1 - t_1 - t_2) \right) \\ (t_1, t_2, t_3) \notin \mathcal{I}_3 \\ \frac{1}{35} \leqslant t_4 < \min \left(t_3, \frac{1}{2}(1 - t_1 - t_2 - t_3) \right) \\ (t_1, t_2, t_3, t_4) \notin \mathcal{I}_4 \\ (t_1, t_2, t_3, t_4, t_4) \notin \mathcal{J}_5 \\ (t_1, t_2, t_3, t_4) \in \mathcal{J}_4 \text{ and } (1 - t_1 - t_2 - t_3 - t_4, t_2, t_3, t_4) \in \mathcal{J}_4 }} S\left(\mathcal{A}_{p_1 p_2 p_3 p_4}, v^{\frac{1}{35}} \right) \\
\nonumber &- \sum_{\substack{\frac{1}{35} \leqslant t_1 < \frac{17}{35} \\ \frac{1}{35} \leqslant t_2 < \min \left(t_1, \frac{1}{2}(1 - t_1) \right) \\ (t_1, t_2) \in \mathcal{M} \\ \frac{1}{35} \leqslant t_3 < \min \left(t_2, \frac{1}{2}(1 - t_1 - t_2) \right) \\ (t_1, t_2, t_3) \notin \mathcal{I}_3 \\ \frac{1}{35} \leqslant t_4 < \min \left(t_3, \frac{1}{2}(1 - t_1 - t_2 - t_3) \right) \\ (t_1, t_2, t_3, t_4) \notin \mathcal{I}_4 \\ (t_1, t_2, t_3, t_4, t_4) \notin \mathcal{J}_5 \\ (t_1, t_2, t_3, t_4) \in \mathcal{J}_4 \text{ and } (1 - t_1 - t_2 - t_3 - t_4, t_2, t_3, t_4) \in \mathcal{J}_4 \\ \frac{1}{35} \leqslant t_5 < \min \left(t_4, \frac{1}{2}(1 - t_1 - t_2 - t_3 - t_4) \right) \\ (t_1, t_2, t_3, t_4, t_5) \in \mathcal{I}_5 }} S\left(\mathcal{A}_{p_1 p_2 p_3 p_4 p_5}, p_5 \right) \\
&- \sum_{\substack{\frac{1}{35} \leqslant t_1 < \frac{17}{35} \\ \frac{1}{35} \leqslant t_2 < \min \left(t_1, \frac{1}{2}(1 - t_1) \right) \\ (t_1, t_2) \in \mathcal{M} \\ \frac{1}{35} \leqslant t_3 < \min \left(t_2, \frac{1}{2}(1 - t_1 - t_2) \right) \\ (t_1, t_2, t_3) \notin \mathcal{I}_3 \\ \frac{1}{35} \leqslant t_4 < \min \left(t_3, \frac{1}{2}(1 - t_1 - t_2 - t_3) \right) \\ (t_1, t_2, t_3, t_4) \notin \mathcal{I}_4 \\ (t_1, t_2, t_3, t_4, t_4) \notin \mathcal{J}_5 \\ (t_1, t_2, t_3, t_4) \in \mathcal{J}_4 \text{ and } (1 - t_1 - t_2 - t_3 - t_4, t_2, t_3, t_4) \in \mathcal{J}_4 \\ \frac{1}{35} \leqslant t_5 < \min \left(t_4, \frac{1}{2}(1 - t_1 - t_2 - t_3 - t_4) \right) \\ (t_1, t_2, t_3, t_4, t_5) \notin \mathcal{I}_5 }} S\left(\mathcal{A}_{p_1 p_2 p_3 p_4 p_5}, p_5 \right).
\end{align}
We have asymptotic formulas for the first and second sum on the right-hand side of (13) by Lemma~\ref{l31} and Lemma~\ref{l32} respectively. The last sum on the right-hand side of (13) counts numbers of the form $p_1 p_2 p_3 p_4 p_5 \beta_1$, where $\beta_1 \sim v^{1 - t_1 - t_2 - t_3 - t_4 - t_5}$ and $\left(\beta_1, P(p_5)\right) = 1$. Now we reverse the roles of $p_1$ and $\beta_1$:
\begin{align}
\nonumber &\ \sum_{\substack{\frac{1}{35} \leqslant t_1 < \frac{17}{35} \\ \frac{1}{35} \leqslant t_2 < \min \left(t_1, \frac{1}{2}(1 - t_1) \right) \\ (t_1, t_2) \in \mathcal{M} \\ \frac{1}{35} \leqslant t_3 < \min \left(t_2, \frac{1}{2}(1 - t_1 - t_2) \right) \\ (t_1, t_2, t_3) \notin \mathcal{I}_3 \\ \frac{1}{35} \leqslant t_4 < \min \left(t_3, \frac{1}{2}(1 - t_1 - t_2 - t_3) \right) \\ (t_1, t_2, t_3, t_4) \notin \mathcal{I}_4 \\ (t_1, t_2, t_3, t_4, t_4) \notin \mathcal{J}_5 \\ (t_1, t_2, t_3, t_4) \in \mathcal{J}_4 \text{ and } (1 - t_1 - t_2 - t_3 - t_4, t_2, t_3, t_4) \in \mathcal{J}_4 \\ \frac{1}{35} \leqslant t_5 < \min \left(t_4, \frac{1}{2}(1 - t_1 - t_2 - t_3 - t_4) \right) \\ (t_1, t_2, t_3, t_4, t_5) \notin \mathcal{I}_5 }} S\left(\mathcal{A}_{p_1 p_2 p_3 p_4 p_5}, p_5 \right) \\
=&\ \sum_{\substack{\frac{1}{35} \leqslant t_1 < \frac{17}{35} \\ \frac{1}{35} \leqslant t_2 < \min \left(t_1, \frac{1}{2}(1 - t_1) \right) \\ (t_1, t_2) \in \mathcal{M} \\ \frac{1}{35} \leqslant t_3 < \min \left(t_2, \frac{1}{2}(1 - t_1 - t_2) \right) \\ (t_1, t_2, t_3) \notin \mathcal{I}_3 \\ \frac{1}{35} \leqslant t_4 < \min \left(t_3, \frac{1}{2}(1 - t_1 - t_2 - t_3) \right) \\ (t_1, t_2, t_3, t_4) \notin \mathcal{I}_4 \\ (t_1, t_2, t_3, t_4, t_4) \notin \mathcal{J}_5 \\ (t_1, t_2, t_3, t_4) \in \mathcal{J}_4 \text{ and } (1 - t_1 - t_2 - t_3 - t_4, t_2, t_3, t_4) \in \mathcal{J}_4 \\ \frac{1}{35} \leqslant t_5 < \min \left(t_4, \frac{1}{2}(1 - t_1 - t_2 - t_3 - t_4) \right) \\ (t_1, t_2, t_3, t_4, t_5) \notin \mathcal{I}_5 }} S\left(\mathcal{A}_{\beta_1 p_2 p_3 p_4 p_5}, \left(\frac{2v}{\beta_1 p_2 p_3 p_4 p_5} \right)^{\frac{1}{2}} \right).
\end{align}
Note that the sum on the right-hand side of (14) counts numbers of the form $\beta_1 p_2 p_3 p_4 p_5 p_1$. Now we perform a Buchstab iteration on the variable $p_1$ to get
\begin{align}
\nonumber &\ \sum_{\substack{\frac{1}{35} \leqslant t_1 < \frac{17}{35} \\ \frac{1}{35} \leqslant t_2 < \min \left(t_1, \frac{1}{2}(1 - t_1) \right) \\ (t_1, t_2) \in \mathcal{M} \\ \frac{1}{35} \leqslant t_3 < \min \left(t_2, \frac{1}{2}(1 - t_1 - t_2) \right) \\ (t_1, t_2, t_3) \notin \mathcal{I}_3 \\ \frac{1}{35} \leqslant t_4 < \min \left(t_3, \frac{1}{2}(1 - t_1 - t_2 - t_3) \right) \\ (t_1, t_2, t_3, t_4) \notin \mathcal{I}_4 \\ (t_1, t_2, t_3, t_4, t_4) \notin \mathcal{J}_5 \\ (t_1, t_2, t_3, t_4) \in \mathcal{J}_4 \text{ and } (1 - t_1 - t_2 - t_3 - t_4, t_2, t_3, t_4) \in \mathcal{J}_4 \\ \frac{1}{35} \leqslant t_5 < \min \left(t_4, \frac{1}{2}(1 - t_1 - t_2 - t_3 - t_4) \right) \\ (t_1, t_2, t_3, t_4, t_5) \notin \mathcal{I}_5 }} S\left(\mathcal{A}_{\beta_1 p_2 p_3 p_4 p_5}, \left(\frac{2v}{\beta_1 p_2 p_3 p_4 p_5} \right)^{\frac{1}{2}} \right) \\
\nonumber =&\ \sum_{\substack{\frac{1}{35} \leqslant t_1 < \frac{17}{35} \\ \frac{1}{35} \leqslant t_2 < \min \left(t_1, \frac{1}{2}(1 - t_1) \right) \\ (t_1, t_2) \in \mathcal{M} \\ \frac{1}{35} \leqslant t_3 < \min \left(t_2, \frac{1}{2}(1 - t_1 - t_2) \right) \\ (t_1, t_2, t_3) \notin \mathcal{I}_3 \\ \frac{1}{35} \leqslant t_4 < \min \left(t_3, \frac{1}{2}(1 - t_1 - t_2 - t_3) \right) \\ (t_1, t_2, t_3, t_4) \notin \mathcal{I}_4 \\ (t_1, t_2, t_3, t_4, t_4) \notin \mathcal{J}_5 \\ (t_1, t_2, t_3, t_4) \in \mathcal{J}_4 \text{ and } (1 - t_1 - t_2 - t_3 - t_4, t_2, t_3, t_4) \in \mathcal{J}_4 \\ \frac{1}{35} \leqslant t_5 < \min \left(t_4, \frac{1}{2}(1 - t_1 - t_2 - t_3 - t_4) \right) \\ (t_1, t_2, t_3, t_4, t_5) \notin \mathcal{I}_5 }} S\left(\mathcal{A}_{\beta_1 p_2 p_3 p_4 p_5}, v^{\frac{1}{35}} \right) \\
\nonumber &- \sum_{\substack{\frac{1}{35} \leqslant t_1 < \frac{17}{35} \\ \frac{1}{35} \leqslant t_2 < \min \left(t_1, \frac{1}{2}(1 - t_1) \right) \\ (t_1, t_2) \in \mathcal{M} \\ \frac{1}{35} \leqslant t_3 < \min \left(t_2, \frac{1}{2}(1 - t_1 - t_2) \right) \\ (t_1, t_2, t_3) \notin \mathcal{I}_3 \\ \frac{1}{35} \leqslant t_4 < \min \left(t_3, \frac{1}{2}(1 - t_1 - t_2 - t_3) \right) \\ (t_1, t_2, t_3, t_4) \notin \mathcal{I}_4 \\ (t_1, t_2, t_3, t_4, t_4) \notin \mathcal{J}_5 \\ (t_1, t_2, t_3, t_4) \in \mathcal{J}_4 \text{ and } (1 - t_1 - t_2 - t_3 - t_4, t_2, t_3, t_4) \in \mathcal{J}_4 \\ \frac{1}{35} \leqslant t_5 < \min \left(t_4, \frac{1}{2}(1 - t_1 - t_2 - t_3 - t_4) \right) \\ (t_1, t_2, t_3, t_4, t_5) \notin \mathcal{I}_5 \\ \frac{1}{35} \leqslant t_6 < \frac{1}{2} t_1 \\ (1 - t_1 - t_2 - t_3 - t_4 - t_5, t_2, t_3, t_4, t_5, t_6) \in \mathcal{I}_6 }} S\left(\mathcal{A}_{\beta_1 p_2 p_3 p_4 p_5 p_6}, p_6 \right) \\
&- \sum_{\substack{\frac{1}{35} \leqslant t_1 < \frac{17}{35} \\ \frac{1}{35} \leqslant t_2 < \min \left(t_1, \frac{1}{2}(1 - t_1) \right) \\ (t_1, t_2) \in \mathcal{M} \\ \frac{1}{35} \leqslant t_3 < \min \left(t_2, \frac{1}{2}(1 - t_1 - t_2) \right) \\ (t_1, t_2, t_3) \notin \mathcal{I}_3 \\ \frac{1}{35} \leqslant t_4 < \min \left(t_3, \frac{1}{2}(1 - t_1 - t_2 - t_3) \right) \\ (t_1, t_2, t_3, t_4) \notin \mathcal{I}_4 \\ (t_1, t_2, t_3, t_4, t_4) \notin \mathcal{J}_5 \\ (t_1, t_2, t_3, t_4) \in \mathcal{J}_4 \text{ and } (1 - t_1 - t_2 - t_3 - t_4, t_2, t_3, t_4) \in \mathcal{J}_4 \\ \frac{1}{35} \leqslant t_5 < \min \left(t_4, \frac{1}{2}(1 - t_1 - t_2 - t_3 - t_4) \right) \\ (t_1, t_2, t_3, t_4, t_5) \notin \mathcal{I}_5 \\ \frac{1}{35} \leqslant t_6 < \frac{1}{2} t_1 \\ (1 - t_1 - t_2 - t_3 - t_4 - t_5, t_2, t_3, t_4, t_5, t_6) \notin \mathcal{I}_6 }} S\left(\mathcal{A}_{\beta_1 p_2 p_3 p_4 p_5 p_6}, p_6 \right).
\end{align}
Combining (13)--(15), we have
\begin{align}
\nonumber S_{4352} =&\ \sum_{\substack{\frac{1}{35} \leqslant t_1 < \frac{17}{35} \\ \frac{1}{35} \leqslant t_2 < \min \left(t_1, \frac{1}{2}(1 - t_1) \right) \\ (t_1, t_2) \in \mathcal{M} \\ \frac{1}{35} \leqslant t_3 < \min \left(t_2, \frac{1}{2}(1 - t_1 - t_2) \right) \\ (t_1, t_2, t_3) \notin \mathcal{I}_3 \\ \frac{1}{35} \leqslant t_4 < \min \left(t_3, \frac{1}{2}(1 - t_1 - t_2 - t_3) \right) \\ (t_1, t_2, t_3, t_4) \notin \mathcal{I}_4 \\ (t_1, t_2, t_3, t_4, t_4) \notin \mathcal{J}_5 \\ (t_1, t_2, t_3, t_4) \in \mathcal{J}_4 \text{ and } (1 - t_1 - t_2 - t_3 - t_4, t_2, t_3, t_4) \in \mathcal{J}_4 }} S\left(\mathcal{A}_{p_1 p_2 p_3 p_4}, p_4 \right) \\
\nonumber =&\ \sum_{\substack{\frac{1}{35} \leqslant t_1 < \frac{17}{35} \\ \frac{1}{35} \leqslant t_2 < \min \left(t_1, \frac{1}{2}(1 - t_1) \right) \\ (t_1, t_2) \in \mathcal{M} \\ \frac{1}{35} \leqslant t_3 < \min \left(t_2, \frac{1}{2}(1 - t_1 - t_2) \right) \\ (t_1, t_2, t_3) \notin \mathcal{I}_3 \\ \frac{1}{35} \leqslant t_4 < \min \left(t_3, \frac{1}{2}(1 - t_1 - t_2 - t_3) \right) \\ (t_1, t_2, t_3, t_4) \notin \mathcal{I}_4 \\ (t_1, t_2, t_3, t_4, t_4) \notin \mathcal{J}_5 \\ (t_1, t_2, t_3, t_4) \in \mathcal{J}_4 \text{ and } (1 - t_1 - t_2 - t_3 - t_4, t_2, t_3, t_4) \in \mathcal{J}_4 }} S\left(\mathcal{A}_{p_1 p_2 p_3 p_4}, v^{\frac{1}{35}} \right) \\
\nonumber &- \sum_{\substack{\frac{1}{35} \leqslant t_1 < \frac{17}{35} \\ \frac{1}{35} \leqslant t_2 < \min \left(t_1, \frac{1}{2}(1 - t_1) \right) \\ (t_1, t_2) \in \mathcal{M} \\ \frac{1}{35} \leqslant t_3 < \min \left(t_2, \frac{1}{2}(1 - t_1 - t_2) \right) \\ (t_1, t_2, t_3) \notin \mathcal{I}_3 \\ \frac{1}{35} \leqslant t_4 < \min \left(t_3, \frac{1}{2}(1 - t_1 - t_2 - t_3) \right) \\ (t_1, t_2, t_3, t_4) \notin \mathcal{I}_4 \\ (t_1, t_2, t_3, t_4, t_4) \notin \mathcal{J}_5 \\ (t_1, t_2, t_3, t_4) \in \mathcal{J}_4 \text{ and } (1 - t_1 - t_2 - t_3 - t_4, t_2, t_3, t_4) \in \mathcal{J}_4 \\ \frac{1}{35} \leqslant t_5 < \min \left(t_4, \frac{1}{2}(1 - t_1 - t_2 - t_3 - t_4) \right) \\ (t_1, t_2, t_3, t_4, t_5) \in \mathcal{I}_5 }} S\left(\mathcal{A}_{p_1 p_2 p_3 p_4 p_5}, p_5 \right) \\
\nonumber &- \sum_{\substack{\frac{1}{35} \leqslant t_1 < \frac{17}{35} \\ \frac{1}{35} \leqslant t_2 < \min \left(t_1, \frac{1}{2}(1 - t_1) \right) \\ (t_1, t_2) \in \mathcal{M} \\ \frac{1}{35} \leqslant t_3 < \min \left(t_2, \frac{1}{2}(1 - t_1 - t_2) \right) \\ (t_1, t_2, t_3) \notin \mathcal{I}_3 \\ \frac{1}{35} \leqslant t_4 < \min \left(t_3, \frac{1}{2}(1 - t_1 - t_2 - t_3) \right) \\ (t_1, t_2, t_3, t_4) \notin \mathcal{I}_4 \\ (t_1, t_2, t_3, t_4, t_4) \notin \mathcal{J}_5 \\ (t_1, t_2, t_3, t_4) \in \mathcal{J}_4 \text{ and } (1 - t_1 - t_2 - t_3 - t_4, t_2, t_3, t_4) \in \mathcal{J}_4 \\ \frac{1}{35} \leqslant t_5 < \min \left(t_4, \frac{1}{2}(1 - t_1 - t_2 - t_3 - t_4) \right) \\ (t_1, t_2, t_3, t_4, t_5) \notin \mathcal{I}_5 }} S\left(\mathcal{A}_{\beta_1 p_2 p_3 p_4 p_5}, v^{\frac{1}{35}} \right) \\
\nonumber &+ \sum_{\substack{\frac{1}{35} \leqslant t_1 < \frac{17}{35} \\ \frac{1}{35} \leqslant t_2 < \min \left(t_1, \frac{1}{2}(1 - t_1) \right) \\ (t_1, t_2) \in \mathcal{M} \\ \frac{1}{35} \leqslant t_3 < \min \left(t_2, \frac{1}{2}(1 - t_1 - t_2) \right) \\ (t_1, t_2, t_3) \notin \mathcal{I}_3 \\ \frac{1}{35} \leqslant t_4 < \min \left(t_3, \frac{1}{2}(1 - t_1 - t_2 - t_3) \right) \\ (t_1, t_2, t_3, t_4) \notin \mathcal{I}_4 \\ (t_1, t_2, t_3, t_4, t_4) \notin \mathcal{J}_5 \\ (t_1, t_2, t_3, t_4) \in \mathcal{J}_4 \text{ and } (1 - t_1 - t_2 - t_3 - t_4, t_2, t_3, t_4) \in \mathcal{J}_4 \\ \frac{1}{35} \leqslant t_5 < \min \left(t_4, \frac{1}{2}(1 - t_1 - t_2 - t_3 - t_4) \right) \\ (t_1, t_2, t_3, t_4, t_5) \notin \mathcal{I}_5 \\ \frac{1}{35} \leqslant t_6 < \frac{1}{2} t_1 \\ (1 - t_1 - t_2 - t_3 - t_4 - t_5, t_2, t_3, t_4, t_5, t_6) \in \mathcal{I}_6 }} S\left(\mathcal{A}_{\beta_1 p_2 p_3 p_4 p_5 p_6}, p_6 \right) \\
\nonumber &+ \sum_{\substack{\frac{1}{35} \leqslant t_1 < \frac{17}{35} \\ \frac{1}{35} \leqslant t_2 < \min \left(t_1, \frac{1}{2}(1 - t_1) \right) \\ (t_1, t_2) \in \mathcal{M} \\ \frac{1}{35} \leqslant t_3 < \min \left(t_2, \frac{1}{2}(1 - t_1 - t_2) \right) \\ (t_1, t_2, t_3) \notin \mathcal{I}_3 \\ \frac{1}{35} \leqslant t_4 < \min \left(t_3, \frac{1}{2}(1 - t_1 - t_2 - t_3) \right) \\ (t_1, t_2, t_3, t_4) \notin \mathcal{I}_4 \\ (t_1, t_2, t_3, t_4, t_4) \notin \mathcal{J}_5 \\ (t_1, t_2, t_3, t_4) \in \mathcal{J}_4 \text{ and } (1 - t_1 - t_2 - t_3 - t_4, t_2, t_3, t_4) \in \mathcal{J}_4 \\ \frac{1}{35} \leqslant t_5 < \min \left(t_4, \frac{1}{2}(1 - t_1 - t_2 - t_3 - t_4) \right) \\ (t_1, t_2, t_3, t_4, t_5) \notin \mathcal{I}_5 \\ \frac{1}{35} \leqslant t_6 < \frac{1}{2} t_1 \\ (1 - t_1 - t_2 - t_3 - t_4 - t_5, t_2, t_3, t_4, t_5, t_6) \notin \mathcal{I}_6 }} S\left(\mathcal{A}_{\beta_1 p_2 p_3 p_4 p_5 p_6}, p_6 \right) \\
=&\ S_{43521} - S_{43522} - S_{43523} + S_{43524} + S_{43525}.
\end{align}
We have an asymptotic formula for $S_{43521}$ by Lemma~\ref{l31}. We have asymptotic formulas for $S_{43522}$ and $S_{43524}$ by Lemma~\ref{l32}. Since we have $\beta_1 p_5 \sim v^{1 - t_1 - t_2 - t_3 - t_4}$, we have an asymptotic formula for $S_{43523}$. For $S_{43525}$ we can decompose some parts in a way similar to that in decomposing $S_{43515}$, but now we have two almost-prime variables. Note that $S_{43525}$ counts numbers of the form $\beta_1 p_2 p_3 p_4 p_5 p_6 \beta_2$, where $\beta_2 \sim v^{t_1 - t_6}$ and $\left(\beta_2, P(p_6)\right) = 1$. Now we can perform Buchstab iterations on either $\beta_1$ or $\beta_2$. For the case of decomposing $\beta_2$, we need the condition $(1 - t_1 - t_2 - t_3 - t_4 - t_5, t_2, t_3, t_4, t_5, t_6, t_6) \in \mathcal{J}_7$, and the corresponding eight-dimensional sum is
\begin{equation}
\sum_{\substack{\frac{1}{35} \leqslant t_1 < \frac{17}{35} \\ \frac{1}{35} \leqslant t_2 < \min \left(t_1, \frac{1}{2}(1 - t_1) \right) \\ (t_1, t_2) \in \mathcal{M} \\ \frac{1}{35} \leqslant t_3 < \min \left(t_2, \frac{1}{2}(1 - t_1 - t_2) \right) \\ (t_1, t_2, t_3) \notin \mathcal{I}_3 \\ \frac{1}{35} \leqslant t_4 < \min \left(t_3, \frac{1}{2}(1 - t_1 - t_2 - t_3) \right) \\ (t_1, t_2, t_3, t_4) \notin \mathcal{I}_4 \\ (t_1, t_2, t_3, t_4, t_4) \notin \mathcal{J}_5 \\ (t_1, t_2, t_3, t_4) \in \mathcal{J}_4 \text{ and } (1 - t_1 - t_2 - t_3 - t_4, t_2, t_3, t_4) \in \mathcal{J}_4 \\ \frac{1}{35} \leqslant t_5 < \min \left(t_4, \frac{1}{2}(1 - t_1 - t_2 - t_3 - t_4) \right) \\ (t_1, t_2, t_3, t_4, t_5) \notin \mathcal{I}_5 \\ \frac{1}{35} \leqslant t_6 < \frac{1}{2} t_1 \\ (1 - t_1 - t_2 - t_3 - t_4 - t_5, t_2, t_3, t_4, t_5, t_6) \notin \mathcal{I}_6 \\ (1 - t_1 - t_2 - t_3 - t_4 - t_5, t_2, t_3, t_4, t_5, t_6, t_6) \in \mathcal{J}_7 \\ \frac{1}{35} \leqslant t_7 < \min \left(t_6, \frac{1}{2}(t_1 - t_6) \right) \\ (1 - t_1 - t_2 - t_3 - t_4 - t_5, t_2, t_3, t_4, t_5, t_6, t_7) \notin \mathcal{I}_7 \\ \frac{1}{35} \leqslant t_8 < \min \left(t_7, \frac{1}{2}(t_1 - t_6 - t_7) \right) \\ (1 - t_1 - t_2 - t_3 - t_4 - t_5, t_2, t_3, t_4, t_5, t_6, t_7, t_8) \notin \mathcal{I}_8 }} S\left(\mathcal{A}_{\beta_1 p_2 p_3 p_4 p_5 p_6 p_7 p_8}, p_8 \right).
\end{equation}
For the case of decomposing $\beta_1$, we need the condition $(t_1 - t_6, t_2, t_3, t_4, t_5, t_6, t_5) \in \mathcal{J}_7$, and the corresponding eight-dimensional sum is
\begin{equation}
\sum_{\substack{\frac{1}{35} \leqslant t_1 < \frac{17}{35} \\ \frac{1}{35} \leqslant t_2 < \min \left(t_1, \frac{1}{2}(1 - t_1) \right) \\ (t_1, t_2) \in \mathcal{M} \\ \frac{1}{35} \leqslant t_3 < \min \left(t_2, \frac{1}{2}(1 - t_1 - t_2) \right) \\ (t_1, t_2, t_3) \notin \mathcal{I}_3 \\ \frac{1}{35} \leqslant t_4 < \min \left(t_3, \frac{1}{2}(1 - t_1 - t_2 - t_3) \right) \\ (t_1, t_2, t_3, t_4) \notin \mathcal{I}_4 \\ (t_1, t_2, t_3, t_4, t_4) \notin \mathcal{J}_5 \\ (t_1, t_2, t_3, t_4) \in \mathcal{J}_4 \text{ and } (1 - t_1 - t_2 - t_3 - t_4, t_2, t_3, t_4) \in \mathcal{J}_4 \\ \frac{1}{35} \leqslant t_5 < \min \left(t_4, \frac{1}{2}(1 - t_1 - t_2 - t_3 - t_4) \right) \\ (t_1, t_2, t_3, t_4, t_5) \notin \mathcal{I}_5 \\ \frac{1}{35} \leqslant t_6 < \frac{1}{2} t_1 \\ (1 - t_1 - t_2 - t_3 - t_4 - t_5, t_2, t_3, t_4, t_5, t_6) \notin \mathcal{I}_6 \\ (t_1 - t_6, t_2, t_3, t_4, t_5, t_6, t_5) \in \mathcal{J}_7 \\ \frac{1}{35} \leqslant t_7 < \min \left(t_5, \frac{1}{2}(1 - t_1 - t_2 - t_3 - t_4 - t_5) \right) \\ (t_1 - t_6, t_2, t_3, t_4, t_5, t_6, t_7) \notin \mathcal{I}_7 \\ \frac{1}{35} \leqslant t_8 < \min \left(t_7, \frac{1}{2}(1 - t_1 - t_2 - t_3 - t_4 - t_5 - t_7) \right) \\ (t_1 - t_6, t_2, t_3, t_4, t_5, t_6, t_7, t_8) \notin \mathcal{I}_8 }} S\left(\mathcal{A}_{\beta_2 p_2 p_3 p_4 p_5 p_6 p_7 p_8}, p_8 \right).
\end{equation}
We also need to discard the remaining parts of $S_{43525}$:
\begin{equation}
\sum_{\substack{\frac{1}{35} \leqslant t_1 < \frac{17}{35} \\ \frac{1}{35} \leqslant t_2 < \min \left(t_1, \frac{1}{2}(1 - t_1) \right) \\ (t_1, t_2) \in \mathcal{M} \\ \frac{1}{35} \leqslant t_3 < \min \left(t_2, \frac{1}{2}(1 - t_1 - t_2) \right) \\ (t_1, t_2, t_3) \notin \mathcal{I}_3 \\ \frac{1}{35} \leqslant t_4 < \min \left(t_3, \frac{1}{2}(1 - t_1 - t_2 - t_3) \right) \\ (t_1, t_2, t_3, t_4) \notin \mathcal{I}_4 \\ (t_1, t_2, t_3, t_4, t_4) \notin \mathcal{J}_5 \\ (t_1, t_2, t_3, t_4) \in \mathcal{J}_4 \text{ and } (1 - t_1 - t_2 - t_3 - t_4, t_2, t_3, t_4) \in \mathcal{J}_4 \\ \frac{1}{35} \leqslant t_5 < \min \left(t_4, \frac{1}{2}(1 - t_1 - t_2 - t_3 - t_4) \right) \\ (t_1, t_2, t_3, t_4, t_5) \notin \mathcal{I}_5 \\ \frac{1}{35} \leqslant t_6 < \frac{1}{2} t_1 \\ (1 - t_1 - t_2 - t_3 - t_4 - t_5, t_2, t_3, t_4, t_5, t_6) \notin \mathcal{I}_6 \\ (1 - t_1 - t_2 - t_3 - t_4 - t_5, t_2, t_3, t_4, t_5, t_6, t_6) \notin \mathcal{J}_7 \\ (t_1 - t_6, t_2, t_3, t_4, t_5, t_6, t_5) \notin \mathcal{J}_7 }} S\left(\mathcal{A}_{\beta_1 p_2 p_3 p_4 p_5 p_6}, p_6 \right).
\end{equation}
The corresponding loss from $S_{4352}$ are three sums in (17), (18) and (19).

For the remaining $S_{4353}$ we cannot use straightforward decompositions and role-reversals, but we can use Buchstab's identity in another way to get possible savings. Note that the numbers that $S_{4353}$ counts is of the form $p_1 p_2 p_3 p_4 m_0$, where the smallest prime factor of $m_0$ is larger than $p_4$. If $t_4 > \frac{1}{2}(1-t_1 -t_2 -t_3 -t_4)$, we know that $m_0$ cannot have two or more prime factors and therefore must be a prime. Otherwise $m_0$ may have more than one prime factor, which means that we can ``split'' the smallest prime factor of $m_0$ and consider it separately:
\begin{align}
\nonumber &\ \sum_{\substack{t_1, t_2, t_3, t_4 \\ t_4 < \frac{1}{2}(1-t_1 -t_2 -t_3 -t_4)}} S\left(\mathcal{A}_{p_{1} p_{2} p_{3} p_{4}}, p_{4}\right) \\
=&\ \sum_{\substack{t_1, t_2, t_3, t_4 \\ t_4 < \frac{1}{2}(1-t_1 -t_2 -t_3 -t_4)}} S\left(\mathcal{A}_{p_{1} p_{2} p_{3} p_{4}}, \left(\frac{x}{p_1 p_2 p_3 p_4}\right)^{\frac{1}{2}}\right) + \sum_{\substack{t_1, t_2, t_3, t_4 \\ t_4 < t_5 < \frac{1}{2}(1-t_1 -t_2 -t_3 -t_4)}} S\left(\mathcal{A}_{p_{1} p_{2} p_{3} p_{4} p_{5}}, p_{5}\right).
\end{align}
The sum $S\left(\mathcal{A}_{p_{1} p_{2} p_{3} p_{4}}, \left(\frac{x}{p_1 p_2 p_3 p_4}\right)^{\frac{1}{2}}\right)$ counts numbers of the form $p_1 p_2 p_3 p_4 p^{\prime}$, where $p^{\prime} = m_0$. The sum $S\left(\mathcal{A}_{p_{1} p_{2} p_{3} p_{4} p_{5}}, p_{5}\right)$ counts numbers of the form $p_1 p_2 p_3 p_4 (p_5 m_1)$, where $p_5 m_1 = m_0$, $p_5 > p_4$ and all prime factors of $m_1$ are larger than $p_5$. In this sum we have a new variable $t_5$, which means that part of this sum may have an asymptotic formula. In this situation, we can give an asymptotic formula for part of the last sum on the right-hand side of (20), hence we can subtract its contribution from the loss from $S_{4353}$. Again, for the remaining part of $S\left(\mathcal{A}_{p_{1} p_{2} p_{3} p_{4} p_{5}}, p_{5}\right)$, we can do a similar process as (20) to subtract the contribution of the numbers $p_1 p_2 p_3 p_4 p_5 (p_6 m_2)$ that have asymptotic formulas. The above process can be rewritten as
\begin{align}
\nonumber &\ \sum_{\substack{t_1, t_2, t_3, t_4 \\ t_4 < \frac{1}{2}(1-t_1 -t_2 -t_3 -t_4)}} S\left(\mathcal{A}_{p_{1} p_{2} p_{3} p_{4}}, \left(\frac{x}{p_1 p_2 p_3 p_4}\right)^{\frac{1}{2}}\right) \\
=&\ \sum_{\substack{t_1, t_2, t_3, t_4 \\ t_4 < \frac{1}{2}(1-t_1 -t_2 -t_3 -t_4)}} S\left(\mathcal{A}_{p_{1} p_{2} p_{3} p_{4}}, p_{4}\right) - \sum_{\substack{t_1, t_2, t_3, t_4 \\ t_4 < t_5 < \frac{1}{2}(1-t_1 -t_2 -t_3 -t_4)}} S\left(\mathcal{A}_{p_{1} p_{2} p_{3} p_{4} p_{5}}, p_{5}\right).
\end{align}
Since (21) is a direct application of Buchstab's identity, we shall call (20) ``Buchstab's identity in reverse'' or ``reversed Buchstab's identity'' in the rest of our paper. Since applying reversed Buchstab's identity twice or more times gives only a small saving over high-dimensional sums, we only apply it once on $S_{4353}$ and subtract the contribution of the sum
\begin{equation}
\sum_{\substack{\frac{1}{35} \leqslant t_1 < \frac{17}{35} \\ \frac{1}{35} \leqslant t_2 < \min \left(t_1, \frac{1}{2}(1 - t_1) \right) \\ (t_1, t_2) \in \mathcal{M} \\ \frac{1}{35} \leqslant t_3 < \min \left(t_2, \frac{1}{2}(1 - t_1 - t_2) \right) \\ (t_1, t_2, t_3) \notin \mathcal{I}_3 \\ \frac{1}{35} \leqslant t_4 < \min \left(t_3, \frac{1}{2}(1 - t_1 - t_2 - t_3) \right) \\ (t_1, t_2, t_3, t_4) \notin \mathcal{I}_4 \\ (t_1, t_2, t_3, t_4, t_4) \notin \mathcal{J}_5 \\ \text{either } (t_1, t_2, t_3, t_4) \notin \mathcal{J}_4 \text{ or } (1 - t_1 - t_2 - t_3 - t_4, t_2, t_3, t_4) \notin \mathcal{J}_4 \\ t_4 < t_5 < \frac{1}{2}(1-t_1 -t_2 -t_3 -t_4) \\ (t_1, t_2, t_3, t_4, t_5) \in \mathcal{I}_5 }} S\left(\mathcal{A}_{p_1 p_2 p_3 p_4 p_5}, p_5 \right)
\end{equation}
from $S_{4353}$. We discard other parts of $S_{4353}$.

Combining all the three cases above, we get a loss from $S_{435}$ and also a loss from $S_{43}$ of
\begin{align}
\nonumber & \left( \int_{(t_1, t_2, t_3, t_4) \in U_{\mathcal{M}1} } \frac{\omega \left(\frac{1 - t_1 - t_2 - t_3 - t_4}{t_4}\right)}{t_1 t_2 t_3 t_4^2} d t_4 d t_3 d t_2 d t_1 \right) \\
\nonumber -& \left( \int_{(t_1, t_2, t_3, t_4, t_5) \in U_{\mathcal{M}2} } \frac{\omega \left(\frac{1 - t_1 - t_2 - t_3 - t_4 - t_5}{t_5}\right)}{t_1 t_2 t_3 t_4 t_5^2} d t_5 d t_4 d t_3 d t_2 d t_1 \right) \\
\nonumber +& \left( \int_{(t_1, t_2, t_3, t_4, t_5, t_6) \in U_{\mathcal{M}3} } \frac{\omega \left(\frac{1 - t_1 - t_2 - t_3 - t_4 - t_5 - t_6}{t_6}\right)}{t_1 t_2 t_3 t_4 t_5 t_6^2} d t_6 d t_5 d t_4 d t_3 d t_2 d t_1 \right) \\
\nonumber +& \left( \int_{(t_1, t_2, t_3, t_4, t_5, t_6) \in U_{\mathcal{M}4} } \frac{\omega \left(\frac{t_1 - t_6}{t_6}\right) \omega \left(\frac{1 - t_1 - t_2 - t_3 - t_4 - t_5}{t_5}\right)}{t_2 t_3 t_4 t_5^2 t_6^2} d t_6 d t_5 d t_4 d t_3 d t_2 d t_1 \right) \\
\nonumber +& \left( \int_{(t_1, t_2, t_3, t_4, t_5, t_6, t_7, t_8) \in U_{\mathcal{M}5} } \frac{\omega \left(\frac{1 - t_1 - t_2 - t_3 - t_4 - t_5 - t_6 - t_7 - t_8}{t_8}\right)}{t_1 t_2 t_3 t_4 t_5 t_6 t_7 t_8^2} d t_8 d t_7 d t_6 d t_5 d t_4 d t_3 d t_2 d t_1 \right) \\
\nonumber +& \left( \int_{(t_1, t_2, t_3, t_4, t_5, t_6, t_7, t_8) \in U_{\mathcal{M}6} } \frac{\omega \left(\frac{t_1 - t_6 - t_7 - t_8}{t_8}\right) \omega \left(\frac{1 - t_1 - t_2 - t_3 - t_4 - t_5}{t_5}\right)}{t_2 t_3 t_4 t_5^2 t_6 t_7 t_8^2} d t_8 d t_7 d t_6 d t_5 d t_4 d t_3 d t_2 d t_1 \right) \\
\nonumber +& \left( \int_{(t_1, t_2, t_3, t_4, t_5, t_6, t_7, t_8) \in U_{\mathcal{M}7} } \frac{\omega \left(\frac{t_1 - t_6}{t_6}\right) \omega \left(\frac{1 - t_1 - t_2 - t_3 - t_4 - t_5 - t_7 - t_8}{t_8}\right)}{t_2 t_3 t_4 t_5 t_6^2 t_7 t_8^2} d t_8 d t_7 d t_6 d t_5 d t_4 d t_3 d t_2 d t_1 \right) \\
\nonumber \leqslant& \left( \int_{(t_1, t_2, t_3, t_4) \in U_{\mathcal{M}1} } \frac{\omega_1 \left(\frac{1 - t_1 - t_2 - t_3 - t_4}{t_4}\right)}{t_1 t_2 t_3 t_4^2} d t_4 d t_3 d t_2 d t_1 \right) \\
\nonumber -& \left( \int_{(t_1, t_2, t_3, t_4, t_5) \in U_{\mathcal{M}2} } \frac{\omega_0 \left(\frac{1 - t_1 - t_2 - t_3 - t_4 - t_5}{t_5}\right)}{t_1 t_2 t_3 t_4 t_5^2} d t_5 d t_4 d t_3 d t_2 d t_1 \right) \\
\nonumber +& \left( \int_{(t_1, t_2, t_3, t_4, t_5, t_6) \in U_{\mathcal{M}3} } \frac{\omega_1 \left(\frac{1 - t_1 - t_2 - t_3 - t_4 - t_5 - t_6}{t_6}\right)}{t_1 t_2 t_3 t_4 t_5 t_6^2} d t_6 d t_5 d t_4 d t_3 d t_2 d t_1 \right) \\
\nonumber +& \left( \int_{(t_1, t_2, t_3, t_4, t_5, t_6) \in U_{\mathcal{M}4} } \frac{\omega_1 \left(\frac{t_1 - t_6}{t_6}\right) \omega_1 \left(\frac{1 - t_1 - t_2 - t_3 - t_4 - t_5}{t_5}\right)}{t_2 t_3 t_4 t_5^2 t_6^2} d t_6 d t_5 d t_4 d t_3 d t_2 d t_1 \right) \\
\nonumber +& \left( \int_{(t_1, t_2, t_3, t_4, t_5, t_6, t_7, t_8) \in U_{\mathcal{M}5} } \frac{\max \left(\frac{t_8}{1 - t_1 - t_2 - t_3 - t_4 - t_5 - t_6 - t_7 - t_8}, 0.5672\right)}{t_1 t_2 t_3 t_4 t_5 t_6 t_7 t_8^2} d t_8 d t_7 d t_6 d t_5 d t_4 d t_3 d t_2 d t_1 \right) \\
\nonumber +& \left( \int_{(t_1, t_2, t_3, t_4, t_5, t_6, t_7, t_8) \in U_{\mathcal{M}6} } \frac{\max \left(\frac{t_8}{t_1 - t_6 - t_7 - t_8}, 0.5672\right) \max \left(\frac{t_5}{1 - t_1 - t_2 - t_3 - t_4 - t_5}, 0.5672\right)}{t_2 t_3 t_4 t_5^2 t_6 t_7 t_8^2} d t_8 d t_7 d t_6 d t_5 d t_4 d t_3 d t_2 d t_1 \right) \\
\nonumber +& \left( \int_{(t_1, t_2, t_3, t_4, t_5, t_6, t_7, t_8) \in U_{\mathcal{M}7} } \frac{\max \left(\frac{t_6}{t_1 - t_6}, 0.5672\right) \max \left(\frac{t_8}{1 - t_1 - t_2 - t_3 - t_4 - t_5 - t_7 - t_8}, 0.5672\right)}{t_2 t_3 t_4 t_5 t_6^2 t_7 t_8^2} d t_8 d t_7 d t_6 d t_5 d t_4 d t_3 d t_2 d t_1 \right) \\
\nonumber \leqslant&\ (0.081234 - 0.002216 + 0.030572 + 0.061125 + 0.005742 + 0.000001 + 0.000001) \\
=&\ 0.176459,
\end{align}
where
\begin{align}
\nonumber U_{\mathcal{M}1}(t_1, t_2, t_3, t_4) :=&\ \left\{ (t_1, t_2) \in \mathcal{M},\ \frac{1}{35} \leqslant t_3 < \min\left(t_2, \frac{1}{2}(1 - t_1 - t_2)\right),\ (t_1, t_2, t_3) \notin \mathcal{I}_3, \right. \\
\nonumber & \quad \frac{1}{35} \leqslant t_4 < \min \left(t_3, \frac{1}{2}(1 - t_1 - t_2 - t_3) \right), \\
\nonumber & \quad (t_1, t_2, t_3, t_4) \notin \mathcal{I}_4,\ (t_1, t_2, t_3, t_4, t_4) \notin \mathcal{J}_5, \\
\nonumber & \quad \text{either } (t_1, t_2, t_3, t_4) \notin \mathcal{J}_4 \text{ or } (1 - t_1 - t_2 - t_3 - t_4, t_2, t_3, t_4) \notin \mathcal{J}_4, \\
\nonumber & \left. \quad \frac{1}{35} \leqslant t_1 < \frac{17}{35},\ \frac{1}{35} \leqslant t_2 < \min\left(t_1, \frac{1}{2}(1-t_1) \right) \right\}, \\
\nonumber U_{\mathcal{M}2}(t_1, t_2, t_3, t_4, t_5) :=&\ \left\{ (t_1, t_2) \in \mathcal{M},\ \frac{1}{35} \leqslant t_3 < \min\left(t_2, \frac{1}{2}(1 - t_1 - t_2)\right),\ (t_1, t_2, t_3) \notin \mathcal{I}_3, \right. \\
\nonumber & \quad \frac{1}{35} \leqslant t_4 < \min \left(t_3, \frac{1}{2}(1 - t_1 - t_2 - t_3) \right), \\
\nonumber & \quad (t_1, t_2, t_3, t_4) \notin \mathcal{I}_4,\ (t_1, t_2, t_3, t_4, t_4) \notin \mathcal{J}_5, \\
\nonumber & \quad \text{either } (t_1, t_2, t_3, t_4) \notin \mathcal{J}_4 \text{ or } (1 - t_1 - t_2 - t_3 - t_4, t_2, t_3, t_4) \notin \mathcal{J}_4, \\
\nonumber & \quad t_4 < t_5 < \frac{1}{2}(1 - t_1 - t_2 - t_3 - t_4),\ (t_1, t_2, t_3, t_4, t_5) \in \mathcal{I}_5, \\
\nonumber & \left. \quad \frac{1}{35} \leqslant t_1 < \frac{17}{35},\ \frac{1}{35} \leqslant t_2 < \min\left(t_1, \frac{1}{2}(1-t_1) \right) \right\}, \\
\nonumber U_{\mathcal{M}3}(t_1, t_2, t_3, t_4, t_5, t_6) :=&\ \left\{ (t_1, t_2) \in \mathcal{M},\ \frac{1}{35} \leqslant t_3 < \min\left(t_2, \frac{1}{2}(1 - t_1 - t_2)\right),\ (t_1, t_2, t_3) \notin \mathcal{I}_3, \right. \\
\nonumber & \quad \frac{1}{35} \leqslant t_4 < \min \left(t_3, \frac{1}{2}(1 - t_1 - t_2 - t_3) \right), \\
\nonumber & \quad (t_1, t_2, t_3, t_4) \notin \mathcal{I}_4,\ (t_1, t_2, t_3, t_4, t_4) \in \mathcal{J}_5, \\
\nonumber & \quad \frac{1}{35} \leqslant t_5 < \min \left(t_4, \frac{1}{2}(1 - t_1 - t_2 - t_3 - t_4) \right),\ (t_1, t_2, t_3, t_4, t_5) \notin \mathcal{I}_5, \\
\nonumber & \quad \frac{1}{35} \leqslant t_6 < \min \left(t_5, \frac{1}{2}(1 - t_1 - t_2 - t_3 - t_4 - t_5) \right),\\
\nonumber & \quad (t_1, t_2, t_3, t_4, t_5, t_6) \notin \mathcal{I}_6,\ (t_1, t_2, t_3, t_4, t_5, t_6, t_6) \notin \mathcal{J}_7, \\
\nonumber & \left. \quad \frac{1}{35} \leqslant t_1 < \frac{17}{35},\ \frac{1}{35} \leqslant t_2 < \min\left(t_1, \frac{1}{2}(1-t_1) \right) \right\}, \\
\nonumber U_{\mathcal{M}4}(t_1, t_2, t_3, t_4, t_5, t_6) :=&\ \left\{ (t_1, t_2) \in \mathcal{M},\ \frac{1}{35} \leqslant t_3 < \min\left(t_2, \frac{1}{2}(1 - t_1 - t_2)\right),\ (t_1, t_2, t_3) \notin \mathcal{I}_3, \right. \\
\nonumber & \quad \frac{1}{35} \leqslant t_4 < \min \left(t_3, \frac{1}{2}(1 - t_1 - t_2 - t_3) \right), \\
\nonumber & \quad (t_1, t_2, t_3, t_4) \notin \mathcal{I}_4,\ (t_1, t_2, t_3, t_4, t_4) \notin \mathcal{J}_5, \\
\nonumber & \quad (t_1, t_2, t_3, t_4) \in \mathcal{J}_4,\ (1 - t_1 - t_2 - t_3 - t_4, t_2, t_3, t_4) \in \mathcal{J}_4, \\
\nonumber & \quad \frac{1}{35} \leqslant t_5 < \min \left(t_4, \frac{1}{2}(1 - t_1 - t_2 - t_3 - t_4) \right),\ (t_1, t_2, t_3, t_4, t_5) \notin \mathcal{I}_5, \\
\nonumber & \quad \frac{1}{35} \leqslant t_6 < \frac{1}{2} t_1,\ (1 - t_1 - t_2 - t_3 - t_4 - t_5, t_2, t_3, t_4, t_5, t_6) \notin \mathcal{I}_6, \\
\nonumber & \quad (1 - t_1 - t_2 - t_3 - t_4 - t_5, t_2, t_3, t_4, t_5, t_6, t_6) \notin \mathcal{J}_7, \\
\nonumber & \quad (t_1 - t_6, t_2, t_3, t_4, t_5, t_6, t_5) \notin \mathcal{J}_7, \\
\nonumber & \left. \quad \frac{1}{35} \leqslant t_1 < \frac{17}{35},\ \frac{1}{35} \leqslant t_2 < \min\left(t_1, \frac{1}{2}(1-t_1) \right) \right\}, \\
\nonumber U_{\mathcal{M}5}(t_1, t_2, t_3, t_4, t_5, t_6, t_7, t_8) :=&\ \left\{ (t_1, t_2) \in \mathcal{M},\ \frac{1}{35} \leqslant t_3 < \min\left(t_2, \frac{1}{2}(1 - t_1 - t_2)\right),\ (t_1, t_2, t_3) \notin \mathcal{I}_3, \right. \\
\nonumber & \quad \frac{1}{35} \leqslant t_4 < \min \left(t_3, \frac{1}{2}(1 - t_1 - t_2 - t_3) \right), \\
\nonumber & \quad (t_1, t_2, t_3, t_4) \notin \mathcal{I}_4,\ (t_1, t_2, t_3, t_4, t_4) \in \mathcal{J}_5, \\
\nonumber & \quad \frac{1}{35} \leqslant t_5 < \min \left(t_4, \frac{1}{2}(1 - t_1 - t_2 - t_3 - t_4) \right),\ (t_1, t_2, t_3, t_4, t_5) \notin \mathcal{I}_5, \\
\nonumber & \quad \frac{1}{35} \leqslant t_6 < \min \left(t_5, \frac{1}{2}(1 - t_1 - t_2 - t_3 - t_4 - t_5) \right), \\
\nonumber & \quad (t_1, t_2, t_3, t_4, t_5, t_6) \notin \mathcal{I}_6,\ (t_1, t_2, t_3, t_4, t_5, t_6, t_6) \in \mathcal{J}_7, \\
\nonumber & \quad \frac{1}{35} \leqslant t_7 < \min \left(t_6, \frac{1}{2}(1 - t_1 - t_2 - t_3 - t_4 - t_5 - t_6) \right), \\
\nonumber & \quad (t_1, t_2, t_3, t_4, t_5, t_6, t_7) \notin \mathcal{I}_7, \\
\nonumber & \quad \frac{1}{35} \leqslant t_8 < \min \left(t_7, \frac{1}{2}(1 - t_1 - t_2 - t_3 - t_4 - t_5 - t_6 - t_7) \right), \\
\nonumber & \quad (t_1, t_2, t_3, t_4, t_5, t_6, t_7, t_8) \notin \mathcal{I}_8, \\
\nonumber & \left. \quad \frac{1}{35} \leqslant t_1 < \frac{17}{35},\ \frac{1}{35} \leqslant t_2 < \min\left(t_1, \frac{1}{2}(1-t_1) \right) \right\}, \\
\nonumber U_{\mathcal{M}6}(t_1, t_2, t_3, t_4, t_5, t_6, t_7, t_8) :=&\ \left\{ (t_1, t_2) \in \mathcal{M},\ \frac{1}{35} \leqslant t_3 < \min\left(t_2, \frac{1}{2}(1 - t_1 - t_2)\right),\ (t_1, t_2, t_3) \notin \mathcal{I}_3, \right. \\
\nonumber & \quad \frac{1}{35} \leqslant t_4 < \min \left(t_3, \frac{1}{2}(1 - t_1 - t_2 - t_3) \right), \\
\nonumber & \quad (t_1, t_2, t_3, t_4) \notin \mathcal{I}_4,\ (t_1, t_2, t_3, t_4, t_4) \notin \mathcal{J}_5, \\
\nonumber & \quad (t_1, t_2, t_3, t_4) \in \mathcal{J}_4,\ (1 - t_1 - t_2 - t_3 - t_4, t_2, t_3, t_4) \in \mathcal{J}_4, \\
\nonumber & \quad \frac{1}{35} \leqslant t_5 < \min \left(t_4, \frac{1}{2}(1 - t_1 - t_2 - t_3 - t_4) \right),\ (t_1, t_2, t_3, t_4, t_5) \notin \mathcal{I}_5, \\
\nonumber & \quad \frac{1}{35} \leqslant t_6 < \frac{1}{2} t_1,\ (1 - t_1 - t_2 - t_3 - t_4 - t_5, t_2, t_3, t_4, t_5, t_6) \notin \mathcal{I}_6, \\
\nonumber & \quad (1 - t_1 - t_2 - t_3 - t_4 - t_5, t_2, t_3, t_4, t_5, t_6, t_6) \in \mathcal{J}_7, \\
\nonumber & \quad \frac{1}{35} \leqslant t_7 < \min \left(t_6, \frac{1}{2}(t_1 - t_6) \right), \\
\nonumber & \quad (1 - t_1 - t_2 - t_3 - t_4 - t_5, t_2, t_3, t_4, t_5, t_6, t_7) \notin \mathcal{I}_7, \\
\nonumber & \quad \frac{1}{35} \leqslant t_8 < \min \left(t_7, \frac{1}{2}(t_1 - t_6 - t_7) \right), \\
\nonumber & \quad (1 - t_1 - t_2 - t_3 - t_4 - t_5, t_2, t_3, t_4, t_5, t_6, t_7, t_8) \notin \mathcal{I}_8, \\
\nonumber & \left. \quad \frac{1}{35} \leqslant t_1 < \frac{17}{35},\ \frac{1}{35} \leqslant t_2 < \min\left(t_1, \frac{1}{2}(1-t_1) \right) \right\}, \\
\nonumber U_{\mathcal{M}7}(t_1, t_2, t_3, t_4, t_5, t_6, t_7, t_8) :=&\ \left\{ (t_1, t_2) \in \mathcal{M},\ \frac{1}{35} \leqslant t_3 < \min\left(t_2, \frac{1}{2}(1 - t_1 - t_2)\right),\ (t_1, t_2, t_3) \notin \mathcal{I}_3, \right. \\
\nonumber & \quad \frac{1}{35} \leqslant t_4 < \min \left(t_3, \frac{1}{2}(1 - t_1 - t_2 - t_3) \right), \\
\nonumber & \quad (t_1, t_2, t_3, t_4) \notin \mathcal{I}_4,\ (t_1, t_2, t_3, t_4, t_4) \notin \mathcal{J}_5, \\
\nonumber & \quad (t_1, t_2, t_3, t_4) \in \mathcal{J}_4,\ (1 - t_1 - t_2 - t_3 - t_4, t_2, t_3, t_4) \in \mathcal{J}_4, \\
\nonumber & \quad \frac{1}{35} \leqslant t_5 < \min \left(t_4, \frac{1}{2}(1 - t_1 - t_2 - t_3 - t_4) \right),\ (t_1, t_2, t_3, t_4, t_5) \notin \mathcal{I}_5, \\
\nonumber & \quad \frac{1}{35} \leqslant t_6 < \frac{1}{2} t_1,\ (1 - t_1 - t_2 - t_3 - t_4 - t_5, t_2, t_3, t_4, t_5, t_6) \notin \mathcal{I}_6, \\
\nonumber & \quad (1 - t_1 - t_2 - t_3 - t_4 - t_5, t_2, t_3, t_4, t_5, t_6, t_6) \notin \mathcal{J}_7, \\
\nonumber & \quad (t_1 - t_6, t_2, t_3, t_4, t_5, t_6, t_5) \in \mathcal{J}_7, \\
\nonumber & \quad \frac{1}{35} \leqslant t_7 < \min \left(t_5, \frac{1}{2}(1 - t_1 - t_2 - t_3 - t_4 - t_5) \right), \\
\nonumber & \quad (t_1 - t_6, t_2, t_3, t_4, t_5, t_6, t_7) \notin \mathcal{I}_7, \\
\nonumber & \quad \frac{1}{35} \leqslant t_8 < \min \left(t_7, \frac{1}{2}(1 - t_1 - t_2 - t_3 - t_4 - t_5 - t_7) \right), \\
\nonumber & \quad (t_1 - t_6, t_2, t_3, t_4, t_5, t_6, t_7, t_8) \notin \mathcal{I}_8, \\
\nonumber & \left. \quad \frac{1}{35} \leqslant t_1 < \frac{17}{35},\ \frac{1}{35} \leqslant t_2 < \min\left(t_1, \frac{1}{2}(1-t_1) \right) \right\}.
\end{align}
Note that the seven integrals above correspond to $S_{4353}$ and sums in (22), (12), (19), (11), (17) and (18) respectively.

Next we shall decompose $S_{44}$. By Buchstab's identity, we have
\begin{align}
\nonumber S_{44} =&\ \sum_{\substack{\frac{1}{35} \leqslant t_1 < \frac{17}{35} \\ \frac{1}{35} \leqslant t_2 < \min \left(t_1, \frac{1}{2}(1 - t_1) \right) \\ (t_1, t_2) \in \mathcal{N} }} S\left(\mathcal{A}_{p_1 p_2}, p_2\right) \\
\nonumber =&\ \sum_{\substack{\frac{1}{35} \leqslant t_1 < \frac{17}{35} \\ \frac{1}{35} \leqslant t_2 < \min \left(t_1, \frac{1}{2}(1 - t_1) \right) \\ (t_1, t_2) \in \mathcal{N} }} S\left(\mathcal{A}_{p_1 p_2}, v^{\frac{1}{35}}\right) - \sum_{\substack{\frac{1}{35} \leqslant t_1 < \frac{17}{35} \\ \frac{1}{35} \leqslant t_2 < \min \left(t_1, \frac{1}{2}(1 - t_1) \right) \\ (t_1, t_2) \in \mathcal{N} \\ \frac{1}{35} \leqslant t_3 < \min \left(t_2, \frac{1}{2}(1 - t_1 - t_2) \right) \\  (t_1, t_2, t_3) \in \mathcal{I}_3 }} S\left(\mathcal{A}_{p_1 p_2 p_3}, p_3 \right) \\
\nonumber &- \sum_{\substack{\frac{1}{35} \leqslant t_1 < \frac{17}{35} \\ \frac{1}{35} \leqslant t_2 < \min \left(t_1, \frac{1}{2}(1 - t_1) \right) \\ (t_1, t_2) \in \mathcal{N} \\ \frac{1}{35} \leqslant t_3 < \min \left(t_2, \frac{1}{2}(1 - t_1 - t_2) \right) \\  (t_1, t_2, t_3) \notin \mathcal{I}_3 \\ (t_1, t_2, t_3) \in \mathcal{J}_3 }} S\left(\mathcal{A}_{p_1 p_2 p_3}, p_3 \right) - \sum_{\substack{\frac{1}{35} \leqslant t_1 < \frac{17}{35} \\ \frac{1}{35} \leqslant t_2 < \min \left(t_1, \frac{1}{2}(1 - t_1) \right) \\ (t_1, t_2) \in \mathcal{N} \\ \frac{1}{35} \leqslant t_3 < \min \left(t_2, \frac{1}{2}(1 - t_1 - t_2) \right) \\  (t_1, t_2, t_3) \notin \mathcal{I}_3 \\ (t_1, t_2, t_3) \notin \mathcal{J}_3 }} S\left(\mathcal{A}_{p_1 p_2 p_3}, p_3 \right) \\
=&\ S_{441} - S_{442} - S_{443} - S_{444}.
\end{align}
By Lemma~\ref{l31} and Lemma~\ref{l32}, we have asymptotic formulas for $S_{441}$ and $S_{442}$.

For $S_{443}$, we can use Buchstab's identity again to get
\begin{align}
\nonumber S_{443} =&\ \sum_{\substack{\frac{1}{35} \leqslant t_1 < \frac{17}{35} \\ \frac{1}{35} \leqslant t_2 < \min \left(t_1, \frac{1}{2}(1 - t_1) \right) \\ (t_1, t_2) \in \mathcal{N} \\ \frac{1}{35} \leqslant t_3 < \min \left(t_2, \frac{1}{2}(1 - t_1 - t_2) \right) \\  (t_1, t_2, t_3) \notin \mathcal{I}_3 \\ (t_1, t_2, t_3) \in \mathcal{J}_3 }} S\left(\mathcal{A}_{p_1 p_2 p_3}, p_3 \right) \\
\nonumber =&\ \sum_{\substack{\frac{1}{35} \leqslant t_1 < \frac{17}{35} \\ \frac{1}{35} \leqslant t_2 < \min \left(t_1, \frac{1}{2}(1 - t_1) \right) \\ (t_1, t_2) \in \mathcal{N} \\ \frac{1}{35} \leqslant t_3 < \min \left(t_2, \frac{1}{2}(1 - t_1 - t_2) \right) \\  (t_1, t_2, t_3) \notin \mathcal{I}_3 \\ (t_1, t_2, t_3) \in \mathcal{J}_3 }} S\left(\mathcal{A}_{p_1 p_2 p_3}, v^{\frac{1}{35}} \right) - \sum_{\substack{\frac{1}{35} \leqslant t_1 < \frac{17}{35} \\ \frac{1}{35} \leqslant t_2 < \min \left(t_1, \frac{1}{2}(1 - t_1) \right) \\ (t_1, t_2) \in \mathcal{N} \\ \frac{1}{35} \leqslant t_3 < \min \left(t_2, \frac{1}{2}(1 - t_1 - t_2) \right) \\  (t_1, t_2, t_3) \notin \mathcal{I}_3 \\ (t_1, t_2, t_3) \in \mathcal{J}_3 \\ \frac{1}{35} \leqslant t_4 < \min \left(t_3, \frac{1}{2}(1 - t_1 - t_2 - t_3) \right) \\  (t_1, t_2, t_3, t_4) \in \mathcal{I}_4 }} S\left(\mathcal{A}_{p_1 p_2 p_3 p_4}, p_4 \right) \\
\nonumber & - \sum_{\substack{\frac{1}{35} \leqslant t_1 < \frac{17}{35} \\ \frac{1}{35} \leqslant t_2 < \min \left(t_1, \frac{1}{2}(1 - t_1) \right) \\ (t_1, t_2) \in \mathcal{N} \\ \frac{1}{35} \leqslant t_3 < \min \left(t_2, \frac{1}{2}(1 - t_1 - t_2) \right) \\  (t_1, t_2, t_3) \notin \mathcal{I}_3 \\ (t_1, t_2, t_3) \in \mathcal{J}_3 \\ \frac{1}{35} \leqslant t_4 < \min \left(t_3, \frac{1}{2}(1 - t_1 - t_2 - t_3) \right) \\  (t_1, t_2, t_3, t_4) \notin \mathcal{I}_4 }} S\left(\mathcal{A}_{p_1 p_2 p_3 p_4}, p_4 \right) \\
=&\ S_{4431} - S_{4432} - S_{4433}.
\end{align}
By Lemma~\ref{l31} and Lemma~\ref{l32}, we have asymptotic formulas for $S_{4431}$ and $S_{4432}$. For $S_{4433}$, we can perform the whole process of decomposing $S_{435}$ to find the loss. The resulting integrals are very similar to those in (23), and the only change we need to do is to replace the condition $(t_1, t_2) \in \mathcal{M}$ in those integrals by $(t_1, t_2) \in \mathcal{N}$ and $(t_1, t_2, t_3) \in \mathcal{J}_3$. Thus, the loss from $S_{443}$ can be bounded by
\begin{align}
\nonumber & \left( \int_{(t_1, t_2, t_3, t_4) \in U_{\mathcal{N}01} } \frac{\omega \left(\frac{1 - t_1 - t_2 - t_3 - t_4}{t_4}\right)}{t_1 t_2 t_3 t_4^2} d t_4 d t_3 d t_2 d t_1 \right) \\
\nonumber -& \left( \int_{(t_1, t_2, t_3, t_4, t_5) \in U_{\mathcal{N}02} } \frac{\omega \left(\frac{1 - t_1 - t_2 - t_3 - t_4 - t_5}{t_5}\right)}{t_1 t_2 t_3 t_4 t_5^2} d t_5 d t_4 d t_3 d t_2 d t_1 \right) \\
\nonumber +& \left( \int_{(t_1, t_2, t_3, t_4, t_5, t_6) \in U_{\mathcal{N}03} } \frac{\omega \left(\frac{1 - t_1 - t_2 - t_3 - t_4 - t_5 - t_6}{t_6}\right)}{t_1 t_2 t_3 t_4 t_5 t_6^2} d t_6 d t_5 d t_4 d t_3 d t_2 d t_1 \right) \\
\nonumber +& \left( \int_{(t_1, t_2, t_3, t_4, t_5, t_6) \in U_{\mathcal{N}04} } \frac{\omega \left(\frac{t_1 - t_6}{t_6}\right) \omega \left(\frac{1 - t_1 - t_2 - t_3 - t_4 - t_5}{t_5}\right)}{t_2 t_3 t_4 t_5^2 t_6^2} d t_6 d t_5 d t_4 d t_3 d t_2 d t_1 \right) \\
\nonumber +& \left( \int_{(t_1, t_2, t_3, t_4, t_5, t_6, t_7, t_8) \in U_{\mathcal{N}05} } \frac{\omega \left(\frac{1 - t_1 - t_2 - t_3 - t_4 - t_5 - t_6 - t_7 - t_8}{t_8}\right)}{t_1 t_2 t_3 t_4 t_5 t_6 t_7 t_8^2} d t_8 d t_7 d t_6 d t_5 d t_4 d t_3 d t_2 d t_1 \right) \\
\nonumber +& \left( \int_{(t_1, t_2, t_3, t_4, t_5, t_6, t_7, t_8) \in U_{\mathcal{N}06} } \frac{\omega \left(\frac{t_1 - t_6 - t_7 - t_8}{t_8}\right) \omega \left(\frac{1 - t_1 - t_2 - t_3 - t_4 - t_5}{t_5}\right)}{t_2 t_3 t_4 t_5^2 t_6 t_7 t_8^2} d t_8 d t_7 d t_6 d t_5 d t_4 d t_3 d t_2 d t_1 \right) \\
\nonumber +& \left( \int_{(t_1, t_2, t_3, t_4, t_5, t_6, t_7, t_8) \in U_{\mathcal{N}07} } \frac{\omega \left(\frac{t_1 - t_6}{t_6}\right) \omega \left(\frac{1 - t_1 - t_2 - t_3 - t_4 - t_5 - t_7 - t_8}{t_8}\right)}{t_2 t_3 t_4 t_5 t_6^2 t_7 t_8^2} d t_8 d t_7 d t_6 d t_5 d t_4 d t_3 d t_2 d t_1 \right) \\
\nonumber \leqslant& \left( \int_{(t_1, t_2, t_3, t_4) \in U_{\mathcal{N}01} } \frac{\omega_1 \left(\frac{1 - t_1 - t_2 - t_3 - t_4}{t_4}\right)}{t_1 t_2 t_3 t_4^2} d t_4 d t_3 d t_2 d t_1 \right) \\
\nonumber -& \left( \int_{(t_1, t_2, t_3, t_4, t_5) \in U_{\mathcal{N}02} } \frac{\omega_0 \left(\frac{1 - t_1 - t_2 - t_3 - t_4 - t_5}{t_5}\right)}{t_1 t_2 t_3 t_4 t_5^2} d t_5 d t_4 d t_3 d t_2 d t_1 \right) \\
\nonumber +& \left( \int_{(t_1, t_2, t_3, t_4, t_5, t_6) \in U_{\mathcal{N}03} } \frac{\omega_1 \left(\frac{1 - t_1 - t_2 - t_3 - t_4 - t_5 - t_6}{t_6}\right)}{t_1 t_2 t_3 t_4 t_5 t_6^2} d t_6 d t_5 d t_4 d t_3 d t_2 d t_1 \right) \\
\nonumber +& \left( \int_{(t_1, t_2, t_3, t_4, t_5, t_6) \in U_{\mathcal{N}04} } \frac{\omega_1 \left(\frac{t_1 - t_6}{t_6}\right) \omega_1 \left(\frac{1 - t_1 - t_2 - t_3 - t_4 - t_5}{t_5}\right)}{t_2 t_3 t_4 t_5^2 t_6^2} d t_6 d t_5 d t_4 d t_3 d t_2 d t_1 \right) \\
\nonumber +& \left( \int_{(t_1, t_2, t_3, t_4, t_5, t_6, t_7, t_8) \in U_{\mathcal{N}05} } \frac{\max \left(\frac{t_8}{1 - t_1 - t_2 - t_3 - t_4 - t_5 - t_6 - t_7 - t_8}, 0.5672\right)}{t_1 t_2 t_3 t_4 t_5 t_6 t_7 t_8^2} d t_8 d t_7 d t_6 d t_5 d t_4 d t_3 d t_2 d t_1 \right) \\
\nonumber +& \left( \int_{(t_1, t_2, t_3, t_4, t_5, t_6, t_7, t_8) \in U_{\mathcal{N}06} } \frac{\max \left(\frac{t_8}{t_1 - t_6 - t_7 - t_8}, 0.5672\right) \max \left(\frac{t_5}{1 - t_1 - t_2 - t_3 - t_4 - t_5}, 0.5672\right)}{t_2 t_3 t_4 t_5^2 t_6 t_7 t_8^2} d t_8 d t_7 d t_6 d t_5 d t_4 d t_3 d t_2 d t_1 \right) \\
\nonumber +& \left( \int_{(t_1, t_2, t_3, t_4, t_5, t_6, t_7, t_8) \in U_{\mathcal{N}07} } \frac{\max \left(\frac{t_6}{t_1 - t_6}, 0.5672\right) \max \left(\frac{t_8}{1 - t_1 - t_2 - t_3 - t_4 - t_5 - t_7 - t_8}, 0.5672\right)}{t_2 t_3 t_4 t_5 t_6^2 t_7 t_8^2} d t_8 d t_7 d t_6 d t_5 d t_4 d t_3 d t_2 d t_1 \right) \\
\nonumber \leqslant&\ (0.016780 - 0 + 0.007814 + 0.015516 + 0.000001 + 0.000001 + 0.000001) \\
=&\ 0.040113,
\end{align}
where
\begin{align}
\nonumber U_{\mathcal{N}01}(t_1, t_2, t_3, t_4) :=&\ \left\{ (t_1, t_2) \in \mathcal{N},\ \frac{1}{35} \leqslant t_3 < \min\left(t_2, \frac{1}{2}(1 - t_1 - t_2)\right), \right. \\
\nonumber & \quad (t_1, t_2, t_3) \notin \mathcal{I}_3,\ (t_1, t_2, t_3) \in \mathcal{J}_3, \\
\nonumber & \quad \frac{1}{35} \leqslant t_4 < \min \left(t_3, \frac{1}{2}(1 - t_1 - t_2 - t_3) \right), \\
\nonumber & \quad (t_1, t_2, t_3, t_4) \notin \mathcal{I}_4,\ (t_1, t_2, t_3, t_4, t_4) \notin \mathcal{J}_5, \\
\nonumber & \quad \text{either } (t_1, t_2, t_3, t_4) \notin \mathcal{J}_4 \text{ or } (1 - t_1 - t_2 - t_3 - t_4, t_2, t_3, t_4) \notin \mathcal{J}_4, \\
\nonumber & \left. \quad \frac{1}{35} \leqslant t_1 < \frac{17}{35},\ \frac{1}{35} \leqslant t_2 < \min\left(t_1, \frac{1}{2}(1-t_1) \right) \right\}, \\
\nonumber U_{\mathcal{N}02}(t_1, t_2, t_3, t_4, t_5) :=&\ \left\{ (t_1, t_2) \in \mathcal{N},\ \frac{1}{35} \leqslant t_3 < \min\left(t_2, \frac{1}{2}(1 - t_1 - t_2)\right), \right. \\
\nonumber & \quad (t_1, t_2, t_3) \notin \mathcal{I}_3,\ (t_1, t_2, t_3) \in \mathcal{J}_3, \\
\nonumber & \quad \frac{1}{35} \leqslant t_4 < \min \left(t_3, \frac{1}{2}(1 - t_1 - t_2 - t_3) \right), \\
\nonumber & \quad (t_1, t_2, t_3, t_4) \notin \mathcal{I}_4,\ (t_1, t_2, t_3, t_4, t_4) \notin \mathcal{J}_5, \\
\nonumber & \quad \text{either } (t_1, t_2, t_3, t_4) \notin \mathcal{J}_4 \text{ or } (1 - t_1 - t_2 - t_3 - t_4, t_2, t_3, t_4) \notin \mathcal{J}_4, \\
\nonumber & \quad t_4 < t_5 < \frac{1}{2}(1 - t_1 - t_2 - t_3 - t_4),\ (t_1, t_2, t_3, t_4, t_5) \in \mathcal{I}_5, \\
\nonumber & \left. \quad \frac{1}{35} \leqslant t_1 < \frac{17}{35},\ \frac{1}{35} \leqslant t_2 < \min\left(t_1, \frac{1}{2}(1-t_1) \right) \right\}, \\
\nonumber U_{\mathcal{N}03}(t_1, t_2, t_3, t_4, t_5, t_6) :=&\ \left\{ (t_1, t_2) \in \mathcal{N},\ \frac{1}{35} \leqslant t_3 < \min\left(t_2, \frac{1}{2}(1 - t_1 - t_2)\right), \right. \\
\nonumber & \quad (t_1, t_2, t_3) \notin \mathcal{I}_3,\ (t_1, t_2, t_3) \in \mathcal{J}_3, \\
\nonumber & \quad \frac{1}{35} \leqslant t_4 < \min \left(t_3, \frac{1}{2}(1 - t_1 - t_2 - t_3) \right), \\
\nonumber & \quad (t_1, t_2, t_3, t_4) \notin \mathcal{I}_4,\ (t_1, t_2, t_3, t_4, t_4) \in \mathcal{J}_5, \\
\nonumber & \quad \frac{1}{35} \leqslant t_5 < \min \left(t_4, \frac{1}{2}(1 - t_1 - t_2 - t_3 - t_4) \right),\ (t_1, t_2, t_3, t_4, t_5) \notin \mathcal{I}_5, \\
\nonumber & \quad \frac{1}{35} \leqslant t_6 < \min \left(t_5, \frac{1}{2}(1 - t_1 - t_2 - t_3 - t_4 - t_5) \right),\\
\nonumber & \quad (t_1, t_2, t_3, t_4, t_5, t_6) \notin \mathcal{I}_6,\ (t_1, t_2, t_3, t_4, t_5, t_6, t_6) \notin \mathcal{J}_7, \\
\nonumber & \left. \quad \frac{1}{35} \leqslant t_1 < \frac{17}{35},\ \frac{1}{35} \leqslant t_2 < \min\left(t_1, \frac{1}{2}(1-t_1) \right) \right\}, \\
\nonumber U_{\mathcal{N}04}(t_1, t_2, t_3, t_4, t_5, t_6) :=&\ \left\{ (t_1, t_2) \in \mathcal{N},\ \frac{1}{35} \leqslant t_3 < \min\left(t_2, \frac{1}{2}(1 - t_1 - t_2)\right), \right. \\
\nonumber & \quad (t_1, t_2, t_3) \notin \mathcal{I}_3,\ (t_1, t_2, t_3) \in \mathcal{J}_3, \\
\nonumber & \quad \frac{1}{35} \leqslant t_4 < \min \left(t_3, \frac{1}{2}(1 - t_1 - t_2 - t_3) \right), \\
\nonumber & \quad (t_1, t_2, t_3, t_4) \notin \mathcal{I}_4,\ (t_1, t_2, t_3, t_4, t_4) \notin \mathcal{J}_5, \\
\nonumber & \quad (t_1, t_2, t_3, t_4) \in \mathcal{J}_4,\ (1 - t_1 - t_2 - t_3 - t_4, t_2, t_3, t_4) \in \mathcal{J}_4, \\
\nonumber & \quad \frac{1}{35} \leqslant t_5 < \min \left(t_4, \frac{1}{2}(1 - t_1 - t_2 - t_3 - t_4) \right),\ (t_1, t_2, t_3, t_4, t_5) \notin \mathcal{I}_5, \\
\nonumber & \quad \frac{1}{35} \leqslant t_6 < \frac{1}{2} t_1,\ (1 - t_1 - t_2 - t_3 - t_4 - t_5, t_2, t_3, t_4, t_5, t_6) \notin \mathcal{I}_6, \\
\nonumber & \quad (1 - t_1 - t_2 - t_3 - t_4 - t_5, t_2, t_3, t_4, t_5, t_6, t_6) \notin \mathcal{J}_7, \\
\nonumber & \quad (t_1 - t_6, t_2, t_3, t_4, t_5, t_6, t_5) \notin \mathcal{J}_7, \\
\nonumber & \left. \quad \frac{1}{35} \leqslant t_1 < \frac{17}{35},\ \frac{1}{35} \leqslant t_2 < \min\left(t_1, \frac{1}{2}(1-t_1) \right) \right\}, \\
\nonumber U_{\mathcal{N}05}(t_1, t_2, t_3, t_4, t_5, t_6, t_7, t_8) :=&\ \left\{ (t_1, t_2) \in \mathcal{N},\ \frac{1}{35} \leqslant t_3 < \min\left(t_2, \frac{1}{2}(1 - t_1 - t_2)\right), \right. \\
\nonumber & \quad (t_1, t_2, t_3) \notin \mathcal{I}_3,\ (t_1, t_2, t_3) \in \mathcal{J}_3, \\
\nonumber & \quad \frac{1}{35} \leqslant t_4 < \min \left(t_3, \frac{1}{2}(1 - t_1 - t_2 - t_3) \right), \\
\nonumber & \quad (t_1, t_2, t_3, t_4) \notin \mathcal{I}_4,\ (t_1, t_2, t_3, t_4, t_4) \in \mathcal{J}_5, \\
\nonumber & \quad \frac{1}{35} \leqslant t_5 < \min \left(t_4, \frac{1}{2}(1 - t_1 - t_2 - t_3 - t_4) \right),\ (t_1, t_2, t_3, t_4, t_5) \notin \mathcal{I}_5, \\
\nonumber & \quad \frac{1}{35} \leqslant t_6 < \min \left(t_5, \frac{1}{2}(1 - t_1 - t_2 - t_3 - t_4 - t_5) \right), \\
\nonumber & \quad (t_1, t_2, t_3, t_4, t_5, t_6) \notin \mathcal{I}_6,\ (t_1, t_2, t_3, t_4, t_5, t_6, t_6) \in \mathcal{J}_7, \\
\nonumber & \quad \frac{1}{35} \leqslant t_7 < \min \left(t_6, \frac{1}{2}(1 - t_1 - t_2 - t_3 - t_4 - t_5 - t_6) \right), \\
\nonumber & \quad (t_1, t_2, t_3, t_4, t_5, t_6, t_7) \notin \mathcal{I}_7, \\
\nonumber & \quad \frac{1}{35} \leqslant t_8 < \min \left(t_7, \frac{1}{2}(1 - t_1 - t_2 - t_3 - t_4 - t_5 - t_6 - t_7) \right), \\
\nonumber & \quad (t_1, t_2, t_3, t_4, t_5, t_6, t_7, t_8) \notin \mathcal{I}_8, \\
\nonumber & \left. \quad \frac{1}{35} \leqslant t_1 < \frac{17}{35},\ \frac{1}{35} \leqslant t_2 < \min\left(t_1, \frac{1}{2}(1-t_1) \right) \right\}, \\
\nonumber U_{\mathcal{N}06}(t_1, t_2, t_3, t_4, t_5, t_6, t_7, t_8) :=&\ \left\{ (t_1, t_2) \in \mathcal{N},\ \frac{1}{35} \leqslant t_3 < \min\left(t_2, \frac{1}{2}(1 - t_1 - t_2)\right), \right. \\
\nonumber & \quad (t_1, t_2, t_3) \notin \mathcal{I}_3,\ (t_1, t_2, t_3) \in \mathcal{J}_3, \\
\nonumber & \quad \frac{1}{35} \leqslant t_4 < \min \left(t_3, \frac{1}{2}(1 - t_1 - t_2 - t_3) \right), \\
\nonumber & \quad (t_1, t_2, t_3, t_4) \notin \mathcal{I}_4,\ (t_1, t_2, t_3, t_4, t_4) \notin \mathcal{J}_5, \\
\nonumber & \quad (t_1, t_2, t_3, t_4) \in \mathcal{J}_4,\ (1 - t_1 - t_2 - t_3 - t_4, t_2, t_3, t_4) \in \mathcal{J}_4, \\
\nonumber & \quad \frac{1}{35} \leqslant t_5 < \min \left(t_4, \frac{1}{2}(1 - t_1 - t_2 - t_3 - t_4) \right),\ (t_1, t_2, t_3, t_4, t_5) \notin \mathcal{I}_5, \\
\nonumber & \quad \frac{1}{35} \leqslant t_6 < \frac{1}{2} t_1,\ (1 - t_1 - t_2 - t_3 - t_4 - t_5, t_2, t_3, t_4, t_5, t_6) \notin \mathcal{I}_6, \\
\nonumber & \quad (1 - t_1 - t_2 - t_3 - t_4 - t_5, t_2, t_3, t_4, t_5, t_6, t_6) \in \mathcal{J}_7, \\
\nonumber & \quad \frac{1}{35} \leqslant t_7 < \min \left(t_6, \frac{1}{2}(t_1 - t_6) \right), \\
\nonumber & \quad (1 - t_1 - t_2 - t_3 - t_4 - t_5, t_2, t_3, t_4, t_5, t_6, t_7) \notin \mathcal{I}_7, \\
\nonumber & \quad \frac{1}{35} \leqslant t_8 < \min \left(t_7, \frac{1}{2}(t_1 - t_6 - t_7) \right), \\
\nonumber & \quad (1 - t_1 - t_2 - t_3 - t_4 - t_5, t_2, t_3, t_4, t_5, t_6, t_7, t_8) \notin \mathcal{I}_8, \\
\nonumber & \left. \quad \frac{1}{35} \leqslant t_1 < \frac{17}{35},\ \frac{1}{35} \leqslant t_2 < \min\left(t_1, \frac{1}{2}(1-t_1) \right) \right\}, \\
\nonumber U_{\mathcal{N}07}(t_1, t_2, t_3, t_4, t_5, t_6, t_7, t_8) :=&\ \left\{ (t_1, t_2) \in \mathcal{N},\ \frac{1}{35} \leqslant t_3 < \min\left(t_2, \frac{1}{2}(1 - t_1 - t_2)\right), \right. \\
\nonumber & \quad (t_1, t_2, t_3) \notin \mathcal{I}_3,\ (t_1, t_2, t_3) \in \mathcal{J}_3, \\
\nonumber & \quad \frac{1}{35} \leqslant t_4 < \min \left(t_3, \frac{1}{2}(1 - t_1 - t_2 - t_3) \right), \\
\nonumber & \quad (t_1, t_2, t_3, t_4) \notin \mathcal{I}_4,\ (t_1, t_2, t_3, t_4, t_4) \notin \mathcal{J}_5, \\
\nonumber & \quad (t_1, t_2, t_3, t_4) \in \mathcal{J}_4,\ (1 - t_1 - t_2 - t_3 - t_4, t_2, t_3, t_4) \in \mathcal{J}_4, \\
\nonumber & \quad \frac{1}{35} \leqslant t_5 < \min \left(t_4, \frac{1}{2}(1 - t_1 - t_2 - t_3 - t_4) \right),\ (t_1, t_2, t_3, t_4, t_5) \notin \mathcal{I}_5, \\
\nonumber & \quad \frac{1}{35} \leqslant t_6 < \frac{1}{2} t_1,\ (1 - t_1 - t_2 - t_3 - t_4 - t_5, t_2, t_3, t_4, t_5, t_6) \notin \mathcal{I}_6, \\
\nonumber & \quad (1 - t_1 - t_2 - t_3 - t_4 - t_5, t_2, t_3, t_4, t_5, t_6, t_6) \notin \mathcal{J}_7, \\
\nonumber & \quad (t_1 - t_6, t_2, t_3, t_4, t_5, t_6, t_5) \in \mathcal{J}_7, \\
\nonumber & \quad \frac{1}{35} \leqslant t_7 < \min \left(t_5, \frac{1}{2}(1 - t_1 - t_2 - t_3 - t_4 - t_5) \right), \\
\nonumber & \quad (t_1 - t_6, t_2, t_3, t_4, t_5, t_6, t_7) \notin \mathcal{I}_7, \\
\nonumber & \quad \frac{1}{35} \leqslant t_8 < \min \left(t_7, \frac{1}{2}(1 - t_1 - t_2 - t_3 - t_4 - t_5 - t_7) \right), \\
\nonumber & \quad (t_1 - t_6, t_2, t_3, t_4, t_5, t_6, t_7, t_8) \notin \mathcal{I}_8, \\
\nonumber & \left. \quad \frac{1}{35} \leqslant t_1 < \frac{17}{35},\ \frac{1}{35} \leqslant t_2 < \min\left(t_1, \frac{1}{2}(1-t_1) \right) \right\}.
\end{align}

For $S_{444}$ we cannot perform Buchstab iteration directly as in (25) since we have $(t_1, t_2, t_3) \notin \mathcal{J}_3$ in this sum. However, we can use role-reversals together with Buchstab's identity to decompose this sum further. Note that $S_{444}$ counts numbers of the form $p_1 p_2 p_3 \beta_3$, where $\beta_3 \sim v^{1 - t_1 - t_2 - t_3}$ and $\left(\beta_3, P(p_3)\right) = 1$. By reversing the roles of $p_1$ and $\beta_3$, together with Buchstab's identity, we have
\begin{align}
\nonumber S_{444} =&\ \sum_{\substack{\frac{1}{35} \leqslant t_1 < \frac{17}{35} \\ \frac{1}{35} \leqslant t_2 < \min \left(t_1, \frac{1}{2}(1 - t_1) \right) \\ (t_1, t_2) \in \mathcal{N} \\ \frac{1}{35} \leqslant t_3 < \min \left(t_2, \frac{1}{2}(1 - t_1 - t_2) \right) \\  (t_1, t_2, t_3) \notin \mathcal{I}_3 \\ (t_1, t_2, t_3) \notin \mathcal{J}_3 }} S\left(\mathcal{A}_{p_1 p_2 p_3}, p_3 \right) \\
\nonumber =&\ \sum_{\substack{\frac{1}{35} \leqslant t_1 < \frac{17}{35} \\ \frac{1}{35} \leqslant t_2 < \min \left(t_1, \frac{1}{2}(1 - t_1) \right) \\ (t_1, t_2) \in \mathcal{N} \\ \frac{1}{35} \leqslant t_3 < \min \left(t_2, \frac{1}{2}(1 - t_1 - t_2) \right) \\  (t_1, t_2, t_3) \notin \mathcal{I}_3 \\ (t_1, t_2, t_3) \notin \mathcal{J}_3 }} S\left(\mathcal{A}_{\beta_3 p_2 p_3}, \left(\frac{2v}{\beta_3 p_2 p_3}\right)^{\frac{1}{2}} \right) \\
\nonumber =&\ \sum_{\substack{\frac{1}{35} \leqslant t_1 < \frac{17}{35} \\ \frac{1}{35} \leqslant t_2 < \min \left(t_1, \frac{1}{2}(1 - t_1) \right) \\ (t_1, t_2) \in \mathcal{N} \\ \frac{1}{35} \leqslant t_3 < \min \left(t_2, \frac{1}{2}(1 - t_1 - t_2) \right) \\  (t_1, t_2, t_3) \notin \mathcal{I}_3 \\ (t_1, t_2, t_3) \notin \mathcal{J}_3 }} S\left(\mathcal{A}_{\beta_3 p_2 p_3}, v^{\frac{1}{35}} \right) - \sum_{\substack{\frac{1}{35} \leqslant t_1 < \frac{17}{35} \\ \frac{1}{35} \leqslant t_2 < \min \left(t_1, \frac{1}{2}(1 - t_1) \right) \\ (t_1, t_2) \in \mathcal{N} \\ \frac{1}{35} \leqslant t_3 < \min \left(t_2, \frac{1}{2}(1 - t_1 - t_2) \right) \\  (t_1, t_2, t_3) \notin \mathcal{I}_3 \\ (t_1, t_2, t_3) \notin \mathcal{J}_3 \\ \frac{1}{35} \leqslant t_4 < \frac{1}{2} t_1 \\  (1 - t_1 - t_2 - t_3, t_2, t_3, t_4) \in \mathcal{I}_4 }} S\left(\mathcal{A}_{\beta_3 p_2 p_3 p_4}, p_4 \right) \\
\nonumber &- \sum_{\substack{\frac{1}{35} \leqslant t_1 < \frac{17}{35} \\ \frac{1}{35} \leqslant t_2 < \min \left(t_1, \frac{1}{2}(1 - t_1) \right) \\ (t_1, t_2) \in \mathcal{N} \\ \frac{1}{35} \leqslant t_3 < \min \left(t_2, \frac{1}{2}(1 - t_1 - t_2) \right) \\  (t_1, t_2, t_3) \notin \mathcal{I}_3 \\ (t_1, t_2, t_3) \notin \mathcal{J}_3 \\ \frac{1}{35} \leqslant t_4 < \frac{1}{2} t_1 \\  (1 - t_1 - t_2 - t_3, t_2, t_3, t_4) \notin \mathcal{I}_4 }} S\left(\mathcal{A}_{\beta_3 p_2 p_3 p_4}, p_4 \right) \\
=&\ S_{4441} - S_{4442} - S_{4443}.
\end{align}
We have an asymptotic formula for $S_{4442}$ by Lemma~\ref{l32}. Since we have $t_2 < \frac{53}{255} < \frac{9}{35}$ in region $\mathcal{N}$, we must have $t_1 + t_2 > \frac{18}{35}$. Otherwise we have $(t_1, t_2) \in \mathcal{M}$ by Lemma~\ref{l22} and the definition of $\mathcal{M}$ with $m = t_1$ and $n = t_2$. Now we have $1 - t_1 - t_2 < \frac{17}{35}$ for $(t_1, t_2) \in \mathcal{N}$. Since $\beta_3 p_3 \sim x^{1 - t_1 - t_2}$, we have $(1 - t_1 - t_2 - t_3, t_2, t_3) \in \mathcal{J}_3$. By Lemma~\ref{l31}, we have an asymptotic formula for $S_{4441}$.

For the remaining $S_{4443}$, we first split it into four parts:
\begin{align}
\nonumber S_{4443} =&\ \sum_{\substack{\frac{1}{35} \leqslant t_1 < \frac{17}{35} \\ \frac{1}{35} \leqslant t_2 < \min \left(t_1, \frac{1}{2}(1 - t_1) \right) \\ (t_1, t_2) \in \mathcal{N} \\ \frac{1}{35} \leqslant t_3 < \min \left(t_2, \frac{1}{2}(1 - t_1 - t_2) \right) \\  (t_1, t_2, t_3) \notin \mathcal{I}_3 \\ (t_1, t_2, t_3) \notin \mathcal{J}_3 \\ \frac{1}{35} \leqslant t_4 < \frac{1}{2} t_1 \\  (1 - t_1 - t_2 - t_3, t_2, t_3, t_4) \notin \mathcal{I}_4 }} S\left(\mathcal{A}_{\beta_3 p_2 p_3 p_4}, p_4 \right) \\
\nonumber =&\ \sum_{\substack{\frac{1}{35} \leqslant t_1 < \frac{17}{35} \\ \frac{1}{35} \leqslant t_2 < \min \left(t_1, \frac{1}{2}(1 - t_1) \right) \\ (t_1, t_2) \in \mathcal{N} \\ \frac{1}{35} \leqslant t_3 < \min \left(t_2, \frac{1}{2}(1 - t_1 - t_2) \right) \\  (t_1, t_2, t_3) \notin \mathcal{I}_3 \\ (t_1, t_2, t_3) \notin \mathcal{J}_3 \\ \frac{1}{35} \leqslant t_4 < \frac{1}{2} t_1 \\  (1 - t_1 - t_2 - t_3, t_2, t_3, t_4) \notin \mathcal{I}_4 \\  (1 - t_1 - t_2 - t_3, t_2, t_3, t_4, t_4) \in \mathcal{J}_5 }} S\left(\mathcal{A}_{\beta_3 p_2 p_3 p_4}, p_4 \right) + \sum_{\substack{\frac{1}{35} \leqslant t_1 < \frac{17}{35} \\ \frac{1}{35} \leqslant t_2 < \min \left(t_1, \frac{1}{2}(1 - t_1) \right) \\ (t_1, t_2) \in \mathcal{N} \\ \frac{1}{35} \leqslant t_3 < \min \left(t_2, \frac{1}{2}(1 - t_1 - t_2) \right) \\  (t_1, t_2, t_3) \notin \mathcal{I}_3 \\ (t_1, t_2, t_3) \notin \mathcal{J}_3 \\ \frac{1}{35} \leqslant t_4 < \frac{1}{2} t_1 \\  (1 - t_1 - t_2 - t_3, t_2, t_3, t_4) \notin \mathcal{I}_4 \\  (1 - t_1 - t_2 - t_3, t_2, t_3, t_4, t_4) \notin \mathcal{J}_5 \\  (t_1 - t_4, t_2, t_3, t_4, t_3) \in \mathcal{J}_5 }} S\left(\mathcal{A}_{\beta_3 p_2 p_3 p_4}, p_4 \right) \\
\nonumber &+ \sum_{\substack{\frac{1}{35} \leqslant t_1 < \frac{17}{35} \\ \frac{1}{35} \leqslant t_2 < \min \left(t_1, \frac{1}{2}(1 - t_1) \right) \\ (t_1, t_2) \in \mathcal{N} \\ \frac{1}{35} \leqslant t_3 < \min \left(t_2, \frac{1}{2}(1 - t_1 - t_2) \right) \\  (t_1, t_2, t_3) \notin \mathcal{I}_3 \\ (t_1, t_2, t_3) \notin \mathcal{J}_3 \\ \frac{1}{35} \leqslant t_4 < \frac{1}{2} t_1 \\  (1 - t_1 - t_2 - t_3, t_2, t_3, t_4) \notin \mathcal{I}_4 \\  (1 - t_1 - t_2 - t_3, t_2, t_3, t_4, t_4) \notin \mathcal{J}_5 \\  (t_1 - t_4, t_2, t_3, t_4, t_3) \notin \mathcal{J}_5 \\ (1 - t_1 - t_2 - t_3, t_2, t_3, t_4) \in \mathcal{J}_4 \text{ and } (1 - t_1 - t_2 - t_3, t_1 - t_4, t_3, \min(t_2, t_4) ) \in \mathcal{J}_4 }} S\left(\mathcal{A}_{\beta_3 p_2 p_3 p_4}, p_4 \right) \\
\nonumber &+ \sum_{\substack{\frac{1}{35} \leqslant t_1 < \frac{17}{35} \\ \frac{1}{35} \leqslant t_2 < \min \left(t_1, \frac{1}{2}(1 - t_1) \right) \\ (t_1, t_2) \in \mathcal{N} \\ \frac{1}{35} \leqslant t_3 < \min \left(t_2, \frac{1}{2}(1 - t_1 - t_2) \right) \\  (t_1, t_2, t_3) \notin \mathcal{I}_3 \\ (t_1, t_2, t_3) \notin \mathcal{J}_3 \\ \frac{1}{35} \leqslant t_4 < \frac{1}{2} t_1 \\  (1 - t_1 - t_2 - t_3, t_2, t_3, t_4) \notin \mathcal{I}_4 \\  (1 - t_1 - t_2 - t_3, t_2, t_3, t_4, t_4) \notin \mathcal{J}_5 \\  (t_1 - t_4, t_2, t_3, t_4, t_3) \notin \mathcal{J}_5 \\ \text{either } (1 - t_1 - t_2 - t_3, t_2, t_3, t_4) \notin \mathcal{J}_4 \text{ or } (1 - t_1 - t_2 - t_3, t_1 - t_4, t_3, \min(t_2, t_4) ) \notin \mathcal{J}_4}} S\left(\mathcal{A}_{\beta_3 p_2 p_3 p_4}, p_4 \right) \\
=&\ S_{44431} + S_{44432} + S_{44433} + S_{44434}.
\end{align}

Similar to the sum $S_{435}$, we can decompose each part of $S_{4443}$ using different ways. For $S_{44431}$ we can apply Buchstab's identity twice directly, and the corresponding six-dimensional sum that will be discarded is
\begin{equation}
\sum_{\substack{\frac{1}{35} \leqslant t_1 < \frac{17}{35} \\ \frac{1}{35} \leqslant t_2 < \min \left(t_1, \frac{1}{2}(1 - t_1) \right) \\ (t_1, t_2) \in \mathcal{N} \\ \frac{1}{35} \leqslant t_3 < \min \left(t_2, \frac{1}{2}(1 - t_1 - t_2) \right) \\  (t_1, t_2, t_3) \notin \mathcal{I}_3 \\ (t_1, t_2, t_3) \notin \mathcal{J}_3 \\ \frac{1}{35} \leqslant t_4 < \frac{1}{2} t_1 \\  (1 - t_1 - t_2 - t_3, t_2, t_3, t_4) \notin \mathcal{I}_4 \\  (1 - t_1 - t_2 - t_3, t_2, t_3, t_4, t_4) \in \mathcal{J}_5 \\ \frac{1}{35} \leqslant t_5 < \min \left(t_4, \frac{1}{2}(t_1 - t_4) \right) \\  (1 - t_1 - t_2 - t_3, t_2, t_3, t_4, t_5) \notin \mathcal{I}_5 \\ \frac{1}{35} \leqslant t_6 < \min \left(t_5, \frac{1}{2}(t_1 - t_4 - t_5) \right) \\  (1 - t_1 - t_2 - t_3, t_2, t_3, t_4, t_5, t_6) \notin \mathcal{I}_6 }} S\left(\mathcal{A}_{\beta_3 p_2 p_3 p_4 p_5 p_6}, p_6 \right).
\end{equation}

Note that $S_{4443}$ counts numbers of the form $\beta_3 p_2 p_3 p_4 \beta_4$, where $\beta_4 \sim v^{t_1 - t_4}$ and $\left(\beta_4, P(p_4)\right) = 1$. For $S_{44432}$ we can rewrite it as
\begin{equation}
\sum_{\substack{\frac{1}{35} \leqslant t_1 < \frac{17}{35} \\ \frac{1}{35} \leqslant t_2 < \min \left(t_1, \frac{1}{2}(1 - t_1) \right) \\ (t_1, t_2) \in \mathcal{N} \\ \frac{1}{35} \leqslant t_3 < \min \left(t_2, \frac{1}{2}(1 - t_1 - t_2) \right) \\  (t_1, t_2, t_3) \notin \mathcal{I}_3 \\ (t_1, t_2, t_3) \notin \mathcal{J}_3 \\ \frac{1}{35} \leqslant t_4 < \frac{1}{2} t_1 \\  (1 - t_1 - t_2 - t_3, t_2, t_3, t_4) \notin \mathcal{I}_4 \\  (1 - t_1 - t_2 - t_3, t_2, t_3, t_4, t_4) \notin \mathcal{J}_5 \\  (t_1 - t_4, t_2, t_3, t_4, t_3) \in \mathcal{J}_5 }} S\left(\mathcal{A}_{\beta_3 p_2 p_3 p_4}, p_4 \right) = \sum_{\substack{\frac{1}{35} \leqslant t_1 < \frac{17}{35} \\ \frac{1}{35} \leqslant t_2 < \min \left(t_1, \frac{1}{2}(1 - t_1) \right) \\ (t_1, t_2) \in \mathcal{N} \\ \frac{1}{35} \leqslant t_3 < \min \left(t_2, \frac{1}{2}(1 - t_1 - t_2) \right) \\  (t_1, t_2, t_3) \notin \mathcal{I}_3 \\ (t_1, t_2, t_3) \notin \mathcal{J}_3 \\ \frac{1}{35} \leqslant t_4 < \frac{1}{2} t_1 \\  (1 - t_1 - t_2 - t_3, t_2, t_3, t_4) \notin \mathcal{I}_4 \\  (1 - t_1 - t_2 - t_3, t_2, t_3, t_4, t_4) \notin \mathcal{J}_5 \\  (t_1 - t_4, t_2, t_3, t_4, t_3) \in \mathcal{J}_5 }} S\left(\mathcal{A}_{\beta_4 p_2 p_3 p_4}, p_3 \right)
\end{equation}
and use Buchstab's identity twice on $\beta_3$. The corresponding six-dimensional sum that will be discarded is
\begin{equation}
\sum_{\substack{\frac{1}{35} \leqslant t_1 < \frac{17}{35} \\ \frac{1}{35} \leqslant t_2 < \min \left(t_1, \frac{1}{2}(1 - t_1) \right) \\ (t_1, t_2) \in \mathcal{N} \\ \frac{1}{35} \leqslant t_3 < \min \left(t_2, \frac{1}{2}(1 - t_1 - t_2) \right) \\  (t_1, t_2, t_3) \notin \mathcal{I}_3 \\ (t_1, t_2, t_3) \notin \mathcal{J}_3 \\ \frac{1}{35} \leqslant t_4 < \frac{1}{2} t_1 \\  (1 - t_1 - t_2 - t_3, t_2, t_3, t_4) \notin \mathcal{I}_4 \\  (1 - t_1 - t_2 - t_3, t_2, t_3, t_4, t_4) \notin \mathcal{J}_5 \\  (t_1 - t_4, t_2, t_3, t_4, t_3) \in \mathcal{J}_5 \\ \frac{1}{35} \leqslant t_5 < \min \left(t_3, \frac{1}{2}(1 - t_1 - t_2 - t_3) \right) \\ (t_1 - t_4, t_2, t_3, t_4, t_5) \notin \mathcal{I}_5 \\ \frac{1}{35} \leqslant t_6 < \min \left(t_5, \frac{1}{2}(1 - t_1 - t_2 - t_3 - t_5) \right) \\ (t_1 - t_4, t_2, t_3, t_4, t_5, t_6) \notin \mathcal{I}_6 }} S\left(\mathcal{A}_{\beta_4 p_2 p_3 p_4 p_5 p_6}, p_6 \right).
\end{equation}

For $S_{44433}$ we use Buchstab's identity once:
\begin{align}
\nonumber S_{44433} =&\ \sum_{\substack{\frac{1}{35} \leqslant t_1 < \frac{17}{35} \\ \frac{1}{35} \leqslant t_2 < \min \left(t_1, \frac{1}{2}(1 - t_1) \right) \\ (t_1, t_2) \in \mathcal{N} \\ \frac{1}{35} \leqslant t_3 < \min \left(t_2, \frac{1}{2}(1 - t_1 - t_2) \right) \\  (t_1, t_2, t_3) \notin \mathcal{I}_3 \\ (t_1, t_2, t_3) \notin \mathcal{J}_3 \\ \frac{1}{35} \leqslant t_4 < \frac{1}{2} t_1 \\  (1 - t_1 - t_2 - t_3, t_2, t_3, t_4) \notin \mathcal{I}_4 \\  (1 - t_1 - t_2 - t_3, t_2, t_3, t_4, t_4) \notin \mathcal{J}_5 \\  (t_1 - t_4, t_2, t_3, t_4, t_3) \notin \mathcal{J}_5 \\ (1 - t_1 - t_2 - t_3, t_2, t_3, t_4) \in \mathcal{J}_4 \text{ and } (1 - t_1 - t_2 - t_3, t_1 - t_4, t_3, \min(t_2, t_4) ) \in \mathcal{J}_4 }} S\left(\mathcal{A}_{\beta_3 p_2 p_3 p_4}, p_4 \right) \\
\nonumber =&\ \sum_{\substack{\frac{1}{35} \leqslant t_1 < \frac{17}{35} \\ \frac{1}{35} \leqslant t_2 < \min \left(t_1, \frac{1}{2}(1 - t_1) \right) \\ (t_1, t_2) \in \mathcal{N} \\ \frac{1}{35} \leqslant t_3 < \min \left(t_2, \frac{1}{2}(1 - t_1 - t_2) \right) \\  (t_1, t_2, t_3) \notin \mathcal{I}_3 \\ (t_1, t_2, t_3) \notin \mathcal{J}_3 \\ \frac{1}{35} \leqslant t_4 < \frac{1}{2} t_1 \\  (1 - t_1 - t_2 - t_3, t_2, t_3, t_4) \notin \mathcal{I}_4 \\  (1 - t_1 - t_2 - t_3, t_2, t_3, t_4, t_4) \notin \mathcal{J}_5 \\  (t_1 - t_4, t_2, t_3, t_4, t_3) \notin \mathcal{J}_5 \\ (1 - t_1 - t_2 - t_3, t_2, t_3, t_4) \in \mathcal{J}_4 \text{ and } (1 - t_1 - t_2 - t_3, t_1 - t_4, t_3, \min(t_2, t_4) ) \in \mathcal{J}_4 }} S\left(\mathcal{A}_{\beta_3 p_2 p_3 p_4}, v^{\frac{1}{35}} \right) \\
\nonumber &- \sum_{\substack{\frac{1}{35} \leqslant t_1 < \frac{17}{35} \\ \frac{1}{35} \leqslant t_2 < \min \left(t_1, \frac{1}{2}(1 - t_1) \right) \\ (t_1, t_2) \in \mathcal{N} \\ \frac{1}{35} \leqslant t_3 < \min \left(t_2, \frac{1}{2}(1 - t_1 - t_2) \right) \\  (t_1, t_2, t_3) \notin \mathcal{I}_3 \\ (t_1, t_2, t_3) \notin \mathcal{J}_3 \\ \frac{1}{35} \leqslant t_4 < \frac{1}{2} t_1 \\  (1 - t_1 - t_2 - t_3, t_2, t_3, t_4) \notin \mathcal{I}_4 \\  (1 - t_1 - t_2 - t_3, t_2, t_3, t_4, t_4) \notin \mathcal{J}_5 \\  (t_1 - t_4, t_2, t_3, t_4, t_3) \notin \mathcal{J}_5 \\ (1 - t_1 - t_2 - t_3, t_2, t_3, t_4) \in \mathcal{J}_4 \text{ and } (1 - t_1 - t_2 - t_3, t_1 - t_4, t_3, \min(t_2, t_4) ) \in \mathcal{J}_4 \\ \frac{1}{35} \leqslant t_5 < \min \left(t_4, \frac{1}{2}(t_1 - t_4) \right) \\  (1 - t_1 - t_2 - t_3, t_2, t_3, t_4, t_5) \in \mathcal{I}_5 }} S\left(\mathcal{A}_{\beta_3 p_2 p_3 p_4 p_5}, p_5 \right) \\
\nonumber &- \sum_{\substack{\frac{1}{35} \leqslant t_1 < \frac{17}{35} \\ \frac{1}{35} \leqslant t_2 < \min \left(t_1, \frac{1}{2}(1 - t_1) \right) \\ (t_1, t_2) \in \mathcal{N} \\ \frac{1}{35} \leqslant t_3 < \min \left(t_2, \frac{1}{2}(1 - t_1 - t_2) \right) \\  (t_1, t_2, t_3) \notin \mathcal{I}_3 \\ (t_1, t_2, t_3) \notin \mathcal{J}_3 \\ \frac{1}{35} \leqslant t_4 < \frac{1}{2} t_1 \\  (1 - t_1 - t_2 - t_3, t_2, t_3, t_4) \notin \mathcal{I}_4 \\  (1 - t_1 - t_2 - t_3, t_2, t_3, t_4, t_4) \notin \mathcal{J}_5 \\  (t_1 - t_4, t_2, t_3, t_4, t_3) \notin \mathcal{J}_5 \\ (1 - t_1 - t_2 - t_3, t_2, t_3, t_4) \in \mathcal{J}_4 \text{ and } (1 - t_1 - t_2 - t_3, t_1 - t_4, t_3, \min(t_2, t_4) ) \in \mathcal{J}_4 \\ \frac{1}{35} \leqslant t_5 < \min \left(t_4, \frac{1}{2}(t_1 - t_4) \right) \\  (1 - t_1 - t_2 - t_3, t_2, t_3, t_4, t_5) \notin \mathcal{I}_5 }} S\left(\mathcal{A}_{\beta_3 p_2 p_3 p_4 p_5}, p_5 \right) \\
=&\ S_{444331} - S_{444332} + S_{444333}.
\end{align}
We have asymptotic formulas for $S_{444331}$ and $S_{444332}$ by Lemma~\ref{l31} and Lemma~\ref{l32} respectively, and we want to perform a role--reversal on $S_{444333}$. Note that $S_{444333}$ counts numbers of the form $\beta_3 p_2 p_3 p_4 p_5 \beta_5$, where $\beta_5 \sim v^{t_1 - t_4 - t_5}$ and $\left(\beta_5, P(p_5)\right) = 1$. Now we can either reverse the roles of $\beta_5$ and $p_2$ or reverse the roles of $\beta_5$ and $p_4$. Since we do not know which prime variable is larger, we choose to reverse the roles of $\beta_5$ and $\min(p_2, p_4)$. We can write this process as:
\begin{align}
\nonumber S_{444333} =&\ \sum_{\substack{\frac{1}{35} \leqslant t_1 < \frac{17}{35} \\ \frac{1}{35} \leqslant t_2 < \min \left(t_1, \frac{1}{2}(1 - t_1) \right) \\ (t_1, t_2) \in \mathcal{N} \\ \frac{1}{35} \leqslant t_3 < \min \left(t_2, \frac{1}{2}(1 - t_1 - t_2) \right) \\  (t_1, t_2, t_3) \notin \mathcal{I}_3 \\ (t_1, t_2, t_3) \notin \mathcal{J}_3 \\ \frac{1}{35} \leqslant t_4 < \frac{1}{2} t_1 \\  (1 - t_1 - t_2 - t_3, t_2, t_3, t_4) \notin \mathcal{I}_4 \\  (1 - t_1 - t_2 - t_3, t_2, t_3, t_4, t_4) \notin \mathcal{J}_5 \\  (t_1 - t_4, t_2, t_3, t_4, t_3) \notin \mathcal{J}_5 \\ (1 - t_1 - t_2 - t_3, t_2, t_3, t_4) \in \mathcal{J}_4 \text{ and } (1 - t_1 - t_2 - t_3, t_1 - t_4, t_3, \min(t_2, t_4) ) \in \mathcal{J}_4 \\ \frac{1}{35} \leqslant t_5 < \min \left(t_4, \frac{1}{2}(t_1 - t_4) \right) \\  (1 - t_1 - t_2 - t_3, t_2, t_3, t_4, t_5) \notin \mathcal{I}_5 }} S\left(\mathcal{A}_{\beta_3 p_2 p_3 p_4 p_5}, p_5 \right) \\
=&\ \sum_{\substack{\frac{1}{35} \leqslant t_1 < \frac{17}{35} \\ \frac{1}{35} \leqslant t_2 < \min \left(t_1, \frac{1}{2}(1 - t_1) \right) \\ (t_1, t_2) \in \mathcal{N} \\ \frac{1}{35} \leqslant t_3 < \min \left(t_2, \frac{1}{2}(1 - t_1 - t_2) \right) \\  (t_1, t_2, t_3) \notin \mathcal{I}_3 \\ (t_1, t_2, t_3) \notin \mathcal{J}_3 \\ \frac{1}{35} \leqslant t_4 < \frac{1}{2} t_1 \\  (1 - t_1 - t_2 - t_3, t_2, t_3, t_4) \notin \mathcal{I}_4 \\  (1 - t_1 - t_2 - t_3, t_2, t_3, t_4, t_4) \notin \mathcal{J}_5 \\  (t_1 - t_4, t_2, t_3, t_4, t_3) \notin \mathcal{J}_5 \\ (1 - t_1 - t_2 - t_3, t_2, t_3, t_4) \in \mathcal{J}_4 \\ (1 - t_1 - t_2 - t_3, t_1 - t_4, t_3, \min(t_2, t_4) ) \in \mathcal{J}_4 \\ \frac{1}{35} \leqslant t_5 < \min \left(t_4, \frac{1}{2}(t_1 - t_4) \right) \\  (1 - t_1 - t_2 - t_3, t_2, t_3, t_4, t_5) \notin \mathcal{I}_5 }} S\left(\mathcal{A}_{\beta_3 \beta_5 p_3 p_5 \min(p_2, p_4)}, \left(\frac{2v}{\beta_3 \beta_5 p_3 p_5 \min(p_2, p_4)}\right)^{\frac{1}{2}} \right).
\end{align}
And, by an application of Buchstab's identity, the six-dimensional sum that we need to discard is
\begin{equation}
\sum_{\substack{\frac{1}{35} \leqslant t_1 < \frac{17}{35} \\ \frac{1}{35} \leqslant t_2 < \min \left(t_1, \frac{1}{2}(1 - t_1) \right) \\ (t_1, t_2) \in \mathcal{N} \\ \frac{1}{35} \leqslant t_3 < \min \left(t_2, \frac{1}{2}(1 - t_1 - t_2) \right) \\  (t_1, t_2, t_3) \notin \mathcal{I}_3 \\ (t_1, t_2, t_3) \notin \mathcal{J}_3 \\ \frac{1}{35} \leqslant t_4 < \frac{1}{2} t_1 \\  (1 - t_1 - t_2 - t_3, t_2, t_3, t_4) \notin \mathcal{I}_4 \\  (1 - t_1 - t_2 - t_3, t_2, t_3, t_4, t_4) \notin \mathcal{J}_5 \\  (t_1 - t_4, t_2, t_3, t_4, t_3) \notin \mathcal{J}_5 \\ (1 - t_1 - t_2 - t_3, t_2, t_3, t_4) \in \mathcal{J}_4 \text{ and } (1 - t_1 - t_2 - t_3, t_1 - t_4, t_3, \min(t_2, t_4) ) \in \mathcal{J}_4 \\ \frac{1}{35} \leqslant t_5 < \min \left(t_4, \frac{1}{2}(t_1 - t_4) \right) \\  (1 - t_1 - t_2 - t_3, t_2, t_3, t_4, t_5) \notin \mathcal{I}_5 \\ \frac{1}{35} \leqslant t_4 < \frac{1}{2} \max(t_2, t_4) \\  (1 - t_1 - t_2 - t_3, t_1 - t_4 - t_5, t_3, t_5, t_6, \min(t_2, t_4)) \notin \mathcal{I}_6 }} S\left(\mathcal{A}_{\beta_3 \beta_5 p_3 p_5 p_6 \min(p_2, p_4)}, p_6 \right).
\end{equation}

For the remaining $S_{44434}$, we can use reversed Buchstab's identity to subtract the contribution of the almost-primes counted. Since we have two almost-prime variables $\beta_3$ and $\beta_4$, we can subtract two sums
\begin{equation}
\sum_{\substack{\frac{1}{35} \leqslant t_1 < \frac{17}{35} \\ \frac{1}{35} \leqslant t_2 < \min \left(t_1, \frac{1}{2}(1 - t_1) \right) \\ (t_1, t_2) \in \mathcal{N} \\ \frac{1}{35} \leqslant t_3 < \min \left(t_2, \frac{1}{2}(1 - t_1 - t_2) \right) \\  (t_1, t_2, t_3) \notin \mathcal{I}_3 \\ (t_1, t_2, t_3) \notin \mathcal{J}_3 \\ \frac{1}{35} \leqslant t_4 < \frac{1}{2} t_1 \\  (1 - t_1 - t_2 - t_3, t_2, t_3, t_4) \notin \mathcal{I}_4 \\  (1 - t_1 - t_2 - t_3, t_2, t_3, t_4, t_4) \notin \mathcal{J}_5 \\  (t_1 - t_4, t_2, t_3, t_4, t_3) \notin \mathcal{J}_5 \\ \text{either } (1 - t_1 - t_2 - t_3, t_2, t_3, t_4) \notin \mathcal{J}_4 \text{ or } (1 - t_1 - t_2 - t_3, t_1 - t_4, t_3, \min(t_2, t_4) ) \notin \mathcal{J}_4 \\ t_4 < t_5 < \frac{1}{2} (t_1 - t_4) \\ (1 - t_1 - t_2 - t_3, t_2, t_3, t_4, t_5) \in \mathcal{I}_5 }} S\left(\mathcal{A}_{\beta_3 p_2 p_3 p_4 p_5}, p_5 \right)
\end{equation}
and
\begin{equation}
\sum_{\substack{\frac{1}{35} \leqslant t_1 < \frac{17}{35} \\ \frac{1}{35} \leqslant t_2 < \min \left(t_1, \frac{1}{2}(1 - t_1) \right) \\ (t_1, t_2) \in \mathcal{N} \\ \frac{1}{35} \leqslant t_3 < \min \left(t_2, \frac{1}{2}(1 - t_1 - t_2) \right) \\  (t_1, t_2, t_3) \notin \mathcal{I}_3 \\ (t_1, t_2, t_3) \notin \mathcal{J}_3 \\ \frac{1}{35} \leqslant t_4 < \frac{1}{2} t_1 \\  (1 - t_1 - t_2 - t_3, t_2, t_3, t_4) \notin \mathcal{I}_4 \\  (1 - t_1 - t_2 - t_3, t_2, t_3, t_4, t_4) \notin \mathcal{J}_5 \\  (t_1 - t_4, t_2, t_3, t_4, t_3) \notin \mathcal{J}_5 \\ \text{either } (1 - t_1 - t_2 - t_3, t_2, t_3, t_4) \notin \mathcal{J}_4 \text{ or } (1 - t_1 - t_2 - t_3, t_1 - t_4, t_3, \min(t_2, t_4) ) \notin \mathcal{J}_4 \\ t_3 < t_5 < \frac{1}{2} (1 - t_1 - t_2 - t_3) \\ (t_1 - t_4, t_2, t_3, t_4, t_5) \in \mathcal{I}_5 }} S\left(\mathcal{A}_{\beta_4 p_2 p_3 p_4 p_5}, p_5 \right).
\end{equation}
The sum (35) counts numbers of the form $\beta_3 p_2 p_3 p_4 (p_5 \beta_6)$, where $\beta_6 \sim v^{t_1 - t_4 - t_5}$ and $\left(\beta_6, P(p_5)\right) = 1$. The sum (36) counts numbers of the form $(\beta_7 p_6) p_2 p_3 p_4 \beta_4$, where $\beta_7 \sim v^{1 - t_1 - t_2 - t_3 - t_6}$ and $\left(\beta_7, P(p_6)\right) = 1$. Now, one can easily find that numbers of the form $(\beta_7 p_6) p_2 p_3 p_4 (p_5 \beta_6)$ are counted in both sum, and we must subtract the following sum from the savings:
\begin{equation}
\sum_{\substack{\frac{1}{35} \leqslant t_1 < \frac{17}{35} \\ \frac{1}{35} \leqslant t_2 < \min \left(t_1, \frac{1}{2}(1 - t_1) \right) \\ (t_1, t_2) \in \mathcal{N} \\ \frac{1}{35} \leqslant t_3 < \min \left(t_2, \frac{1}{2}(1 - t_1 - t_2) \right) \\  (t_1, t_2, t_3) \notin \mathcal{I}_3 \\ (t_1, t_2, t_3) \notin \mathcal{J}_3 \\ \frac{1}{35} \leqslant t_4 < \frac{1}{2} t_1 \\  (1 - t_1 - t_2 - t_3, t_2, t_3, t_4) \notin \mathcal{I}_4 \\  (1 - t_1 - t_2 - t_3, t_2, t_3, t_4, t_4) \notin \mathcal{J}_5 \\  (t_1 - t_4, t_2, t_3, t_4, t_3) \notin \mathcal{J}_5 \\ \text{either } (1 - t_1 - t_2 - t_3, t_2, t_3, t_4) \notin \mathcal{J}_4 \text{ or } (1 - t_1 - t_2 - t_3, t_1 - t_4, t_3, \min(t_2, t_4) ) \notin \mathcal{J}_4 \\ t_4 < t_5 < \frac{1}{2} (t_1 - t_4) \\ (1 - t_1 - t_2 - t_3, t_2, t_3, t_4, t_5) \in \mathcal{I}_5 \\ t_3 < t_6 < \frac{1}{2} (1 - t_1 - t_2 - t_3) \\ (t_1 - t_4, t_2, t_3, t_4, t_6) \in \mathcal{I}_5 }} S\left(\mathcal{A}_{\beta_6 p_2 p_3 p_4 p_5 p_6}, p_6 \right).
\end{equation}

Combining all the four cases above, we get a loss from $S_{4443}$ and also a loss from $S_{444}$ of
\begin{align}
\nonumber & \left( \int_{(t_1, t_2, t_3, t_4) \in U_{\mathcal{N}08} } \frac{\omega \left(\frac{t_1 - t_4}{t_4}\right) \omega \left(\frac{1 - t_1 - t_2 - t_3}{t_3}\right)}{t_2 t_3^2 t_4^2} d t_4 d t_3 d t_2 d t_1 \right) \\
\nonumber -& \left( \int_{(t_1, t_2, t_3, t_4, t_5) \in U_{\mathcal{N}09} } \frac{\omega \left(\frac{t_1 - t_4 - t_5}{t_5}\right) \omega \left(\frac{1 - t_1 - t_2 - t_3}{t_3}\right)}{t_2 t_3^2 t_4 t_5^2} d t_5 d t_4 d t_3 d t_2 d t_1 \right) \\
\nonumber -& \left( \int_{(t_1, t_2, t_3, t_4, t_5) \in U_{\mathcal{N}10} } \frac{\omega \left(\frac{t_1 - t_4}{t_4}\right) \omega \left(\frac{1 - t_1 - t_2 - t_3 - t_5}{t_5}\right)}{t_2 t_3 t_4^2 t_5^2} d t_5 d t_4 d t_3 d t_2 d t_1 \right) \\
\nonumber +& \left( \int_{(t_1, t_2, t_3, t_4, t_5, t_6) \in U_{\mathcal{N}11} } \frac{\omega \left(\frac{t_1 - t_4 - t_5}{t_5}\right) \omega \left(\frac{1 - t_1 - t_2 - t_3 - t_6}{t_6}\right)}{t_2 t_3 t_4 t_5^2 t_6^2} d t_6 d t_5 d t_4 d t_3 d t_2 d t_1 \right) \\
\nonumber +& \left( \int_{(t_1, t_2, t_3, t_4, t_5, t_6) \in U_{\mathcal{N}12} } \frac{\omega \left(\frac{t_1 - t_4 - t_5 - t_6}{t_6}\right) \omega \left(\frac{1 - t_1 - t_2 - t_3}{t_3}\right)}{t_2 t_3^2 t_4 t_5 t_6^2} d t_6 d t_5 d t_4 d t_3 d t_2 d t_1 \right) \\
\nonumber +& \left( \int_{(t_1, t_2, t_3, t_4, t_5, t_6) \in U_{\mathcal{N}13} } \frac{\omega \left(\frac{t_1 - t_4}{t_4}\right) \omega \left(\frac{1 - t_1 - t_2 - t_3 - t_5 - t_6}{t_6}\right)}{t_2 t_3 t_4^2 t_5 t_6^2} d t_6 d t_5 d t_4 d t_3 d t_2 d t_1 \right) \\
\nonumber +& \left( \int_{(t_1, t_2, t_3, t_4, t_5, t_6) \in U_{\mathcal{N}14} } \frac{\omega \left(\frac{\max(t_2, t_4) - t_6}{t_6}\right) \omega \left(\frac{t_1 - t_4 - t_5}{t_5}\right) \omega \left(\frac{1 - t_1 - t_2 - t_3}{t_3}\right)}{\min(t_2, t_4) t_3^2 t_5^2 t_6^2} d t_6 d t_5 d t_4 d t_3 d t_2 d t_1 \right) \\
\nonumber \leqslant& \left( \int_{(t_1, t_2, t_3, t_4) \in U_{\mathcal{N}08} } \frac{\omega_1 \left(\frac{t_1 - t_4}{t_4}\right) \omega \left(\frac{1 - t_1 - t_2 - t_3}{t_3}\right)}{t_2 t_3^2 t_4^2} d t_4 d t_3 d t_2 d t_1 \right) \\
\nonumber -& \left( \int_{(t_1, t_2, t_3, t_4, t_5) \in U_{\mathcal{N}09} } \frac{\omega_0 \left(\frac{t_1 - t_4 - t_5}{t_5}\right) \omega_0 \left(\frac{1 - t_1 - t_2 - t_3}{t_3}\right)}{t_2 t_3^2 t_4 t_5^2} d t_5 d t_4 d t_3 d t_2 d t_1 \right) \\
\nonumber -& \left( \int_{(t_1, t_2, t_3, t_4, t_5) \in U_{\mathcal{N}10} } \frac{\omega_0 \left(\frac{t_1 - t_4}{t_4}\right) \omega_0 \left(\frac{1 - t_1 - t_2 - t_3 - t_5}{t_5}\right)}{t_2 t_3 t_4^2 t_5^2} d t_5 d t_4 d t_3 d t_2 d t_1 \right) \\
\nonumber +& \left( \int_{(t_1, t_2, t_3, t_4, t_5, t_6) \in U_{\mathcal{N}11} } \frac{\omega_1 \left(\frac{t_1 - t_4 - t_5}{t_5}\right) \omega_1 \left(\frac{1 - t_1 - t_2 - t_3 - t_6}{t_6}\right)}{t_2 t_3 t_4 t_5^2 t_6^2} d t_6 d t_5 d t_4 d t_3 d t_2 d t_1 \right) \\
\nonumber +& \left( \int_{(t_1, t_2, t_3, t_4, t_5, t_6) \in U_{\mathcal{N}12} } \frac{\omega_1 \left(\frac{t_1 - t_4 - t_5 - t_6}{t_6}\right) \omega_1 \left(\frac{1 - t_1 - t_2 - t_3}{t_3}\right)}{t_2 t_3^2 t_4 t_5 t_6^2} d t_6 d t_5 d t_4 d t_3 d t_2 d t_1 \right) \\
\nonumber +& \left( \int_{(t_1, t_2, t_3, t_4, t_5, t_6) \in U_{\mathcal{N}13} } \frac{\omega_1 \left(\frac{t_1 - t_4}{t_4}\right) \omega_1 \left(\frac{1 - t_1 - t_2 - t_3 - t_5 - t_6}{t_6}\right)}{t_2 t_3 t_4^2 t_5 t_6^2} d t_6 d t_5 d t_4 d t_3 d t_2 d t_1 \right) \\
\nonumber +& \left( \int_{(t_1, t_2, t_3, t_4, t_5, t_6) \in U_{\mathcal{N}14} } \frac{\omega_1 \left(\frac{\max(t_2, t_4) - t_6}{t_6}\right) \omega_1 \left(\frac{t_1 - t_4 - t_5}{t_5}\right) \omega_1 \left(\frac{1 - t_1 - t_2 - t_3}{t_3}\right)}{\min(t_2, t_4) t_3^2 t_5^2 t_6^2} d t_6 d t_5 d t_4 d t_3 d t_2 d t_1 \right) \\
\nonumber \leqslant&\ (0.057182 - 0.004301 - 0.000022 + 0.000001 + 0.002635 + 0.000001 + 0.000001) \\
=&\ 0.055497,
\end{align}
where
\begin{align}
\nonumber U_{\mathcal{N}08}(t_1, t_2, t_3, t_4) :=&\ \left\{ (t_1, t_2) \in \mathcal{N},\ \frac{1}{35} \leqslant t_3 < \min\left(t_2, \frac{1}{2}(1 - t_1 - t_2)\right), \right. \\
\nonumber & \quad (t_1, t_2, t_3) \notin \mathcal{I}_3,\ (t_1, t_2, t_3) \notin \mathcal{J}_3, \\
\nonumber & \quad \frac{1}{35} \leqslant t_4 < \frac{1}{2} t_1,\ (1 - t_1 - t_2 - t_3, t_2, t_3, t_4) \notin \mathcal{I}_4, \\
\nonumber & \quad (1 - t_1 - t_2 - t_3, t_2, t_3, t_4, t_4) \notin \mathcal{J}_5,\ (t_1 - t_4, t_2, t_3, t_4, t_3) \notin \mathcal{J}_5, \\
\nonumber & \quad \text{either } (1 - t_1 - t_2 - t_3, t_2, t_3, t_4) \notin \mathcal{J}_4 \\
\nonumber & \qquad \text{or } (1 - t_1 - t_2 - t_3, t_1 - t_4, t_3, \min(t_2, t_4)) \notin \mathcal{J}_4, \\
\nonumber & \left. \quad \frac{1}{35} \leqslant t_1 < \frac{17}{35},\ \frac{1}{35} \leqslant t_2 < \min\left(t_1, \frac{1}{2}(1-t_1) \right) \right\}, \\
\nonumber U_{\mathcal{N}09}(t_1, t_2, t_3, t_4, t_5) :=&\ \left\{ (t_1, t_2) \in \mathcal{N},\ \frac{1}{35} \leqslant t_3 < \min\left(t_2, \frac{1}{2}(1 - t_1 - t_2)\right), \right. \\
\nonumber & \quad (t_1, t_2, t_3) \notin \mathcal{I}_3,\ (t_1, t_2, t_3) \notin \mathcal{J}_3, \\
\nonumber & \quad \frac{1}{35} \leqslant t_4 < \frac{1}{2} t_1,\ (1 - t_1 - t_2 - t_3, t_2, t_3, t_4) \notin \mathcal{I}_4, \\
\nonumber & \quad (1 - t_1 - t_2 - t_3, t_2, t_3, t_4, t_4) \notin \mathcal{J}_5,\ (t_1 - t_4, t_2, t_3, t_4, t_3) \notin \mathcal{J}_5, \\
\nonumber & \quad \text{either } (1 - t_1 - t_2 - t_3, t_2, t_3, t_4) \notin \mathcal{J}_4 \\
\nonumber & \qquad \text{or } (1 - t_1 - t_2 - t_3, t_1 - t_4, t_3, \min(t_2, t_4)) \notin \mathcal{J}_4, \\
\nonumber & \quad t_4 < t_5 < \frac{1}{2}(t_1 - t_4),\ (1 - t_1 - t_2 - t_3, t_2, t_3, t_4, t_5) \in \mathcal{I}_5, \\
\nonumber & \left. \quad \frac{1}{35} \leqslant t_1 < \frac{17}{35},\ \frac{1}{35} \leqslant t_2 < \min\left(t_1, \frac{1}{2}(1-t_1) \right) \right\}, \\
\nonumber U_{\mathcal{N}10}(t_1, t_2, t_3, t_4, t_5) :=&\ \left\{ (t_1, t_2) \in \mathcal{N},\ \frac{1}{35} \leqslant t_3 < \min\left(t_2, \frac{1}{2}(1 - t_1 - t_2)\right), \right. \\
\nonumber & \quad (t_1, t_2, t_3) \notin \mathcal{I}_3,\ (t_1, t_2, t_3) \notin \mathcal{J}_3, \\
\nonumber & \quad \frac{1}{35} \leqslant t_4 < \frac{1}{2} t_1,\ (1 - t_1 - t_2 - t_3, t_2, t_3, t_4) \notin \mathcal{I}_4, \\
\nonumber & \quad (1 - t_1 - t_2 - t_3, t_2, t_3, t_4, t_4) \notin \mathcal{J}_5,\ (t_1 - t_4, t_2, t_3, t_4, t_3) \notin \mathcal{J}_5, \\
\nonumber & \quad \text{either } (1 - t_1 - t_2 - t_3, t_2, t_3, t_4) \notin \mathcal{J}_4 \\
\nonumber & \qquad \text{or } (1 - t_1 - t_2 - t_3, t_1 - t_4, t_3, \min(t_2, t_4)) \notin \mathcal{J}_4, \\
\nonumber & \quad t_3 < t_5 < \frac{1}{2}(1 - t_1 - t_2 - t_3),\ (t_1 - t_4, t_2, t_3, t_4, t_5) \in \mathcal{I}_5, \\
\nonumber & \left. \quad \frac{1}{35} \leqslant t_1 < \frac{17}{35},\ \frac{1}{35} \leqslant t_2 < \min\left(t_1, \frac{1}{2}(1-t_1) \right) \right\}, \\
\nonumber U_{\mathcal{N}11}(t_1, t_2, t_3, t_4, t_5, t_6) :=&\ \left\{ (t_1, t_2) \in \mathcal{N},\ \frac{1}{35} \leqslant t_3 < \min\left(t_2, \frac{1}{2}(1 - t_1 - t_2)\right), \right. \\
\nonumber & \quad (t_1, t_2, t_3) \notin \mathcal{I}_3,\ (t_1, t_2, t_3) \notin \mathcal{J}_3, \\
\nonumber & \quad \frac{1}{35} \leqslant t_4 < \frac{1}{2} t_1,\ (1 - t_1 - t_2 - t_3, t_2, t_3, t_4) \notin \mathcal{I}_4, \\
\nonumber & \quad (1 - t_1 - t_2 - t_3, t_2, t_3, t_4, t_4) \notin \mathcal{J}_5,\ (t_1 - t_4, t_2, t_3, t_4, t_3) \notin \mathcal{J}_5, \\
\nonumber & \quad \text{either } (1 - t_1 - t_2 - t_3, t_2, t_3, t_4) \notin \mathcal{J}_4 \\
\nonumber & \qquad \text{or } (1 - t_1 - t_2 - t_3, t_1 - t_4, t_3, \min(t_2, t_4)) \notin \mathcal{J}_4, \\
\nonumber & \quad t_4 < t_5 < \frac{1}{2}(t_1 - t_4),\ (1 - t_1 - t_2 - t_3, t_2, t_3, t_4, t_5) \in \mathcal{I}_5, \\
\nonumber & \quad t_3 < t_6 < \frac{1}{2}(1 - t_1 - t_2 - t_3),\ (t_1 - t_4, t_2, t_3, t_4, t_6) \in \mathcal{I}_5, \\
\nonumber & \left. \quad \frac{1}{35} \leqslant t_1 < \frac{17}{35},\ \frac{1}{35} \leqslant t_2 < \min\left(t_1, \frac{1}{2}(1-t_1) \right) \right\}, \\
\nonumber U_{\mathcal{N}12}(t_1, t_2, t_3, t_4, t_5, t_6) :=&\ \left\{ (t_1, t_2) \in \mathcal{N},\ \frac{1}{35} \leqslant t_3 < \min\left(t_2, \frac{1}{2}(1 - t_1 - t_2)\right), \right. \\
\nonumber & \quad (t_1, t_2, t_3) \notin \mathcal{I}_3,\ (t_1, t_2, t_3) \notin \mathcal{J}_3, \\
\nonumber & \quad \frac{1}{35} \leqslant t_4 < \frac{1}{2} t_1,\ (1 - t_1 - t_2 - t_3, t_2, t_3, t_4) \notin \mathcal{I}_4, \\
\nonumber & \quad (1 - t_1 - t_2 - t_3, t_2, t_3, t_4, t_4) \in \mathcal{J}_5, \\
\nonumber & \quad \frac{1}{35} \leqslant t_5 < \min \left(t_4, \frac{1}{2}(t_1 - t_4) \right), \\
\nonumber & \quad (1 - t_1 - t_2 - t_3, t_2, t_3, t_4, t_5) \notin \mathcal{I}_5, \\
\nonumber & \quad \frac{1}{35} \leqslant t_6 < \min \left(t_5, \frac{1}{2}(t_1 - t_4 - t_5) \right),\\
\nonumber & \quad (1 - t_1 - t_2 - t_3, t_2, t_3, t_4, t_5, t_6) \notin \mathcal{I}_6, \\
\nonumber & \left. \quad \frac{1}{35} \leqslant t_1 < \frac{17}{35},\ \frac{1}{35} \leqslant t_2 < \min\left(t_1, \frac{1}{2}(1-t_1) \right) \right\}, \\
\nonumber U_{\mathcal{N}13}(t_1, t_2, t_3, t_4, t_5, t_6) :=&\ \left\{ (t_1, t_2) \in \mathcal{N},\ \frac{1}{35} \leqslant t_3 < \min\left(t_2, \frac{1}{2}(1 - t_1 - t_2)\right), \right. \\
\nonumber & \quad (t_1, t_2, t_3) \notin \mathcal{I}_3,\ (t_1, t_2, t_3) \notin \mathcal{J}_3, \\
\nonumber & \quad \frac{1}{35} \leqslant t_4 < \frac{1}{2} t_1,\ (1 - t_1 - t_2 - t_3, t_2, t_3, t_4) \notin \mathcal{I}_4, \\
\nonumber & \quad (1 - t_1 - t_2 - t_3, t_2, t_3, t_4, t_4) \notin \mathcal{J}_5,\ (t_1 - t_4, t_2, t_3, t_4, t_3) \in \mathcal{J}_5, \\
\nonumber & \quad \frac{1}{35} \leqslant t_5 < \min \left(t_3, \frac{1}{2}(1 - t_1 - t_2 - t_3) \right), \\
\nonumber & \quad (t_1 - t_4, t_2, t_3, t_4, t_5) \notin \mathcal{I}_5, \\
\nonumber & \quad \frac{1}{35} \leqslant t_6 < \min \left(t_5, \frac{1}{2}(1 - t_1 - t_2 - t_3 - t_5) \right),\\
\nonumber & \quad (t_1 - t_4, t_2, t_3, t_4, t_5, t_6) \notin \mathcal{I}_6, \\
\nonumber & \left. \quad \frac{1}{35} \leqslant t_1 < \frac{17}{35},\ \frac{1}{35} \leqslant t_2 < \min\left(t_1, \frac{1}{2}(1-t_1) \right) \right\}, \\
\nonumber U_{\mathcal{N}14}(t_1, t_2, t_3, t_4, t_5, t_6) :=&\ \left\{ (t_1, t_2) \in \mathcal{N},\ \frac{1}{35} \leqslant t_3 < \min\left(t_2, \frac{1}{2}(1 - t_1 - t_2)\right), \right. \\
\nonumber & \quad (t_1, t_2, t_3) \notin \mathcal{I}_3,\ (t_1, t_2, t_3) \notin \mathcal{J}_3, \\
\nonumber & \quad \frac{1}{35} \leqslant t_4 < \frac{1}{2} t_1,\ (1 - t_1 - t_2 - t_3, t_2, t_3, t_4) \notin \mathcal{I}_4, \\
\nonumber & \quad (1 - t_1 - t_2 - t_3, t_2, t_3, t_4, t_4) \notin \mathcal{J}_5,\ (t_1 - t_4, t_2, t_3, t_4, t_3) \notin \mathcal{J}_5, \\
\nonumber & \quad (1 - t_1 - t_2 - t_3, t_2, t_3, t_4) \in \mathcal{J}_4 \\
\nonumber & \qquad \text{and } (1 - t_1 - t_2 - t_3, t_1 - t_4, t_3, \min(t_2, t_4)) \in \mathcal{J}_4, \\
\nonumber & \quad \frac{1}{35} \leqslant t_5 < \min \left(t_4, \frac{1}{2}(t_1 - t_4) \right), \\
\nonumber & \quad (1 - t_1 - t_2 - t_3, t_2, t_3, t_4, t_5) \notin \mathcal{I}_5, \\
\nonumber & \quad \frac{1}{35} \leqslant t_6 < \frac{1}{2} \max(t_2, t_4),\\
\nonumber & \quad (1 - t_1 - t_2 - t_3, t_1 - t_4 - t_5, t_3, t_5, t_6, \min(t_2, t_4)) \notin \mathcal{I}_6, \\
\nonumber & \left. \quad \frac{1}{35} \leqslant t_1 < \frac{17}{35},\ \frac{1}{35} \leqslant t_2 < \min\left(t_1, \frac{1}{2}(1-t_1) \right) \right\}.
\end{align}
Note that the seven integrals above correspond to $S_{44434}$ and sums in (35), (36), (37), (29), (31) and (34) respectively.

Finally, by (7), (23), (26) and (38), the total loss is less than
$$
0.7226 + 0.176459 + 0.040113 + 0.055497 < 0.995,
$$
and we have
$$
S\left(\mathcal{A},(2 v)^{\frac{1}{2}}\right) > 0.005 \left(\sum_{P < p \leqslant 2P}\frac{1}{p}\right)^{K} \frac{x^{\frac{1}{2}+\varepsilon}}{\log v}.
$$
Now the proof of Theorem~\ref{t1} is completed.

\bibliographystyle{plain}
\bibliography{bib}

@book{HarmanBOOK,
  title={Prime-detecting Sieves},
  author={Harman, G.},
  year={2007},
  volume={33},
  publisher={Princeton University Press, Princeton, NJ},
  series={London Mathematical Society Monographs
 (New Series)}
}

@article{JiaLiu,
  title={On the largest prime factor of integers},
  author={Jia, C. and Liu, M.-C.},
  journal={Acta Arith.},
  year={2000},
  pages={17--48},
  volume = {95},
  number = {1},
}

@article{BalogHarmanPintz2,
  title={Numbers with a large prime factor {IV}},
  author={Balog, A. and Harman, G. and Pintz, J.},
  journal={J. London Math. Soc.},
  year={1983},
  pages={218--226},
  volume = {28},
  number = {2},
}

@article{HeathBrown1112,
  title={The largest prime factor of the integers in an interval},
  author={Heath-Brown, D. R.},
  journal={Sci. China Ser. A},
  year={1996},
  pages={449--476},
  volume = {39},
  number = {5},
}

@article{HeathBrownJia1718,
  title={The largest prime factor of the integers in an interval, {II}},
  author={Heath-Brown, D. R. and Jia, C.},
  journal={J. Reine Angew. Math.},
  year={1998},
  pages={35--59},
  volume = {498},
}

@article {Haugland,
    AUTHOR = {Haugland, J. K.},
     TITLE = {Application of sieve methods to prime numbers},
   JOURNAL = {Ph.D. Thesis},
  FJOURNAL = {},
    VOLUME = {University of Oxford},
      YEAR = {1998},
    NUMBER = {},
     PAGES = {},
      ISSN = {},
   MRCLASS = {},
  MRNUMBER = {},
MRREVIEWER = {},
       DOI = {},
       URL = {},
}

@article{Jutila23,
  title={On numbers with a large prime factor},
  author={Jutila, M.},
  journal={J. Indian Math. Soc. (N. S.)},
  year={1973},
  pages={43--53},
  volume = {37},
}

@article{DI84,
  title={Power mean-values for {D}irichlet's polynomials and the {R}iemann zeta-function, {II}},
  author={Deshouillers, J.-M. and Iwaniec, H.},
  journal={Acta Arith.},
  year={1984},
  pages={305--312},
  volume = {43},
  number = {3},
}

@article{Balog73,
  title={Numbers with a large prime factor},
  author={Balog, A.},
  journal={Studia Sci. Math. Hungar.},
  year={1980},
  pages={139--146},
  volume = {15},
}

@incollection {Balog772,
    AUTHOR = {Balog, A.},
     TITLE = {Numbers with a large prime factor {II}},
 BOOKTITLE = {Topics in Classical Number Theory},
     PAGES = {49--67},
 PUBLISHER = {Math. Soc. János Bolyai 34, Elsevier, North Holland, Amsterdam},
      YEAR = {1984},
}

@ARTICLE{LRB052,
       author = {R. Li},
        title = "{The number of primes in short intervals and numerical calculations for Harman's sieve}",
      journal = {arXiv e-prints},
         year = 2025,
          eid = {arXiv:2308.04458},
        pages = {arXiv:2308.04458v8},
          doi = {10.48550/arXiv.2308.04458},
archivePrefix = {arXiv},
       eprint = {2308.04458},
 primaryClass = {math.NT},
       adsurl = {https://ui.adsabs.harvard.edu/abs/2023arXiv230804458L}
}

@ARTICLE{LRB4923,
       author = {Li, R.},
        title = "{Hybrid estimation of single exponential sums, exceptional characters and primes in short intervals}",
      journal = {arXiv e-prints},
         year = 2024,
          eid = {arXiv:2401.11139},
        pages = {arXiv:2401.11139v3},
          doi = {10.48550/arXiv.2401.11139},
archivePrefix = {arXiv},
       eprint = {2401.11139},
 primaryClass = {math.NT},
       adsurl = {https://ui.adsabs.harvard.edu/abs/2024arXiv240111139L}
}

@ARTICLE{LRB5153,
       author = {Li, R.},
        title = "{On the largest prime factor of integers in short intervals}",
      journal = {Preprints},
         year = 2025,
        pages = {preprints202504.1212.v2},
          doi = {10.20944/preprints202504.1212.v2},
archivePrefix = {Preprints},
       eprint = {202504.1212.v1},
}

@article{WattTheorem,
  title={Kloosterman sums and a mean value for {D}irichlet polynomials},
  author={Watt, N.},
  journal={J. Number Theory},
  year={1995},
  pages={179--210},
  volume = {53},
}

@article{LRB122,
  title={Primes in almost all short intervals {II}},
  author={Li, R.},
  journal={Cambridge Open Engage},
  year={2025},
  pages={10.33774/coe-2025-jrnjl},
}

@article{LRB7437,
  title={On the largest prime factor of integers in short intervals {II}},
  author={Li, R.},
  journal={Cambridge Open Engage},
  year={2025},
  pages={10.33774/coe-2025-xnbjq-v2},
}

@article{JiaPSV,
  title={On the {P}iatetski--{S}hapiro--{V}inogradov {T}heorem},
  author={Jia, C.},
  journal={Acta Arith.},
  year={1995},
  pages={1--28},
  volume = {73},
  number = {1},
}

\end{document}